\documentclass[reqno]{amsart}
\usepackage{amsmath}
\usepackage{amssymb}
\usepackage{graphicx}
\usepackage{graphics}
\usepackage{mathrsfs}
\usepackage{cite}
\usepackage{tikz-cd}
\usepackage{hyperref}
\DeclareMathOperator\R{\mathbb R}
\DeclareMathOperator\C{\mathbb C}
\DeclareMathOperator\Z{\mathbb Z}
\DeclareMathOperator\id{\operatorname{id}}
\DeclareMathOperator\End{\operatorname{End}}
\DeclareMathOperator\Hom{\operatorname{Hom}}
\DeclareMathOperator\tr{\operatorname{tr}}
\newcommand{\arrowIn}{
	\tikz \draw[-stealth] (-1pt,0) -- (1pt,0);
}
\tikzcdset{scale cd/.style={every label/.append style={scale=#1},
		cells={nodes={scale=#1}}}}

\newtheorem{theorem}{Theorem}[section]
\newtheorem{lemma}[theorem]{Lemma}
\newtheorem{cor}[theorem]{Corollary}
\newtheorem{prop}[theorem]{Proposition}
\theoremstyle{definition}
\newtheorem{definition}[theorem]{Definition}
\newtheorem{example}[theorem]{Example}

\theoremstyle{remark}
\newtheorem{remark}[theorem]{Remark}

\newcommand{\dontprint}[1]\relax

\hypersetup{linktocpage}
\begin{document}
\title[Groupoid Intertwiner and Twist:part I]{Groupoid intertwiner and twist for dynamical Yang--Baxter equation: part I}
\author{Muze Ren}
\address{Section de mathématiques, University of Geneva,
	rue du Conseil-Général 7-9
	1205 Geneva, Switzerland.}
\email{muze.ren@unige.ch}
\maketitle	
\begin{abstract}
Intertwiner is a homomorphism between two existing dynamical $\check{R}$ matrices, first introduced by Baxter in eight vertex-SOS correspondence, we develop certain equivalence relations among $\check{R}$ matrices using intertwiners. 

Twist is a homomorphism that twist a dynamical $\check{R}$ matrix to get a new dynamical $\check{R}$ matrix, we introduce a kind of notion of twist that generalize classical Drinfeld twist in quasi-triangular Hopf algebra and some dynamical twist. As applications, we obtain some examples of twists from Ocneanu cell calculus and Fendley--Ginsparg orbifold constructions. The relations between intertwiner and twist are also discussed, the groupoid structures are emphasized.
\end{abstract}

\tableofcontents
\section{Introduction}
\subsection{Groupoid graded vector space}
\hfill\\
 The dynamical Yang--Baxter equation \eqref{eq:int_dYBE} was initially considered by Gervais and Neveu in the study of Liouville theory \cite{Gervais1984}. And later Felder in \cite{Felder1994,FelderICMP1995} rediscovered this equation and the classical dynamical Yang--Baxter equation and found the theory of elliptic quantum group and elliptic R matrix corresponding to the eight vertex SOS model \cite{Andrews1984}. Later the equations were widely studied in different aspects and have many connections to other fields, see for examples \cite{FelderVarchenko1996, FelderTarasovVarchenkoAMS1997,FelderTarasovVarchenko1999,EtingofVarchenko1998,EtingofVarchenkoCMP1998,EtingofVarchenko1999,Aganagic2021,CostelloWittenYamazakiI2018,CostelloWittenYamazakiII2018,Stokman2022}, also the standard books  \cite{EtingofSchiffmann1999,Etingof2005,Konno2020} and reference there.

\begin{equation}\label{eq:int_dYBE}
	\begin{split}
		&\check{R}^{(23)}(z-w,ah^{(1)})\check{R}^{(12)}(z,a)\check{R}^{(23)}(w,ah^{(1)})\\
		&=\check{R}^{(12)}(w,a)\check{R}^{(23)}(z,ah^{(1)})\check{R}^{(12)}(z-w,a)
	\end{split}
\end{equation}

There are many different types of dynamical $\check{R}(z,\lambda)$ matrices. With respect to the spectral parameters, we can roughly divided into the following type
\begin{itemize}
	\item Rational
	\item Trigonometric
	\item Elliptic
\end{itemize}
And with respect to the dynamical parameters $\lambda$, we can roughly divided them into the following
\begin{itemize}
	\item If $\lambda$ is trivial, these are the classical quantum group, see for example the book \cite{LusztigBook1993,ChariPressleyBook1994,KasselBook1995}.
	\item If $\lambda$ is unrestricted, this is the Elliptic quantum group and its various trigonometric and rational degeneration, for examples \cite{Felder1994,FelderICMP1995}.
	\item If $\lambda$ is restricted to a finite set, this is the restricted quantum group, for examples in \cite{Felder2020}.
\end{itemize}

These different types of solutions of dynamical Yang--Baxter equation can be seen as the representations of the usual Yang--Baxter equation \eqref{eq:int_YBE} on different types of vector spaces 
\begin{equation}\label{eq:int_YBE}
		\check{R}^{(23)}(z-w)\check{R}^{(12)}(z)\check{R}^{(23)}(w)=\check{R}^{(12)}(w)\check{R}^{(23)}(z)\check{R}^{(12)}(z-w)
\end{equation}

\begin{enumerate}
	\item If we use the usual vector spaces, there are certain elliptic solutions can not appear.\\
	\item If we use the $\Z$ graded or $\Z^n$ graded (for example some weight lattices) vector spaces, then more elliptic solutions appears, but not restricted ones.\\
	\item If we use the groupoid graded vector spaces which means more diverse and subtle gradings, then the restricted type solutions also appears.
\end{enumerate}

From the above perspectives, groupoid vector spaces language developed in \cite{Felder2020} is a tool to study the general types of $\check{R}(z,\lambda)$ matrices. 

See also the papers of Ocneanu \cite{Ocneanu1988,Ocneanu1991}, Goodman--Harpe--Jones \cite{Goodman-Harpe-Jones}, Pasquier \cite{Pasquier1987b,Pasquier1988a,Pasquier1988}, Pasquier--Saleur \cite{Pasquier1990a}, Roche \cite{Roche1990}, Francesco--Zuber \cite{DiFrancesco1990} and Kawahigashi \cite{Kawahigashi} from the perspective of path model of Temperley--Lieb algebra or the string algebras side to study the integrable system in statistic mechanics, which is more "global" in the sense of discussion of trace and boundary conditions and also many of them are without spectral parameters. And dynamical Yang-Baxter equation perspective is more local. The work \cite{Ren2023} is also try to understand the local nature of above works.

\subsection{Intertwiner}
\hfill\\
Intertwiner is a homomorphism which relates two different dynamical $\check{R}$ matrices. The concept of intertwiner that we studied here was first introduced by Baxter in \cite{Baxter1973}, where he shows the equivalence between "zero field" 8 vertex model and a SOS model, the equivalence here means that with carefully chosen parameters and boundary conditions the partition function of these two models are the same. 

Using the groupoid graded vector space language, it is equivalent to say that there exists an intertwiner $C$ such that we have the following relation \eqref{eq:intro_RCC} and we take $\check{R}_1$ to be the dynamical $\check{R}$ matrix $\check{R}^{\text{8v}}$ of zero field 8 vertex model, take $\check{R}_2$ to be the dynamical $\check{R}$ matrix correspinding to the SOS model $\check{R}^{\text{sos}}$.
\begin{equation}\label{eq:intro_RCC}
	C(w)^{(12)}C(z)^{(23)}\check{R}^{\text{8v}}_1(z-w)^{(12)}=\check{R}^{\text{sos}}_2(z-w)^{(23)}C^{(12)}(z)C^{(23)}(w)
\end{equation}

For physics considerations, the transfer matrix is more interesting, taking the partial trace, we have the relations between "transfer matrix", $*_{\pi_1},*_{\pi_2}$ are certain convolution product.	\begin{equation}\label{eq:intr_s_RC_CR}
	\overline{\operatorname{tr}}_{V^{\pi_1}}\check{R}_1(z)*_{\pi_1}\overline{\operatorname{tr}}_{V^{\pi}}C(z')=\overline{\operatorname{tr}}_{V^{\pi}}C(z')*_{\pi_2}\overline{\operatorname{tr}}_{V^{\pi_2}}\check{R}_2(z).
\end{equation}
In the situations that $\check{R}_1,\check{R}_2$ are indeed usual matrices, from this equation, we know that some "eigenvalues" and "eigenvectors" of $\check{R_1}$ and $\check{R_2}$ are the same, which will implies certain equivalence of the partition function. 

{\bf{Relation with quantum groups}.}
\hfill\\
The basic $RLL$ algebraic formulae of the quantum inverse problem method developed \cite{Faddeev1998,Faddeev1988} by Faddeev, Reshetikhin, Takhtadzhyan can be got by taking the $\check{R}_1=\check{R}_2$ in \eqref{eq:intr_s_RC_CR} and we use $L$ for $C$, we will have
\begin{equation}\label{eq:RLL}
	L(w)^{(12)}L(z)^{(23)}\check{R}(z-w)^{(12)}(z)=\check{R}(z-w)^{(23)}L^{(12)}(z)L^{(23)}(w)
\end{equation}

{\bf{Relation with fusion operations}}.
\hfill\\
Suppose that we have an ($\check{R}_1$,$\check{R}_2$) intertwiner $C$ and ($\check{R}_2,\check{R}_3$) intertwiner $\widehat{C}$, their composition $\widehat{C}^{(23)}C^{(12)}$ is an ($\check{R}_1,\check{R}_3$) intertwiner which means that
\begin{equation}
	\widehat{C}^{(23)}C^{(12)}\widehat{C}^{(34)}C^{(23)}\check{R}_1(z)^{(12)}=\check{R}_3(z)^{(34)}\widehat{C}^{(23)}C^{(12)}\widehat{C}^{(34)}C^{(23)}
\end{equation}

Various fusion construction in statistic mechanics, for example \cite{Date1987} and also the coproduct operations constructed in \eqref{eq:RLL} which leads to the Hopf algebra formalism can be seen as certain applications of this fact. 

{\bf{Equivalence relation}}
\hfill\\
In the above, we discuss the composition property of intertwiners, we can also consider the tranpose relation to make the existence of intertwiners an equivalence relation. We define the transpose of the $\check{R}$ matrix and the transpose of an intertwiner, which means taking the inverse directions from $\check{R}_2$ to $\check{R}_1$. And with the assumption that $\check{R}_1=\check{R}^T_1$ and $\check{R}^T_2=\check{R}_2$, this is an equivalence relation. If we further assume that the intertwiners are invertible, we still have an equivalence relation, we call it the transfer matrices equivalence.

From the perspective of conformal field theory and integrable lattice models, the intertwiner are also widely studied, see for examples the work of Ocneanu \cite{Ocneanu1988,Ocneanu1991}, Pasquier \cite{Pasquier1987b,Pasquier1988a,Pasquier1988}, Pasquier--Saleur \cite{Pasquier1990a},  Roche\cite{Roche1990}, Francesco--Zuber \cite{DiFrancesco1990} and Pearce--Zhou \cite{Pearce1993}.

In the fundamental work of Baxter \cite{Baxter1973}, actually the intertwiners and the $\check{R}^{\text{sos}}$ are all constructed. Implicitly it suggested that from an existing $\check{R}$ matrices and intertwiners, based on some properties of intertwiner, we could build a new dynamical $\check{R}$ matrix. And one of the motivation to study the twist is to understand what conditions we need have, in order to have a new $\check{R}$ matrix that satisfies the dynamical Yang--Baxter equation.

\subsection{Twist}
\hfill\\
Twist is a homomorphism which transform one dynamical $\check{R}$ matrix to get a new dynamical $\check{R}$ matrix.

{\bf{Twist non-dynamical $\check{R}$ matrices to non-dynamical $\check{R}$ matrices}}
\hfill\\
The first step of the theory of twist is to twist a non-dynamical $\check{R}$ matrix to a non-dynamical $\check{R}$ matrix. In the theory of quasi-triangular quasi-Hopf algebra \cite{Drinfeld1991}, Drinfeld introduced this kind notion of twist.
See also the notion of fiber functors in category theory, see chapter 5 of the book \cite{Etingof2015} and the historical remarks and reference there. 

For a quasitriangular quasi-Hopf algebra $(A,\Delta,\epsilon,\Phi,R)$, let $J$ be an invertible element of $A\otimes A$ such that $(\id\otimes \epsilon)(J)=1=(\epsilon\otimes \id)(J)$, the equation (1.10), (1.11) and (1.12) in \cite{Drinfeld1991} are
\begin{subequations}
\begin{equation}
	\widetilde{\Delta}(a)=J\Delta(a)J^{-1},
\end{equation}
\begin{equation}
\widetilde{\Phi}=J^{(23)}(\id\otimes \Delta)(J)\Phi(\Delta\otimes \id)(J^{-1})(J^{12})^{-1},
\end{equation}
\begin{equation}
\widetilde{R}=J^{21}RJ^{-1},
\end{equation}
\end{subequations}
we obtain a new quasitriangular quasi-Hopf algebra $(A,\widetilde{\Delta},\epsilon,\widetilde{\Phi},\widetilde{R})$. If we assume that $\Phi$ always be 1, we goes to the quasi-triangular Hopf algebra, then in order to let $\widetilde{\Phi}$ be 1, we got the following twist equation of $J$, (2 cocycle equation)
\begin{equation}\label{eq:int_twist_1}
	(\Delta\otimes \operatorname{id})(J)(J\otimes 1)=(\operatorname{id}\otimes \Delta)(J)(1\otimes J)
\end{equation}

{\bf{Twist non-dynamical $\check{R}$ matrices to dynamical $\check{R}$ matrices}}
\hfill\\
In the work of \cite{Babelon1991,Babelon1996}, Babelon, Bernard and Billy generalize the notion of twist that can twist in $\mathfrak{sl}_2$ case, a non-dynamical $\check{R}$ matrix to a dynamical $\check{R}$ matrix, in this case, they got the dynamical twist equation, the shifted cocycle equation (4) in \cite{Babelon1996}.
\begin{equation}\label{eq:intro_twist_2}
[(\Delta\otimes\id)J]\cdot [J(xq^3)\otimes \id]=\big((\id\otimes \Delta)J)\cdot [\id\otimes J]
\end{equation}
and the "shifted coboundary" equation (5) in \cite{Babelon1996},
\begin{equation}
J_{12}(x)=\Delta M(x)[\id\otimes M(x)]^{-1}[M(xq^{H_2})\otimes \id]^{-1}.
\end{equation}

Using this formalism, there are many important results, for example the Arnaudon--Buffenoir--Rogoucy--Roche equation \cite{Arnaudon1998} was found and also \cite{Jimbo1999} Jimbo--Konno--Odake--Shiraishi twistor between the elliptic vertex type and face type $\check{R}$ matrices. For the general case beyond $\mathfrak{sl}_2$, they were done by Etingof--Varchenko \cite{EtingofVarchenko1999}.

{\bf{Twist dynamical $\check{R}$ matrices to dynamical $\check{R}$ matrices}}
\hfill\\
In order to twist dynamical $\check{R}$ matrices to dynamical $\check{R}$ matrices based on the groupoid graded vector spaces language, we need to find the twist equations which generalize the equation \eqref{eq:int_twist_1} and \eqref{eq:intro_twist_2}.

We first reformulate \eqref{eq:int_twist_1}, an invertible operator $J:V\otimes V\to V\otimes V$ is a twist if there exists an invertible operator $Q\in \End(V\otimes V\otimes V)$ such that
\begin{equation}\label{eq:intr_twist_3}
\begin{split}
&\check{R}_J=J^{-1}\check{R}J\\
&\check{R}_J\otimes\operatorname{id}=Q(\check{R}\otimes\operatorname{id})Q^{-1}\\
&\operatorname{id}\otimes \check{R}_J=Q(\operatorname{id}\otimes\check{R})Q^{-1}\\
\end{split}
\end{equation}
The relation between \eqref{eq:int_twist_1} and \eqref{eq:intr_twist_3} is in the section \ref{sec:Drinfeld_twist}. Similarly, we reformulate the equation \eqref{eq:intro_twist_2} as the following \eqref{eq:int_twist_4}.
\begin{equation}\label{eq:int_twist_4}
	\begin{aligned}
		&\check{R}_J=J^{-1}\check{R}J\\
		&\check{R}_J(\lambda)\otimes \operatorname{id}=Q(\lambda)(\check{R}\otimes \operatorname{id})Q^{-1}(\lambda)\\
		&\operatorname{id}\otimes \check{R}_J(\lambda-h^1)=Q(\lambda)(\operatorname{id}\otimes \check{R})Q^{-1}(\lambda)	
	\end{aligned}
\end{equation}
The relation between \eqref{eq:intr_twist_3} and \eqref{eq:int_twist_4} is in 
the section \ref{sec:dynamical_twist}. 

One obstacle when we want to twist dynamical $\check{R}$ matrix on a groupoid $\pi_1$ to get a new dynamical $\check{R}$ matrix on another groupoid $\pi_2$ is to relate the two different groupoid structures. We first introduce some groupoid $\pi$ that connects $\pi_1$ and $\pi_2$, then we introduce the notion of connecting system which are arrows in $\pi$ that relate the paths in these two groupoid. We distinguish two cases, the unique connecting system in subsection \ref{sub:unique_twist} and the quasi-unique connecting system in the subsection \ref{sub:quasi_unique}. For the unique connecting system, the twist equations is similar to \eqref{eq:intr_twist_3} and \eqref{eq:int_twist_4}.

For a quasi-unique connecting system and an invertible map $j$, we can define 
\begin{equation}
	\check{R}_2(\beta_2,\beta_1)=j^{-1}*_{\circ}(\check{R_1}j):=\sum_{\beta}j^{-1}(\beta,\beta_2)\check{R}_1j(\beta_1,\beta),
\end{equation} 
$j$ is said to be a twist if it satisfies the following two conditions 
\begin{itemize}
	\item We have the weight 0 conditions \eqref{eq:intro_ice_rule} for any $\beta_2$ with $t_{\pi}(\beta_2)=s_{\pi_2}(\gamma)$ and any $s_{\pi}(\beta'_2)=s_{\pi}(\beta_2)$,graphically it is \eqref{eq:graph_uqtwist}.\\
	\item The factorization property \eqref{eq:intro_factorization} for invertible $q_1,q_2,q$ and path $\gamma$ in $\pi_2$, graphically it is \eqref{eq:graph_qutwist_q23}.
\end{itemize} 
\begin{subequations}\label{eq:intro_ice_rule}
	\begin{equation}\label{eq:intro_quntwist23}
		\operatorname{id}_{V^{\pi_2}}\otimes (\check{R}_2)_{\lhd}\delta_{\beta_2,\beta'_2}=\big(q_2^{-1}*_{\circ}\big((\operatorname{id}_{V^{\pi_1}}\otimes \check{R}_1)q_2\big)\big)(\beta_2,\beta'_2)
	\end{equation}
	\begin{equation}\label{eq:intro_quntwist12}
		(\check{R}_2)_{\lhd}\otimes \operatorname{id}_{V^{\pi_2}}\delta_{\beta_2,\beta'_2}=\big(q_2^{-1}*_{\circ}\big((\check{R}_1\otimes \operatorname{id}_{V^{\pi_1}})q_2\big)\big)(\beta_2,\beta'_2)
	\end{equation}
\end{subequations} 

\begin{subequations}\label{eq:intro_factorization}
	\begin{equation}
	q|_{(V^{\pi_2})^{\otimes n}_{\gamma\gamma'}}=q_1|_{(V^{\pi_2})^{\otimes n-3}_{\gamma'}}*_{\otimes}q_2|_{(V^{\pi_2})^{\otimes 3}_{\gamma}}
	\end{equation}
	\begin{equation}\label{eq:intro_qutwist_q23}
		\operatorname{id}^{\otimes n-2}_{V^{\pi_2}}\otimes (\check{R}_2)_{\lhd}|_{(V^{\pi_2})^{\otimes n}_{\gamma\gamma'}}=(q^{-1}*_{\circ}\big((\operatorname{id}^{\otimes n-2}_{V^{\pi_1}}\otimes \check{R}_1)q\big))_{\lhd}
	\end{equation}
	\begin{equation}\label{eq:intro_qutwist_q12}
		\id^{\otimes n-3}_{V^{\pi_2}}\otimes (\check{R}_2)_{\lhd}\otimes \operatorname{id}_{V^{\pi_2}}|_{(V^{\pi_2})^{\otimes n}_{\gamma\gamma'}}=(q^{-1}*_{\circ}\big((\id^{\otimes n-3}_{V^{\pi_1}}\otimes \check{R}_1\otimes \operatorname{id}^{\otimes n-2}_{V^{\pi_1}})q\big))_{\lhd}
	\end{equation}
\end{subequations}

\hfill\\

The structure of the paper is organized as follows. In section \ref{sec:prelim}, we introduce the main setting of groupoid graded vector space, the YBE and dynamical YBE. In section \ref{sec:intertwiner}, we first develop the language of modules of convolution algebra in \ref{sub:module_convol} in order to express "relation between transfer matrix" precisely. And then in subsection \ref{sub:int_bet_YBE}, \ref{sub:weight_zero_relation} and \ref{sub:transposed}, we develop the equivalence relations and analyze the compatibility relations implied by the existence of intertwiner. In subsection \ref{sub:intertwiner_spectral} and \ref{sub:8v_sos_example}, we study the intertwiner with spectral parameters and analyze the example of eight vertex sos intertwiner and the equivalence relation.

In section \ref{sec:Twist}, we first develop the language of transfer operators and its operations in \ref{sub:transfer_operator} and then define the twist and connecting system in \ref{sub:unique_twist} and \ref{sub:quasi_unique}. In subsection \ref{sub:intertwiner_from_twist}, we discuss the relation between intertwiner and twist.

In section \ref{sec:Drinfeld_twist},\ref{sec:dynamical_twist} and \ref{sec:example_cell_system}, we discuss the examples of different type of twists.

\hfill\\
{\bf Acknowledgment} The author would like to thank Anton Alekseev for discussion on associators, thank Giovanni Felder and Anic Jelena for invitations to Zurich, introducing their works and discussions. Research of the author is supported by the grant number 208235 of the Swiss National Science Foundation (SNSF) and by the NCCR SwissMAP of the SNSF.

\section{Preliminary}\label{sec:prelim}
We first recall the groupoid language developed in \cite{Felder2020}, the definition of Yang--Baxter operator and dynamical Yang--Baxter operator in this setting.

\subsection{Groupoid graded vector space}\label{sub:groupoid_graded}
A groupoid $\pi$ is a small category in which every morphism is invertible. We denote its object by $\text{Ob}(\pi)$, for any $a,b\in \text{Ob}(\pi)$, the set of morphisms from an object $a$ to $b$ is denoted by $\pi(a,b)$, the morphisms are also called as arrows, the composition of arrows $\gamma\in \pi(a,b)$ and $\eta\in \pi(a,b)$ is denoted by $\eta\circ \gamma\in \pi(a,c)$. The object of $\pi$ is identified with the identity arrows. 

By abuse of notation, the set of arrows of $\pi$ is also denoted by $\pi$, then we can denote the source and target maps by $s,t:\pi\to \text{Ob}(\pi)$. For $a\in \rm{Ob}(\pi)$, the source fibers are defined by $s^{-1}(a)$ and the target fibers are $t^{-1}(a)$.

\begin{example}\label{ex:graph}
	For any unoriented graph $\mathcal{G}$, it can be seen as a groupoid $\pi(\mathcal{G})$ by taking the vertices as objects, for each unoriented edge $e\in \mathcal{G}$, we have corresponding two inverse direction arrows $\alpha_{e},\alpha^{-1}_{e}\in \pi(\mathcal{G})$.
\end{example}

\begin{example}[Action groupoid]
	Let $A$ be a set with right group action $A\times G\to A$, the action groupoid $A\rtimes G$ has the set of objects $A$ and for each $a'=ag$, there is an arrow $a\xrightarrow{g}a'$, thus an arrow is described by a pair $(a,g)\in A\times G$. The source and target are $s(a,g)=a, t(a,g)=ag$ and the composition is
\begin{align*}
	(a',g')\circ (a,g)=(a,gg'),\quad \text{with}~a'=ag,
\end{align*}
the identity arrows are $(a,e),a\in A$ and the inverse of $(a,g)$ is $(ag,g^{-1})$.
\end{example}

With groupoids, we can define the groupoid graded vector spaces with certain finite conditions.

\begin{definition}
Let $\pi$ be a groupoid, a $\pi$-graded vector space of finite type over a field $k$ is a collection $(V_{\alpha})_{\alpha\in \pi}$ of finite dimensional vector spaces indexed by the arrows of $\pi$, such that for each $a\in \rm{Ob}(\pi)$, there are only finitely many arrows $\alpha$ with source or target $a$ and nonzero $V_{\alpha}$. For $a\in \rm{Ob}(\pi)$, the $a$ source fibers of the vector space $V\in \pi$ is the direct sum of components with source $a$, that is  $\oplus_{\alpha\in s^{-1}(a)}V_{\alpha}$. And similar for the target fibers.

\end{definition}

The $k$ vector space $\Hom(V,W)$ of two $\pi$ graded vector spaces consists of families $(f_{\alpha})_{\alpha\in \pi}$ of linear maps $f_{\alpha}:V_{\alpha}\to W_{\alpha}$, that is $f_{\alpha}\in \End_{k}(V)$, the composition is also defined componentwise. The category of $\pi$ groupoid graded vector spaces is a monoidal category, we denote by $\text{Vect}_k(\pi)$, with the tensor product defined
\begin{align*}
	(V\otimes W)_{\gamma}=\oplus_{\beta\circ \alpha=\gamma}V_{\alpha}\otimes W_\beta,\quad V,W\in \text{Vect}_k(\pi)
\end{align*}

For the $k$ additive category $\rm{Vect}_{k}(\pi)$, we can also define the dual groupoid graded vector space and the resulting category is an abelian pivotal monoidal category, see section 2 of \cite{Felder2020}.

\subsection{Yang--Baxter operator and its restriction on source fibers}\label{sub:YBE_operator}
Let $k=\C$, a \emph{Yang-Baxter operator} on $V\in \text{Vect}_k(\pi)$ is a meromorphic function $z\to \check{R}(z)\in \End(V\otimes V)$ of the spectral parameter $z\in \C$ with values in the endormorphisms of $V\otimes V$, obeying the Yang--Baxter equation
\begin{align}\label{eq:YBE}
	\check{R}(z-w)^{(23)}\check{R}(z)^{(12)}\check{R}(w)^{(23)}=\check{R}(w)^{(12)}\check{R}(z)^{(23)}\check{R}(z-w)^{(12)}
\end{align}
and also the inversion relation
\begin{equation}\label{eq:inversion_R}
\check{R}(z)\check{R}(-z)=\id_{V\otimes V}
\end{equation}

The restriction of $\check{R}(z)$ to $V_{\alpha}\otimes V_{\beta}$ for composable arrows $\alpha,\beta$ has components in each direct summand of the decomposition
\begin{align*}
	\check{R}(z)|_{V_{\alpha}\otimes V_{\beta}}=\oplus_{\gamma,\delta}\mathcal{W}(z;\alpha,\beta,\gamma,\delta),
\end{align*}
the sum is over $\gamma,\delta$ such that $\beta\circ \alpha=\delta\circ \gamma$ with the component $\mathcal{W}(z;\alpha,\beta,\gamma,\delta)\in \Hom_k(V_{\alpha}\otimes V_{\beta},V_{\gamma}\otimes V_{\delta})$, graphically it is simply presented as the square \eqref{eq:square}.
\begin{equation}\label{eq:square}
\begin{tikzcd}
a_4\arrow[d,"V_{\alpha}"']\arrow[rd,phantom,"\mathcal{W}"]\arrow[r,"V_{\gamma}"]&a_3\arrow[d,"V_{\delta}"]\\
a_1\arrow[r,"V^{\beta}"']&a_2
\end{tikzcd}
\end{equation}

When restricting to the source fibers  $s^{-1}(a)$ of $V\otimes V$, $a\in \text{Ob}(\pi)$, the Yang--Baxter operator $\check{R}(z)$ can be written as $\check{R}(z,a)$ and 
\[
\check{R}(z,a)\in \oplus_{\alpha\in s^{-1}(a)}\End_{k}\big((V\otimes V)_{\alpha}\big),
\]
the Yang--Baxter equation then turns to the following dynamical Yang-Baxter equation
\begin{equation}\label{eq:dYBE}
	\begin{split}
		&\check{R}^{(23)}(z-w,ah^{(1)})\check{R}^{(12)}(z,a)\check{R}^{(23)}(w,ah^{(1)})\\
		&=\check{R}^{(12)}(w,a)\check{R}^{(23)}(z,ah^{(1)})\check{R}^{(12)}(z-w,a)
	\end{split}
\end{equation}
here we use the "dynamical" notation with the placeholder $h^{(i)}$:
\begin{align*}
	\check{R}^{(23)}(ah^{(1)})(u\otimes v\otimes w)=u\otimes \check{R}(t(\alpha_1))(v\otimes w),\quad \text{if} \quad u\in V_{\alpha_1},s(\alpha_1)=a
\end{align*}

\begin{example}

In the case of an action groupoid $\pi=A\rtimes G$ and its subgroupoids, the operator can be written explicitly as
\begin{equation}
	\check{R}(z,d)\in \oplus_{g\in G}\End_{k}((V\otimes V)_{(d,g)}),\quad (V\otimes V)_{(d,g)}=\sum_{h\in G}V_{(d,h)}\otimes V_{(dh,h^{-1}g)},
\end{equation}
for any composable edges, we have
\begin{align*}
	\check{R}(z,d)|_{V_{(d,g_1)}\otimes V_{(a,g_2)}}=\oplus_{(d,g_3),(c,g_4)}\mathcal{W}\big(x;(d,g_1),(a,g_2),(d,g_3),(c,g_4)\big),
\end{align*} 
where  $\mathcal{W}\big(x;(d,g_1),(a,g_2),(d,g_3),(c,g_4)\big)\in \Hom_k(V_{(d,g_1)}\otimes V_{(a,g_2)},V_{(d,g_3)}\otimes V_{(c,g_3)})$, graphically it is the following with $d$ in the left upper corner.

\begin{equation}
\begin{tikzcd}
d\arrow[r,"V_{(d,g_3)}"]\arrow[d,"V_{(d,g_1)}"']&c\arrow[d,"V_{(c,g_4)}"]\\
a\arrow[r,"V_{(a,g_2)}"']&b
\end{tikzcd}
\end{equation}
\end{example}

\section{Intertwiner}\label{sec:intertwiner}
For two groupoids $\pi_1,\pi_2$, suppose that we have a set $\pi$ with the map  $\operatorname{Ob}(\pi_1)\stackrel{s_{\pi}}\leftarrow \pi\stackrel{t_{\pi}}\rightarrow\operatorname{Ob}(\pi_2)$, with the source and target map defined, we can view $\pi$ as arrows directing from $\pi_1$ to $\pi_2$,
\begin{equation*}
\begin{tikzcd}
\rm{Ob}(\pi_1)\arrow[r,"\pi"]&\rm{Ob}(\pi_2),
\end{tikzcd}
\end{equation*} we further assume that $t_{\pi}$ and $s_{\pi}$ are surjective. The left, right and left-right fiber products are formed respectively by
\begin{equation}\label{eq:left_module}
	\pi_1\times_{\rm{Ob}(\pi_1)}\pi:=\{(\alpha,\beta)|t_{\pi_1}(\alpha)=s_{\pi}(\beta)\}
\end{equation}
\begin{equation}\label{eq:right_module}
	\pi\times_{\rm{Ob}(\pi_2)}\pi_2:=\{(\beta,\gamma)|t_{\pi}(\beta)=s_{\pi_2}(\gamma)\}
\end{equation}
\begin{equation}\label{eq:bimodule}
\pi_1\times_{\rm{Ob}(\pi_1)}\pi\times_{\rm{Ob}(\pi_2)}\pi_2:=\{(\alpha,\beta,\gamma)|t_{\pi_1}(\alpha)=s_{\pi}(\beta),t_{\pi}(\beta)=s_{\pi_2}(\gamma)\}
\end{equation}

The left $\pi_1$-module generated by the set $\pi$ is the set $\pi_1\times_{\rm{Ob}(\pi_1)}\pi$ together with the action map
\begin{equation}\label{eq:left_action}
\begin{split}
&\pi_1*_{\pi_1} \big(\pi_1\times_{\rm{Ob}(\pi_1)}\pi\big)\to \big(\pi_1\times_{\rm{Ob}(\pi_1)}\pi\big)\\
&\alpha_1*_{\pi_1}(\alpha,\beta)\mapsto (\alpha\circ \alpha_1,\beta),
\end{split}
\end{equation}
and similarly definition for the right $\pi_2$ module structure on $\pi\times_{\rm{Ob}(\pi_2)}\pi_2$ and the $(\pi_1,\pi_2)$ bimodule structure on the set $\pi_1\times_{\rm{Ob}(\pi_1)}\pi\times_{\rm{Ob}(\pi_2)}\pi_2$.

By viewing the object of a groupoid as identity arrow, we also have the following relations
\begin{subequations}
\begin{equation}
	\pi_1\times_{\rm{Ob}(\pi_1)}\pi \subset \pi_1\times_{\rm{Ob}(\pi_1)}\pi\times_{\rm{Ob}(\pi_2)}\pi_2
\end{equation}
\begin{equation}
	\pi\times_{\rm{Ob}(\pi_2)}\pi_2\subset \pi_1\times_{\rm{Ob}(\pi_1)}\pi\times_{\rm{Ob}(\pi_2)}\pi_2
\end{equation}
\end{subequations}

For simplicity of notations, sometimes we use $\pi$ to also denote the bimodule  $\pi_1\times_{\rm{Ob}(\pi_1)} \pi\times_{\rm{Ob}(\pi_2)} \pi_2$ generated by $\pi$, then the bimodule structure will be denoted by $\pi_1*_{\pi_1}\pi*_{\pi_2}\pi_2$.

\subsection{Modules of convolution algebra}\label{sub:modules_of_convolution_algebra}
\begin{definition}[\cite{Felder2020}]\label{sub:module_convol}
A $\pi_1$ graded algebra $R$ over $k$ is a collection $(R_{\alpha})_{\alpha\in \pi_1}$ of $k$ vector spaces labeled by arrows of $\pi_1$ with bilinear products $R_{\alpha}\times R_{\beta}\to R_{\beta\circ \alpha},(x,y)\to xy$ defined for composable arrows $\alpha,\beta$ and units $1_a\in R_{a}$, for $a\in \operatorname{\pi_1}$ such that $(i) (xy)z=x(yz)$ whenever defined and $(ii) x1_b=x=1_ax$ for all $x\in R_{\alpha}$ of degree $\alpha\in \pi_1(a,b)$.
\end{definition}

The prototype of groupoid graded algebra and the ones that we will use are the following two examples.

\begin{example}\label{exa:lower_case_End}
Let $V^{\pi_1}\in \text{Vect}_k(\pi_1)$ and let $\underline{\operatorname{End}}V^{\pi_1}$ be the $\pi_1$ graded vector space with $(\underline{\operatorname{End}}V^{\pi_1})_{\alpha_1}=\oplus_{\alpha\in \pi(a_1,a_1)}\operatorname{Hom}(V^{\pi_1}_{\alpha_1\circ\alpha\circ\alpha_1^{-1}},V^{\pi_1}_{\alpha})$ where $a_1=s(\alpha_1)$. Let $\alpha'=\alpha_1\alpha\alpha_1^{-1}$ and $\alpha''=\alpha_2\alpha_1\alpha\alpha_1^{-1}\alpha_2^{-1}$, then the product of $\underline{\operatorname{End}}V^{\pi_1}$ is given by the composition of linear maps
\begin{equation}\label{eq:com_End_pi_1}
	\operatorname{Hom}_k(V^{\pi_1}_{\alpha'},V^{\pi_1}_{\alpha})\otimes \operatorname{Hom}_{k}(V^{\pi_1}_{\alpha''},V^{\pi_1}_{\alpha'})\to \operatorname{Hom}_{k}(V^{\pi_1}_{\alpha''},V^{\pi_1}_{\alpha}),
\end{equation}
the equation \eqref{eq:com_End_pi_1} graphically can be described by \eqref{eq:graph_com_End} and the direction of the maps is from bottom to top.
\begin{equation}\label{eq:graph_com_End}
\begin{tikzcd}
	a_1\arrow[r,"V^{\pi_1}_{\alpha}"]\arrow[d,"\alpha_1"]&a_1\arrow[d,"\alpha_1"]\\
	a_2\arrow[r,"V^{\pi_1}_{\alpha'}"]\arrow[d,"\alpha_2"]&a_2\arrow[d,"\alpha_2"]\\
	a_3\arrow[r,"V^{\pi_1}_{\alpha''}"]&a_3
\end{tikzcd}
\quad
\begin{tikzcd}
{}&\\
&\\
{}\arrow[uu,"\rm{direction}"']&
\end{tikzcd}
\end{equation}
and unit $1_{\alpha_1}=\oplus_{\alpha\in \pi(a_1,a_1)}\operatorname{id}_{V_{\alpha}}$ is a $\pi_1$ graded algebra.
\end{example}

\begin{example}\label{exa:upper_case_End}
Let $V^{\pi_1}\in \text{Vect}_k(\pi_1)$ and let $\overline{\End}V^{\pi_1}$ be the $\pi_1$ graded vector space with $(\overline{\End})_{\alpha_1}=\oplus_{\alpha\in \pi(a_1,a_1)}\Hom(V^{\pi_1}_{\alpha},V^{\pi_1}_{\alpha_1\circ\alpha\circ\alpha^{-1}_1})$, where $a_1=s(\alpha_1)$. Let $\alpha'=\alpha_1\circ\alpha\circ\alpha^{-1}_1$ and $\alpha''=\alpha_2\alpha_1\alpha\alpha^{-1}_1\alpha^{-1}_2$, then the product of $\overline{\End}V^{\pi_1}$ is given by the composition of linear maps
\begin{equation}\label{eq:com_end_pi_1}
\Hom(V^{\pi_1}_{\alpha'},V^{\pi_1}_{\alpha''})\otimes\Hom_k(V^{\pi_1}_{\alpha},V^{\pi_1}_{\alpha'})\to \Hom_k(V^{\pi_1}_{\alpha},V^{\pi_1}_{\alpha''})
\end{equation}
and the equation \eqref{eq:com_end_pi_1} can be described graphically as \eqref{eq:graph_com_end} and the direction is horizontal from left to right.
\begin{equation}\label{eq:graph_com_end}
\begin{tikzcd}
	a_1\arrow[d,"V^{\pi_1}_{\alpha}"]\arrow[r,"\alpha_1"]&a_2\arrow[r,"\alpha_2"]\arrow[d,"V^{\pi_1}_{\alpha'}"]&a_3\arrow[d,"V^{\pi_1}_{\alpha''}"]\\
	a_1\arrow[r,"\alpha_1"]&a_2\arrow[r,"\alpha_2"]&a_3\\
\end{tikzcd}
\begin{tikzcd}
{}\arrow[rr,"\rm{direction}"]&&{}\\
\end{tikzcd}
\end{equation}

\end{example}

\begin{definition}
Let $R$ be a $\pi_1$ graded algebra over $k$, an $R$ module over $\pi$ is a collection $(Q_{\beta})_{\beta\in \pi}$ labeled by $\beta\in \pi$ with bilinear maps $R_{\alpha}\times Q_{\beta}\to Q_{\alpha{*_{\pi_1}}\beta},(x,q)\to x*_{\pi_1}q$ and satisfies the relation $(i) x(yq)=(xy)q$ and $1_a q=q$ for all $q\in Q_{\beta}$ with $s_{\pi}(\beta)=a$.

Here we make a little abuse of notations $*_{\pi_1}$ to denote both the module structure of the set $\pi$ and the graded module $Q$.
\end{definition}

\begin{example}
Let $V^{\pi_1}\in \text{Vect}_k(\pi_1)$ $V^{\pi_2}\in  \text{Vect}_k(\pi_2)$ and we use $\overline{\operatorname{Hom}}(V^{\pi_1},V^{\pi_2})$ denote the $\pi$ graded vector space, $\overline{\operatorname{Hom}}(V^{\pi_1},V^{\pi_2})_{\beta}=\oplus_{\alpha,\gamma}\operatorname{Hom}(V^{\pi_1}_{\alpha},V^{\pi_2}_{\gamma})$, with $\alpha\in \pi_1(s_{\pi}(\beta),s_{\pi}(\beta)),\gamma\in \pi_2(t_{\pi}(\beta),t_{\pi}(\beta))$, graphically it is described by
\begin{equation}\label{eq:graph_Hom}
\begin{tikzcd}
a_1\arrow[r,"\beta"]\arrow[d,"V^{\pi_1}_{\alpha}"']&c_1\arrow[d,"V^{\pi_2}_{\gamma}"]\\
a_1\arrow[r,"\beta"]&c_1\\
\end{tikzcd}
\end{equation}

The graded vector space $\overline{\operatorname{Hom}}(V^{\pi_1},V^{\pi_2})$ is an $\big(\overline{\operatorname{End}}(V^{\pi_1}),\overline{\operatorname{End}}(V^{\pi_2})\big)$ bimodule and the bimodule structure is easily described by the picture
\begin{equation}
\begin{tikzcd}
a_1\arrow[r,"\alpha_1"]\arrow[d,"V^{\pi_1}_{\alpha}"]&a_2\arrow[d,"V^{\pi_1}_{\alpha'}"]\\
a_1\arrow[r,"\alpha_1"]&a_2\\
\end{tikzcd}
\times
\begin{tikzcd}
a_2\arrow[r,"\beta"]\arrow[d,"V^{\pi_1}_{\alpha'}"]&c_1\arrow[d,"V^{\pi_2}_{\gamma}"]\\
a_2\arrow[r,"\beta"]&c_1\\
\end{tikzcd}
\mapsto
\begin{tikzcd}
a_1\arrow[r,"\alpha_1*_{\pi_1}\beta"]\arrow[d,"V^{\pi_1}_{\alpha}"']&c_1\arrow[d,"V^{\pi_2}_{\gamma}"]\\
a_1\arrow[r,"\alpha_1*_{\pi_1}\beta"']&c_1\\
\end{tikzcd}
\end{equation}
\begin{equation}
\begin{tikzcd}
	a_2\arrow[r,"\beta"]\arrow[d,"V^{\pi_1}_{\alpha'}"]&c_1\arrow[d,"V^{\pi_2}_{\gamma}"]\\
	a_2\arrow[r,"\beta"]&c_1
\end{tikzcd}
\begin{tikzcd}
c_1\arrow[d,"V^{\pi_2}_{\gamma}"]\arrow[r,"\gamma_1"]&c_2\arrow[d,"V^{\pi_2}_{\gamma'}"]\\
c_1\arrow[r,"\gamma_1"]&c_2
\end{tikzcd}
\mapsto
\begin{tikzcd}
a_2\arrow[r,"\beta*_{\pi_2}\gamma_1"]\arrow[d,"V^{\pi_1}_{\alpha}"']&c_2\arrow[d,"V^{\pi_2}_{\gamma'}"]\\
a_2\arrow[r,"\beta*_{\pi_2}\gamma_1"]&c_2
\end{tikzcd}
\end{equation}

More generally, it can be graded by two arrows of $\pi$, $\overline{\operatorname{Hom}}(V^{\pi_1},V^{\pi_2})_{\beta_1,\beta_2}=\oplus_{\alpha,\gamma}\operatorname{Hom}(V^{\pi_1}_{\alpha},V^{\pi_2}_{\gamma})$, with $\alpha\in \pi_1(s_{\pi}(\beta_2),s_{\pi}(\beta_1)),\gamma\in \pi_2(t_{\pi}(\beta_2),t(\pi)(\beta_1))$, and graphically it is
\begin{equation}
\begin{tikzcd}
	a\arrow[r,"\beta_2"]\arrow[d,"V^{\pi_1}_{\alpha}"']&c\arrow[d,"V^{\pi_2}_{\gamma}"]\\
	a'\arrow[r,"\beta_1"]&c'
\end{tikzcd}
\quad \rm{Or}\quad
\begin{tikzcd}
	c\arrow[r,"V^{\pi_2}_{\gamma}"]&c'\\
	a\arrow[r,"V^{\pi_1}_{\alpha}"]\arrow[u,"\beta_2"]&a'\arrow[u,"\beta_1"]
\end{tikzcd}
\end{equation}
\end{example}

And similarly we can define the space $\overline{\operatorname{Hom}}(V^{\pi_2},V^{\pi_1})$, it is graphically 
\begin{equation}
\begin{tikzcd}
	c\arrow[d,"V^{\pi_2}_{\gamma}"']&a\arrow[l,"\beta_1"']\arrow[d,"V^{\pi_1}_{\alpha}"]\\
	c'&a'\arrow[l,"\beta_2"]
\end{tikzcd}
\end{equation}
and also transfer operators of various tensor products of vector spaces.

\begin{definition}\cite{Felder2020}
	Let $R$ be a $\pi_1$ graded algebra, the convolution algebra $\Gamma(\pi_1,R)$ with coefficients in $R$ is the $k$ algebra of maps $f:\pi_1\to \sqcup_{\alpha\in \pi}R_{\alpha}$ such that
\begin{enumerate}
	\item $f(\alpha)\in R_{\alpha}$ for all arrows $\alpha\in \pi_1$,\\
	\item for every $a\in A$, there are finitely many $\alpha\in s^{-1}(a)\cup t^{-1}(a)$ such that $f(\alpha)\ne 0$.
\end{enumerate}

The product of the convolution algebra is the convolution product
\begin{equation}
g*f(\alpha)=\sum_{\alpha_2\circ\alpha_1=\alpha}f(\alpha_2)g(\alpha_1)
\end{equation}
\end{definition}

\begin{definition}
Let $Q$ be an $R$ module, the $\Gamma(\pi_1,R)$ module $\Gamma(\pi,Q)$ with coefficients in $Q$ is the $k$ algebra of maps $l:\pi\to \sqcup_{\beta\in \pi}Q_{\beta}$ such that
\begin{enumerate}
	\item $l(\beta)$ in $Q_{\beta}$ for all arrows $\beta\in \pi$.\\
	\item for every $a\in \operatorname{Ob}(\pi_1)$, there are finitely many $\beta\in s^{-1}_{\pi}(a)$ such that $l(\beta)\ne 0$.\\
\end{enumerate}

The $\Gamma(\pi_1,R)$ module structure of $\Gamma(\pi,Q)$ structure is the following,
\begin{equation}
f*_{\pi_1}l(\beta)=\sum_{\alpha'*_{\pi_1}\beta'=\beta}l(\beta')f(\alpha')
\end{equation}
\end{definition}

Similarly, we can define the bimodules of two convolution algebras.

\subsection{Intertwiner between Yang--Baxter operators}\label{sub:int_bet_YBE}
\begin{definition}
Suppose that $\check{R}_1(z)\in\operatorname{End}(V^{\pi_1}\otimes_{\pi_1} V^{\pi_1})$ and $\check{R}_2(z)\in \operatorname{End}(V^{\pi_2}\otimes_{\pi_2} V^{\pi_2})$ are two Yang--Baxter operators on the space of $V^{\pi_1}$ and $V^{\pi_2}$, a map $C\in \operatorname{Hom}(V^{\pi_1}\otimes_{*\pi_1}V^{\pi}, V^{\pi}\otimes_{*\pi_2} V^{\pi_2})$ is called an $(\check{R_1},\check{R_2})$ intertwiner if it satisfies the following RCC relation
\begin{equation}\label{eq:RCC}
C^{(12)}C^{(23)}\check{R}^{(12)}_1(z)=\check{R}^{(23)}_2(z)C^{(12)}C^{(23)}
\end{equation}
in the space $\operatorname{Hom}(V^{\pi_1}\otimes_{\pi_1} V^{\pi_1}\otimes_{*\pi_1}V^{\pi}, V^{\pi}\otimes_{*\pi_2}V^{\pi_2}\otimes_{\pi_2} V^{\pi_2})$.
\end{definition}

Now we consider the composition of intertwiners, suppose that we have another groupoid $\rm{Ob}(\pi_2)\stackrel{s_{\hat{\pi}}}\leftarrow \hat{\pi}\stackrel{t_{\hat{\pi}}}\rightarrow\rm{Ob}(\pi_3)$. Let $\widehat{C}\in \Hom(V^{\pi_2}\otimes_{*_{\pi_2}}V^{\hat{\pi}},V^{\hat{\pi}}\otimes_{*_{\pi_3}}V^{\pi_3})$ be an $(\check{R}_2,\check{R}_3)$ intertwiner,which satisfies the relation
\begin{equation}
\widehat{C}^{(12)}\widehat{C}^{(23)}\check{R}^{(12)}_2(z)=\check{R}^{(23)}_3(z)\widehat{C}^{(12)}\widehat{C}^{(23)}.
\end{equation}

The composition of two groupoids $\hat{\pi}\circ \pi$ connects the groupoid $\pi_1$ and $\pi_3$ and also the homomorphism $\widehat{C}^{23}C^{12}\in \Hom(V^{\pi_1}\otimes V^{\pi}\otimes V^{\hat{\pi}},V^{\pi}\otimes V^{\hat{\pi}}\otimes V^{\pi_3})$
\begin{equation}
\rm{Ob}(\pi_1)\stackrel{s_{\hat{\pi}\circ \pi}}\leftarrow \hat{\pi}\circ \pi\stackrel{t_{\hat{\pi}\circ \pi}}\rightarrow\rm{Ob}(\pi_3)
\end{equation}

\begin{prop}\label{prop:composition}
$\widehat{C}^{(23)}C^{(12)}$ is an $(\check{R}_1,\check{R}_3)$ intertwiner, it satisfies the equation
\begin{equation}
	\widehat{C}^{(23)}C^{(12)}\widehat{C}^{(34)}C^{(23)}\check{R}_1(z)^{(12)}=\check{R}_3(z)^{(34)}\widehat{C}^{(23)}C^{(12)}\widehat{C}^{(34)}C^{(23)}
\end{equation}
\end{prop}

From the intertwiner, we can construct some modules of convolution algebra, we first discuss some inversion relations.  The map $C\in \operatorname{Hom}(V^{\pi_1}\otimes_{*\pi_1}V^{\pi}, V^{\pi}\otimes_{*\pi_2} V^{\pi_2})$ is called \emph{left invertible} if there exists a map $C_l^{-1}\in \operatorname{Hom}(V^{\pi}\otimes_{*_{\pi_2}} V^{\pi_2},V^{\pi_1}\otimes_{*\pi_1}V^{\pi})$ such that
\begin{equation}\label{eq:left_invert}
C_l^{-1}C=\operatorname{id}_{V^{\pi_1}\otimes V^{\pi}}\in \operatorname{Hom}(V^{\pi_1}\otimes_{*\pi_1}V^{\pi},V^{\pi_1}\otimes_{*\pi_1}V^{\pi})
\end{equation}
and graphically it is presented as \eqref{eq:gra_left_invertible} and notice the convention for the order of composition of the block,
\begin{equation}\label{eq:gra_left_invertible}
\sum_{\beta,\gamma,c_1}
\begin{tikzcd}
a_1\arrow[r,"V^{\pi}_{\beta}"]\arrow[rd,phantom,"C"]\arrow[d,"V^{\pi_1}_{\alpha_1}"']&c_1\arrow[d,"V^{\pi_2}_{\gamma}"]\\
a_2\arrow[r,"V^{\pi}_{\beta_1}"']&c_2
\end{tikzcd}
\circ
\begin{tikzcd}
a_1\arrow[r,"V^{\pi_1}_{\alpha_2}"]\arrow[rd,phantom,"C_l^{-1}"]\arrow[d,"V^{\pi}_{\beta}"']&a'_2\arrow[d,"V^{\pi}_{\beta_2}"]\\
c_1\arrow[r,"V^{\pi_2}_{\gamma}"']&c_2
\end{tikzcd}
=\delta_{\alpha_1,\alpha_2}\delta_{\beta_1,\beta_2}
\begin{tikzcd}
a_1\arrow[r,"V^{\pi}_{\alpha_1}"]\arrow[rd,phantom,"C_l^{-1}C"]\arrow[d,"V^{\pi_1}_{\alpha_1}"']&a_2\arrow[d,"V^{\pi}_{\beta_1}"]\\
a_2\arrow[r,"V^{\pi}_{\beta_1}"']&c_2
\end{tikzcd}
\end{equation}

And $C$ is called \emph{right invertible} if there exists $C^{-1}_r\in \Hom(V^{\pi}\otimes_{*\pi_2}V^{\pi_2},V^{\pi_1}\otimes_{*\pi})$ that satisfies the equation
\begin{equation}\label{eq:right_invert}
CC_r^{-1}=\operatorname{id}_{V^{\pi}\otimes V^{\pi_2}}\in \operatorname{Hom}(V^{\pi}\otimes_{*\pi_2} V^{\pi_2},V^{\pi}\otimes_{*\pi_2} V^{\pi_2})
\end{equation}
and graphically the equation \eqref{eq:right_invert} is presented as
\begin{equation}
\sum_{\alpha,\beta,a_2}
\begin{tikzcd}
a_1\arrow[d,"V^{\pi}_{\beta_1}"']\arrow[rd,phantom,"C_r^{-1}"]\arrow[r,"V^{\pi_1}_{\alpha}"]&a_2\arrow[d,"V^{\pi}_{\beta}"]\\
c_1\arrow[r,"V^{\pi_2}_{\gamma_1}"']&c_2\\	
\end{tikzcd}
\circ
\begin{tikzcd}
a_1\arrow[r,"V^{\pi}_{\beta_2}"]\arrow[rd,phantom,"C"]\arrow[d,"V^{\pi_1}_{\alpha}"']&c'_1\arrow[d,"V^{\pi_2}_{\gamma_2}"]\\
a_2\arrow[r,"V^{\pi}_{\beta}"']&c_2\\
\end{tikzcd}
=\delta_{\beta_1,\beta_2}\delta_{\gamma_1,\gamma_2}
\begin{tikzcd}
a_1\arrow[r,"V^{\pi}_{\beta_1}"]\arrow[rd,phantom,"CC^{-1}_r"]\arrow[d,"V^{\pi}_{\beta_1}"']&c_1\arrow[d,"V^{\pi}_{\gamma_1}"]\\
c_1\arrow[r,"V^{\pi}_{\gamma_1}"']&c_2
\end{tikzcd}
\end{equation}

And $C$ is called \emph{invertible} if it is both left  \eqref{eq:left_invert}  and right invertible \eqref{eq:right_invert} and also
\begin{equation}\label{eq:C_invertible}
C^{-1}:=C^{-1}_l=C^{-1}_r
\end{equation}

Now suppose that we have an $(\check{R}_1(z),V^{\pi_1}),(\check{R}_2(z),V^{\pi_2})$ intertwiner $C$, then its partial trace $\overline{\operatorname{tr}}_{V^{\pi}}C\in \Gamma(\pi,\overline{\operatorname{Hom}}(V^{\pi_1},V^{\pi_2}))$ is defined as follows.
\begin{definition}\label{def_trace}
For each $\beta\in \pi(a,b),\gamma\in \pi_2(b,b),\alpha\in \pi_1(a,a)$, the homogeneous component of $C$ is the linear map
\[
C(\beta,\alpha,\gamma): V^{\pi_1}_{\alpha}\otimes V^{\pi}_{\beta}\to V^{\pi}_{\beta}\otimes V^{\pi_2}_{\gamma},
\]
define $\overline{\tr}_{V^{\pi}}C(\beta,\alpha,\gamma)=\sum_i (e^{*}_{i,\beta}\otimes \id)C(\beta,\alpha,\gamma)(\operatorname{id}\otimes e_{i,\beta})\in \operatorname{Hom}(V^{\pi_1}_{\alpha},V^{\pi_2}_{\gamma})$, for the chosen basis $e_{i,\beta}$ of $V^{\pi}_{\beta}$, the partial trace is defined by 
\begin{equation}
\overline{\tr}_{V^{\pi}}C:\beta\to \oplus_{\gamma\in \pi_2(b,b),\alpha\in \pi_1(a,a)}\operatorname{tr}_{V^{\pi}}C(\beta,\alpha,\gamma)
\end{equation}
\end{definition}

\begin{example}
If we set $\pi=\pi_2=\pi_1$ and $\check{R}_1\in \End(V^{\pi_1})$, then we have
\begin{equation}
\overline{\tr}_{V^{\pi_1}}\check{R}_1\in \Gamma(\pi_1,\overline{\End}(V^{\pi_1}))
\end{equation}
and similarly if we take the trace over the first coordinate, then we will have
\begin{equation}
\underline{\tr}_{V^{\pi_1}}\check{R}_1\in \Gamma(\pi_1,\underline{\End}(V^{\pi_1}))
\end{equation} 
\end{example}

\begin{prop}\label{prp:transfer_relation}
Suppose that $C$ is an $(\check{R_1},\check{R_2})$ intertwiner and is invertible, then we have the following relation \eqref{eq:RC_CR} in the space $\Gamma(\pi,\overline{\Hom}(V^{\pi_1},V^{\pi_2}))$
\begin{equation}\label{eq:RC_CR}
\overline{\tr}_{V^{\pi_1}}\check{R}_1*_{\pi_1}\overline{\tr}_{V^{\pi}}C=\overline{\tr}_{V^{\pi}}C*_{\pi_2}\overline{\tr}_{V^{\pi_2}}\check{R}_2
\end{equation}

\end{prop}
\begin{proof}
From the $RCC$ equation \eqref{eq:RCC}, let $\phi=C$, we have that
\begin{equation*}
C^{(23)}\check{R}^{(12)}_1(z)=(\phi^{-1}\otimes \id)\check{R}^{(23)}_2(z)C^{(12)}(\id\otimes \phi)
\end{equation*}
then we have that
\begin{equation}\label{eq:proof_RCC}
\begin{split}
&\tr_{V^{\pi_1}\otimes V^{\pi}}C^{(23)}\check{R}^{(12)}_1(z)=\tr_{V^{\pi_1}\otimes V^{\pi}}(\phi^{-1}\otimes \id)\check{R}^{(23)}_2(z)C^{(12)}(\id\otimes \phi)\\
&=\tr_{V^{\pi}\otimes V^{\pi_2}}\check{R}^{(23)}_2(z)C^{(12)}\\
\end{split}
\end{equation}

Then for any $\sum_{\alpha',\beta'}\alpha'*_{\pi_1}\beta'=\beta=\sum_{\beta'',\gamma'} \beta''*_{\pi_2}\gamma'$, we take the component $\beta$ of \eqref{eq:proof_RCC} and decompose the trace $\tr_{V^{\pi}\otimes V^{\pi_1}}$ to $\tr_{V^{\pi}}$ and $\tr_{V^{\pi_1}}$, similarly for the $\tr_{V^{\pi_2}\otimes V^{\pi}}$. We get the relation 
\begin{equation}
\sum_{\alpha',\beta'}\overline{\tr}_{V^{\pi}_{\beta'}}C^{(23)}\overline{\tr}_{V^{\pi_1}_{\alpha'}}\check{R}^{(12)}_1(z)=\sum_{\beta'',\gamma'}\overline{\tr}_{V^{\pi_2}_{\gamma'}}\check{R}^{(23)}_2(z)\overline{\tr}_{V^{\pi}_{\beta''}}C^{(12)}.
\end{equation}
this is essentially the relation \eqref{eq:RC_CR}.	
\end{proof}

\subsection{Weight zero relation and compatibility}\label{sub:weight_zero_relation}
\hfill\\
From the RCC relation \eqref{eq:RCC} and assume that $C$ is right invertible (resp. left invertible and invertible), we multiply by the map $(C^{(12)}C^{(23)})_{r}^{-1}=(C^{(12)})^{-1}_r(C^{(23)})_r^{-1}$, we get the equation \eqref{eq:further_relation} in the space $\Hom(V^{\pi}\otimes_{*_{\pi_2}}V^{\pi_2}\otimes_{\pi_2} V^{\pi_2},V^{\pi}\otimes_{*_{\pi_2}}V^{\pi_2}\otimes_{\pi_2} V^{\pi_2})$.
\begin{equation}\label{eq:further_relation}
	C^{(12)}C^{(23)}\check{R}^{(12)}_1(z)\big(C^{(12)}C^{(23)}\big)_r^{-1}=\check{R}^{(23)}_2(z)
\end{equation}
and graphically it is the following \eqref{eq:graph_further_relation} which was observed by Roche as the equation $S'$ \cite{Roche1990}, 
\begin{equation}\label{eq:graph_further_relation}
\sum_{\color{red}\rm{red}}
\begin{tikzcd}
&\circ\arrow[ld,"V^{\pi_2}_{\gamma_1}"']&\bullet\arrow[l,"V^{\pi}_{\beta_1}"']\arrow[ld,red]\arrow[rd,red]\arrow[r,"V^{\pi}_{\beta_2}"]&\circ\arrow[rd,"V^{\pi_2}_{\gamma_3}"]&\\
\circ\arrow[rd,"V^{\pi_2}_{\gamma_2}"']&{\color{red}\bullet}\arrow[l,red]\arrow[rd,red]&\check{R}_1(z)&{\color{red}\bullet}\arrow[r,red]\arrow[ld,red]&\circ\arrow[ld,"V^{\pi_2}_{\gamma_4}"]\\
&\circ&{\color{red}\bullet}\arrow[l,red,"V^{\pi}_{\beta''}"]\arrow[r,red,"V^{\pi}_{\beta''}"']&\circ&
\end{tikzcd}
=\delta_{\beta_1,\beta_2}
\begin{tikzcd}
\bullet\arrow[r,"V^{\pi}_{\beta_1}"]&\circ\arrow[ld,"V^{\pi_2}_{\gamma_1}"']\arrow[rd,"V^{\pi_2}_{\gamma_3}"]&\bullet\arrow[l,"V^{\pi}_{\beta_1}"']\\
\circ\arrow[rd,"V^{\pi_2}_{\gamma_2}"']&\check{R}_2(z)&\circ\arrow[ld,"V^{\pi_2}_{\gamma_4}"]\\
&\circ&
\end{tikzcd}
\end{equation}

Then by taking the trace on $V^{\pi}_{\beta_1}$, for the $\beta_1$ that connect $\gamma_2\circ\gamma_1$, we get the identity
\begin{equation}\label{eq:tr_ice_rule}
\tr_{V^{\pi}}\big(	C^{(12)}C^{(23)}\check{R}^{(12)}_1(z)\big(C^{(12)}C^{(23)}\big)_{r}^{-1}\big)=\tr_{V^{\pi}}\big(\check{R}^{(23)}_2(z)\big),
\end{equation}
these kind of identities are graded by $\beta_1$ or two copies of $\beta_1$.

But $\check{R}_2(z)$ does not depend on $\beta_1$, the equality \eqref{eq:tr_ice_rule} will work for any $\beta\in t_{\pi}^{-1}(t_{\pi}(\beta_1)$, if we assume that $\dim V^{\pi}_{\beta}=n$ for all $\beta\in t_{\pi}^{-1}\big(t_{\pi}(\beta_1)\big)$, then we will have that the LHS of $\eqref{eq:tr_ice_rule}$ are all the same for all $\beta\in t_{\pi}^{-1}\big(t_{\pi}(\beta_1)\big)$. This is the compatibility conditions for the target fibers of $\pi$.

\begin{remark}
We call \eqref{eq:graph_further_relation} weight 0 or ice rule relation, because with the $\delta_{\beta_1,\beta_2}$, $\gamma_2\circ \gamma_1$ and $\gamma_4\circ\gamma_3$ have the same start point and end point.
\end{remark}

\subsection{Transposed intertwiner, groupoid structure and equivalence relation}\label{sub:transposed}

In the previous subsections, we have discussed about the set $\pi$ with the map  $\operatorname{Ob}(\pi_1)\stackrel{s_{\pi}}\leftarrow \pi\stackrel{t_{\pi}}\rightarrow\operatorname{Ob}(\pi_2)$, the set can be seen as arrows pointing from $\pi_1$ to $\pi_2$, now we consider the inverse direction.

Suppose that we have a set $\kappa$ with the map $\operatorname{Ob}(\pi_1)\stackrel{t_{\kappa}}\leftarrow \kappa\stackrel{s_{\kappa}}\rightarrow\operatorname{Ob}(\pi_2)$, with the source and target direction defined, we can view $\kappa$ as arrows from $\pi_2$ to $\pi_1$, we use the letters $\rho$ with subscripts to denote these kinds of arrows.
\begin{equation*}
\begin{tikzcd}
\rm{Ob}(\pi_1)&\arrow[l,"\kappa"]\rm{Ob}(\pi_2)
\end{tikzcd}
\end{equation*}

\begin{example}
For the set $\operatorname{Ob}(\pi_1)\stackrel{s_{\pi}}\leftarrow \pi\stackrel{t_{\pi}}\rightarrow\operatorname{Ob}(\pi_2)$, we can simply reverse the direction of arrows and then we get another set $\pi^T$, $(\pi,\pi^T)$ together give us a groupoid.
\end{example}

The fiber product and modules structures are defined analog to \eqref{eq:left_module},\eqref{eq:right_module} and \eqref{eq:bimodule}. For the modules of convolution algebra, we can now consider the "transpose" of previous example.

\begin{example}
For the vector space $\underline{\End}(V^{\pi_1})$ defined in example \ref{exa:lower_case_End}, we can use the following $T$ presentation of each component
\begin{equation}
\begin{tikzcd}
a_1\arrow[rd,phantom,"(T)"]&a_1\arrow[l,"V^{\pi_1}_{\alpha^{-1}}"']\\
a_2\arrow[u,"\alpha^{-1}_1"]&a_2\arrow[l,"V^{\pi_1}_{(\alpha')^{-1}}"]\arrow[u,"\alpha^{-1}_1"']
\end{tikzcd}
\quad
\begin{tikzcd}
{}\arrow[dd,"\rm{direction}"']\\
{}\\
{}
\end{tikzcd}
\end{equation}
and similarly for the space $\overline{\End}(V^{\pi_1})$, its component can be described by the following
\begin{equation}
\begin{tikzcd}
a_1\arrow[rd,phantom,"(T)"]&a_2\arrow[l,"\alpha^{-1}_1"']\\
a_1\arrow[u,"V^{\pi_1}_{\alpha^{-1}}"]&a_2\arrow[l,"\alpha^{-1}_1"]\arrow[u,"V^{\pi_1}_{(\alpha')^{-1}}"']
\end{tikzcd}
\begin{tikzcd}
{}&&{}\arrow[ll,"\rm{direction}"]
\end{tikzcd}
\end{equation}
\end{example}

\begin{example}
For $V^{\pi_1}\in \text{Vect}_k(\pi_1)$ and $V^{\pi_2}\in \text{Vect}_k(\pi_2)$, let $\overline{\Hom}^{T}(V^{\pi_2},V^{\pi_1})$ denote the $\kappa$ graded vector space with component $\overline{\Hom}^{T}(V^{\pi_2},V^{\pi_1})_{\rho}=\oplus_{\alpha,\gamma}$ with $\alpha\in \pi_1(t_{\kappa}(\rho),t_{\kappa}(\rho))$ and $\gamma\in \pi_2(s_{\kappa}(\rho),s_{\kappa}(\rho))$, graphically it is described by
\begin{equation}
\begin{tikzcd}
a_1\arrow[rd,phantom,"(T)"]&c_1\arrow[l,"\rho"']\\
a_1\arrow[u,"V^{\pi}_{\alpha}"]&\arrow[l,"\rho"]c_1\arrow[u,"V^{\pi_2}_{\gamma}"']
\end{tikzcd}
\begin{tikzcd}
{}&&{}\arrow[ll,"\rm{direction}"]
\end{tikzcd}
\end{equation}
\end{example}

\begin{example}
Let $\overline{\Hom}^{T}(V^{\pi_1},V^{\pi_2})$ denote the $\kappa$ graded vector space $\overline{\Hom}^{T}(V^{\pi_1},V^{\pi_2})_{\rho}=\oplus_{\alpha,\gamma}(V^{\pi_1}_{\alpha},V^{\pi_2}_{\gamma})$ with $\alpha\in \pi_1(t_{\kappa}(\rho),t_{\kappa}(\rho))$ and $\beta\in \pi_2(s_{\kappa}(\rho),s_{\kappa}(\rho))$, graphically it is described by
\begin{equation}
\begin{tikzcd}
c_1\arrow[r,"\rho"]\arrow[rd,phantom,"(T)"]&a_1\\
c_2\arrow[r,"\rho"']\arrow[u,"V^{\pi_2}_{\gamma}"]&a_1\arrow[u,"V^{\pi_1}_{\alpha}"']
\end{tikzcd}
\begin{tikzcd}
{}&&{}\arrow[ll,"\rm{direction}"]
\end{tikzcd}
\end{equation}
\end{example}

The homomorphism $D\in \Hom(V^{\pi_2}\otimes_{*_{\pi_2}} V^{\kappa},V^{\kappa}\otimes_{*_{\pi_1}} V^{\pi_1})$ is called \emph{left invertible} if there exists a map $D^{-1}_l\in \Hom(V^{\kappa}\otimes_{*_{\pi_1}}V^{\pi_1},V^{\pi_2}\otimes_{*_{\pi_2}} V^{\kappa})$ such that
\begin{equation}
D^{-1}_lD=\id_{V^{\pi_2}\otimes V^{\kappa}}
\end{equation} 
We can also define the \emph{right invertibility} and \emph{invertibility} similar to \eqref{eq:right_invert} and \eqref{eq:C_invertible} by just interchanging the role of $\pi_1$ and $\pi_2$.

\begin{example}\label{exa:tranposed_C}
For a homomorphism $C\in \Hom(V^{\pi_1}\otimes_{*\pi_1}V^{\pi},V^{\pi}\otimes_{*_{\pi_2}}V^{\pi_2})$, let each component of $V^{\pi}$, $V^{\pi_2}$, $V^{\pi_1}$ be one dimension. We further assume that $V^{\pi_1}_{\alpha}\simeq V^{\pi_1}_{\alpha^{-1}},V^{\pi_2}_{\gamma}\simeq V^{\pi_2}_{\gamma^{-1}}$ and set $V^{\pi^T}_{\rho}=V^{\pi}_{\rho^{-1}}$, we can define the homomorphism $C^{T}\in \Hom(V^{\pi_2}\otimes_{*_{\pi_2}}V^{\pi^T},V^{\pi^T}\otimes_{*_{\pi_1}}V^{\pi_1})$ as in \eqref{eq:def_transposed} by specifying the coefficients of the homomorphism with respect to these one dimension basis.
\begin{equation}\label{eq:def_transposed}
\text{coefficent of}~
\begin{tikzcd}
a_1\arrow[rd,phantom,"(T)"]&c_1\arrow[l,"V^{\pi^T}_{\rho_1}"']\\
a_2\arrow[u,"V^{\pi_1}_{\alpha^{-1}}"]&c_2\arrow[l,"V^{\pi^T}_{\rho_2}"]\arrow[u,"V^{\pi_2}_{\gamma^{-1}}"']\\
{}&{}\arrow[l,"\text{direction}"']
\end{tikzcd}
:=
\text{coefficent of}~
\begin{tikzcd}
a_1\arrow[r,"V^{\pi}_{\rho_2^{-1}}"]\arrow[rd,phantom,"C"]\arrow[d,"V^{\pi}_{\alpha}"']&c_1\arrow[d,"V^{\pi_2}_{\gamma}"]\\
a_2\arrow[r,"V^{\pi}_{\rho^{-1}_1}"']&c_2\\
{}\arrow[r,"\text{direction}"]&{}
\end{tikzcd}
\end{equation}

\begin{lemma}\label{lem:inversion}
We have the following relation 
\begin{equation}\label{eq:RTCTCT}
	\check{R}^T_1(z)^{(23)}\big(C^T\big)^{(12)}\big(C^T\big)^{(23)}=\big(C^T\big)^{(12)}\big(C^T\big)^{(23)}\check{R}^T_2(z)^{(12)}	
\end{equation}
\end{lemma}

\begin{proof}
	Inverse the direction of the equation \ref{eq:RCC}.
\end{proof}

\begin{lemma}
Suppose that $C$ is left invertible with $C^{-1}_l$, then $C^{T}$ is right invertible with $\big(C^{T}\big)_{r}=(C^{-1}_l)^T$.
\begin{equation}
C^{T}(C^{-1}_l)^T=\id_{V^{\pi^T}\otimes_{*_{\pi_1}}V^{\pi_1}}
\end{equation}
and similar for the right invertibility and invertibility.
\end{lemma}

\end{example}

\begin{definition}
	Suppose that $\check{R}_1(z)\in\operatorname{End}(V^{\pi_1}\otimes_{\pi_1} V^{\pi_1})$ and $\check{R}_2(z)\in \operatorname{End}(V^{\pi_2}\otimes_{\pi_2} V^{\pi_2})$ are two Yang--Baxter operators on the space of $V^{\pi_1}$ and $V^{\pi_2}$, a homomorphism $D\in \Hom(V^{\pi_2}\otimes_{*_{\pi_2}} V^{\kappa},V^{\kappa}\otimes_{*_{\pi_1}} V^{\pi_1})$ is called an transposed $(\check{R_1},\check{R_2})$ intertwiner if it satisfies the following $RDD$ relation
	\begin{equation}\label{eq:RDD}
	\check{R}_1(z)^{(23)}D^{(12)}D^{(23)}=D^{(12)}D^{(23)}\check{R}_2(z)^{(12)}	
	\end{equation} 
in the space $\Hom(V^{\pi_2}\otimes V^{\pi_2}\otimes V^{\kappa},V^{\kappa}\otimes V^{\pi_1}\otimes V^{\pi_1})$.
\end{definition}

\begin{prop}\label{prp:transposed_transfer}
Suppose that $D$ is an $(\check{R}_1,\check{R}_2)$ intertwiner and is invertible, then we have the following relation in the space $\Gamma(\kappa,\overline{\Hom}(V^{\pi_2},V^{\pi_1}))$
\begin{equation}\label{eq:RD_DR}
\overline{\tr}_{V^{\kappa}}D*_{\pi_1}\overline{\tr}_{V^{\pi_1}}\check{R}_1(z)=\overline{tr}_{V^{\pi_2}}\check{R}_2(z)*_{\pi_2}\overline{\tr}_{V^{\kappa}}D
\end{equation}
\end{prop}

Suppose that $\check{R}^T_1(z)=\check{R}_1(z)$ and also $\check{R}^T_2(z)=\check{R}_2(z)$, then compare the relation \eqref{eq:RTCTCT} and \eqref{eq:RD_DR}, we know that $C^T$ is a transposed $(\check{R}_1,\check{R}_2)$ intertwiner. We call an Yang--Baxter operator symmetric if 
\begin{equation}\label{eq:symmetric_R}
\check{R}^T(z)=\check{R}(z)
\end{equation}
and from the inversion relation of $\check{R}$ matrix \eqref{eq:inversion_R}, we know that $\check{R}^T(z)=\check{R}^{-1}(-z)$

From the composition proposition \ref{prop:composition} and the inversion relation \ref{lem:inversion} lemma for symmetric $\check{R}$ matrix, then we have the following equivalence relation.

\begin{prop}[Symmetric equivalence]\label{prop:equivalence}
The existence of intertwiner is an equivenlce relation among symmetric $\check{R}$ matrices.
\end{prop}

\begin{remark}
We can also define the complex conjugate transpose similar to \eqref{eq:def_transposed}, then also define similar notion as \eqref{eq:symmetric_R}, which also forms th equivalence relation as in \eqref{prop:equivalence}. 
\end{remark}

And also from proposition \ref{prp:transfer_relation} and proposition \ref{prp:transposed_transfer}, we have

\begin{prop}[Transfer matrix equivalence]\label{prop:transfer_matrix_equavalence}
The existence of an invertible intertwiner is an equivalence relation between symmetric $\check{R}$ matrix and also the equation \eqref{eq:RC_CR} holds among equivalence classes.
\end{prop}

\subsection{Intertwiner with spectral parameters}\label{sub:intertwiner_spectral}
More generally, we can define the following intertwiner with spectral parameter.

\begin{definition}
	Suppose that $\check{R}_1(z)\in\operatorname{End}(V^{\pi_1}\otimes_{\pi_1} V^{\pi_1})$ and $\check{R}_2(z)\in \operatorname{End}(V^{\pi_2}\otimes_{\pi_2} V^{\pi_2})$ are two Yang--Baxter operators on the space of $V^{\pi_1}$ and $V^{\pi_2}$, a map $C(z)\in \operatorname{Hom}(V^{\pi_1}\otimes_{*\pi_1}V^{\pi}, V^{\pi}\otimes_{*\pi_2} V^{\pi_2})$ is called an $(\check{R_1},\check{R_2})$ intertwiner if it satisfies the following RCC relation
	\begin{equation}\label{eq:RCC_p}
		C(w)^{(12)}C(z)^{(23)}\check{R}_1(z-w)^{(12)}=\check{R}^{(23)}_2(z-w)C^{(12)}(z)C^{(23)}(w)
	\end{equation}
	in the space $\operatorname{Hom}(V^{\pi_1}\otimes_{\pi_1} V^{\pi_1}\otimes_{*\pi_1}V^{\pi}, V^{\pi}\otimes_{*\pi_2}V^{\pi_2}\otimes_{\pi_2} V^{\pi_2})$.
\end{definition}

\begin{prop}
	Suppose that $C(z)$ is an $(\check{R_1},\check{R_2})$ intertwiner and there exists $C^{-1}(z)$ satisfies the relation
\begin{subequations}
\begin{equation}
	C(z)C^{-1}(z)=\operatorname{id}_{V^{\pi}\otimes V^{\pi_2}}\in \operatorname{Hom}(V^{\pi}\otimes_{*\pi_2} V^{\pi_2},V^{\pi}\otimes_{*\pi_2} V^{\pi_2})
\end{equation}
\begin{equation}
C^{-1}(z)C(z)=\operatorname{id}_{V^{\pi_1}\otimes V^{\pi}}\in \operatorname{Hom}(V^{\pi_1}\otimes_{*\pi_1}V^{\pi},V^{\pi_1}\otimes_{*\pi_1}V^{\pi})
\end{equation}
\end{subequations}
then we have the following relation \eqref{eq:RC_CR} in the space $\Gamma(\pi,\overline{\Hom}(V^{\pi_1},V^{\pi_2}))$
	\begin{equation}\label{eq:s_RC_CR}
		\overline{\operatorname{tr}}_{V^{\pi_1}}\check{R}_1(z)*_{\pi_1}\overline{\operatorname{tr}}_{V^{\pi}}C(z')=\overline{\operatorname{tr}}_{V^{\pi}}C(z')*_{\pi_2}\overline{\operatorname{tr}}_{V^{\pi_2}}\check{R}_2(z)
	\end{equation}.
\end{prop}

\begin{cor}[\cite{Felder2020}]
	Suppose that we have dynamical $\check{R}_1$ matrix on the vector space $V^{\pi_1}$, then set $\pi=\pi_2=\pi_1$ and also $V^{\pi}=V^{\pi_2}=V^{\pi_1}$, then we can take the intertwiner $C:=\check{R}_1$, the RCC relation \eqref{eq:RCC} will become the Yang-Baxter equation, then we get the commuting of the "transfer matrix",
	\begin{align}
		\overline{\tr}_{V^{\pi_1}}\check{R}(z)*\overline{\tr}_{V^{\pi_1}}\check{R}(z')=\overline{\tr}_{V^{\pi_1}}\check{R}(z')*\overline{\tr}_{V^{\pi_1}}\check{R}(z)
	\end{align}
\end{cor}

The constructions in the previous sections are similar defined for the intertwiner with spectral parameters, just to be careful for the difference relation of the spectral parameters.

\subsection{Example: eight vertex-SOS intertwiner}\label{sub:8v_sos_example}
The eight vertex-SOS intertwiner was discovered by Baxter in \cite{Baxter1973} and we use the parametrizations in the section 2.3 of \cite{Date1987}.

Let $H(z)$ and $\Theta(z)$ denote the Jacobian elliptic theta funcations with the half periods $K$ and $iK'$, they depends on the elliptic nome $p=e^{-\pi K'/K}$ as well as the parameter $z$.
\begin{subequations}
\begin{equation}
H(z)=2\sum_{n=1}^{\infty}(-1)^{n-1}p^{(n-1/2)^2}\sin [(2n-1)\pi z/(2K)]
\end{equation}
\begin{equation}
\Theta(z)=1+2\sum_{n=1}^{\infty}(-1)^np^{n^2}\cos(nz\pi/K)
\end{equation}
\end{subequations}

We define the function $\phi(p)=\prod_{k=1}^{\infty}(1-p^k)$,
and let $h(z):=\zeta H(\lambda(z))\Theta(\lambda z)$ where $\zeta=p^{-1/8}\phi(p)/\phi(p^2)^2$, $\lambda$ is a free parameter.
	
Let $\pi_1$ denote the groupoid with a single vertex $v$ and 2 edges $+$ and $-$ and set $V^{\pi_1}$ be the groupoid graded vector space with each component one dimension, $V^{\pi_1}_{+}=\C v_{+}$ and $V^{\pi_1}_{-}=\C v_{-}$.
\begin{equation}\label{eq:vertex_sos_graph_pi_1}
\pi_1:\quad
\begin{tikzcd}
\bullet_{v}\arrow[loop right,"+"]\arrow[loop left,"-"]
\end{tikzcd}
\end{equation}

The source fiber of $V^{\pi_1}\otimes_{\pi_1} V^{\pi_1}$ is the space $\oplus_{\alpha}\in s^{-1}(a)(V^{\pi_1}\otimes V^{\pi_1})_{\alpha}$ and it is isomorphic to the four dimensional vector space $\C^{2}\otimes \C^2$ under the identification
\begin{equation}
v_{+}\to \begin{pmatrix}
1\\
0
\end{pmatrix},\quad v_{-}\to \begin{pmatrix}
0\\
1
\end{pmatrix}.
\end{equation}
In this case, the Yang-Baxter operator $\check{R}_1$ restricted to the source fiber can be written as a 4$\times$ 4 matrix. In other words the dynamical Yang--Baxter operator can be written as
\begin{equation}
\begin{split}
\check{R}^{8v}(z)=\frac{\Theta(\lambda)\Theta(\lambda z) H(\lambda(z+1))}{\Theta(0)H(\lambda)\Theta(\lambda)}
\begin{pmatrix}
1&&&\\
&\frac{H(\lambda)H(\lambda z)}{\Theta(\lambda)\Theta(\lambda z)}&\frac{H(\lambda z)\Theta\big(\lambda(z+1)\big)}{\Theta(\lambda z) H\big(\lambda(z+1)\big)}&\\
&\frac{H(\lambda z)\Theta\big(\lambda(z+1)\big)}{\Theta(\lambda z) H\big(\lambda(z+1)\big)}&\frac{H(\lambda u)\Theta\big(\lambda(z+1)\big)}{\Theta(\lambda z) H\big(\lambda(z+1)\big)}&\\
&&&1\\
\end{pmatrix}
\end{split}
\end{equation}

\begin{lemma}
	$\check{R}^{\text{8v}}$ is symmetric as defined in \eqref{eq:symmetric_R}.
\end{lemma}

\begin{remark}
This is also called zero field condition by Baxter.
\end{remark}

\begin{remark}
	The transposed intertwiner defined in \ref{sub:transposed} is not the same as transposition of matrix. In the case of $\check{R}^{\text{8v}}$, they coincide simply because of $\check{R}^{\text{8v}}$ is special, in the sense that $\check{R}^{\text{8v}}_{11}=\check{R}^{\text{8v}}_{44}=1$. 
\end{remark}

Let $\pi_2$ be the groupoid that has the vertices $\Z+\xi,\xi\in \C$ and the arrows are $a\to a-1$ and $a\to a+1$ for any $a\in \Z+\xi,\xi\in \C$. And we choose $V^{\pi_2}$ to be the groupoid graded vector space with each component 1 dimension, that is $V^{\pi_2}_{a,a-1}=\C v_{a,-}$ and  $V^{\pi_2}_{a,a+1}=\C v_{a,+}$. And graphically it is simply \eqref{eq:vertex_sos_graph_pi_2}
\begin{equation}\label{eq:vertex_sos_graph_pi_2}
\pi_2:
\begin{tikzcd}[scale cd=0.8]
&\dots\arrow[r,bend right,"+"]&(3+\xi)\arrow[l, bend right,"-"']\arrow[r,bend right,"+"]&(4+\xi)\arrow[r,bend right,"+"]\arrow[l, bend right,"-"']&(5+\xi)\arrow[l,bend right,"-"']\arrow[r,bend right,"+"]&\dots\arrow[l,bend right,"-"']
\end{tikzcd}
\end{equation}

The source fibers of $V^{\pi_2}\otimes_{\pi_2} V^{\pi_2}$ is the space $\oplus_{\alpha\in s^{-1}(a)}(V^{\pi_2}\otimes V^{\pi_2})_{\alpha}$ and it is isomorphic to the four dimensional vector space $\C^{2}\otimes \C^2$ under the identification
\begin{equation}\label{eq:identi_source_fiber}
	v_{a,+}\to \begin{pmatrix}
		1\\
		0
	\end{pmatrix},\quad v_{a,-}\to \begin{pmatrix}
		0\\
		1
	\end{pmatrix}.
\end{equation}
In this case, the Yang-Baxter operator $\check{R}^{sos}(z)$ restricted to the source fiber can be written as a $4\times4$ matrix. In other words the dynamical Yang--Baxter operator $\check{R}^{sos}(z,a)$ can be written as
\begin{equation}
\check{R}^{\text{sos}}(z,a)=\frac{h(1)}{h(z+1)}
\begin{pmatrix}
1&&&\\
&\frac{h(a-z)h(1)}{h(a)h(z+1)}&\frac{h(a+1)h(z)}{h(a)h(z+1)}&\\
&\frac{h(a-1)h(z)}{h(a)h(z+1)}&\frac{h(a+z)h(1)}{h(a)h(z+1)}&\\
&&&1\\
\end{pmatrix}
\end{equation}

The connecting groupoid $\pi$ are the arrows from  the vertex $v$ of $\pi_1$ to the object $\Z+\xi$ of $\pi_2$, the vector space $V^{\pi}$ is 1 dimension for each component, $V^{\pi}_{\beta}=\C v_{\beta}$.
\begin{equation}
\begin{tikzcd}
\bullet_{v}\arrow[r,"\beta_{a+\xi}"']&(a+\xi)
\end{tikzcd}.
\end{equation}

We now assume that $\xi=\frac{s^{+}+s^{-}}{2}-\frac{K}{\lambda}$, as each component of the vector spaces are one dimensional, the intertwiner $C\in \Hom(V^{\pi_1}\otimes_{*_{\pi_1}}V^{\pi},V^{\pi}\otimes_{*_{\pi_2}}V^{\pi_2})$ can be described by the following coefficients with respect to the chosen basis elements.
\begin{equation}
\begin{tikzcd}
\bullet\arrow[d,"+"']\arrow[rd,phantom,"C"]\arrow[r]&a\arrow[d]\\
\bullet\arrow[r]&a+1
\end{tikzcd}
=H\big(\lambda(s^{+}+a-z-\xi)\big),\quad
\begin{tikzcd}
	\bullet\arrow[d,"-"']\arrow[rd,phantom,"C"]\arrow[r]&a\arrow[d]\\
	\bullet\arrow[r]&a+1
\end{tikzcd}
=\Theta\big(\lambda(s^{+}+a-z-\xi)\big)
\end{equation}

\begin{equation}
	\begin{tikzcd}
		\bullet\arrow[d,"+"']\arrow[rd,phantom,"C"]\arrow[r]&a\arrow[d]\\
		\bullet\arrow[r]&a-1
	\end{tikzcd}
	=H\big(\lambda(s^{-}+a+z-\xi)\big),\quad
	\begin{tikzcd}
		\bullet\arrow[d,"-"']\arrow[rd,phantom,"C"]\arrow[r]&a\arrow[d]\\
		\bullet\arrow[r]&a-1
	\end{tikzcd}
	=\Theta(\lambda(s^{-}+a+z-\xi))
\end{equation}

\begin{theorem}[Baxter \cite{Baxter1973,Date1987}]
	$C$ is an $(\check{R}^{8v},\check{R}^{sos})$ intertwiner.
\end{theorem}

We now introduce another intertwiner which relate $\check{R}^{\text{sos}}$ to a symmetric Yang-Baxter operator $\check{R}^{\text{sym-sos}}$.

Let $\pi_3=\pi_2$ and the connecting groupoid $\hat{\pi}$ from $\pi_2$ to $\pi_3$ are the identity arrows, in order to distinguish the 2 same groupoid, we use $\circ$ to denote the object in $\pi_3$ and $\bullet$ to denote the object in $\pi_2$.
\begin{equation}\label{eq:vertex_sos_hat_pi}
\begin{tikzcd}
\bullet_{a}\arrow[r,"\beta_a"]&\circ_a
\end{tikzcd}
,\forall a\in \Z+\xi
\end{equation}

Let $V^{\pi_3}=V^{\pi_2}$ and $V^{\hat{\pi}}_{\beta_a}=\C e_a$ for any $a$. By identifying the source fibers of $V^{\pi_3}$ as \eqref{eq:identi_source_fiber}. The $\check{R}^{\text{sym-sos}}$ can be written as the following.
\begin{equation}
	\check{R}^{\text{sym-sos}}(z,a)=\frac{h(1)}{h(z+1)}
	\begin{pmatrix}
		1&&&\\
		&\frac{h(a-z)h(1)}{h(a)h(z+1)}&\frac{\sqrt{h(a+1)h(a-1)}h(z)}{h(a)h(z+1)}&\\
		&\frac{\sqrt{h(a-1)h(a+1)}h(z)}{h(a)h(z+1)}&\frac{h(a+z)h(1)}{h(a)h(z+1)}&\\
		&&&1\\
	\end{pmatrix}
\end{equation}

\begin{lemma}
$\check{R}^{\text{sym-sos}}$ is symmetric as defined in \eqref{eq:symmetric_R}.
\end{lemma}
\begin{remark}
The transposed intertwiner defined in \ref{sub:transposed} is not the same as transposition of matrix. In the case of $\check{R}^{\text{sym-sos}}$, they coincide simply because of $\check{R}^{\text{sym-sos}}$ is special, in the sense that $\check{R}^{\text{sym-sos}}_{11}=\check{R}^{\text{sym-sos}}_{44}=1$. 
\end{remark}

The homomorphism $\widehat{C}\in \Hom(V^{\pi_2}\otimes_{*_{\pi_2}}V^{\tilde{\pi}},V^{\tilde{\pi}}\otimes_{*_{\pi_2}}V^{\otimes \pi_3})$ can be described by the following coefficients with respect to the chosen basis of the one dimensional vector space of each component.
\begin{equation}
\begin{tikzcd}[scale cd=0.8]
\bullet_a\arrow[d]\arrow[r]\arrow[rd,phantom,"\widehat{C}"]&\circ_a\arrow[d]\\
\bullet_{(a-1)}\arrow[r]&\circ_{(a-1)}
\end{tikzcd}
=\frac{1}{\big(h(a)h(a-1)\big)^{1/4}},\quad
\begin{tikzcd}[scale cd=0.8]
	\bullet_a\arrow[d]\arrow[r]\arrow[rd,phantom,"\widehat{C}"]&\circ_a\arrow[d]\\
	\bullet_{(a+1)}\arrow[r]&\circ_{(a+1)}
\end{tikzcd}
=\frac{1}{\big(h(a)h(a+1)\big)^{1/4}}
\end{equation}

\begin{lemma}
$\widehat{C}$ is an $\check{R}^{sos}$ and $\check{R}^{\text{sym-sos}}$ intertwiner.
\end{lemma}

\begin{prop}
$\check{R}^{8v}$ and $\check{R}^{\text{sym-sos}}$ are symmetric equivalent as defined in proposition \ref{prop:equivalence} with the intertwiner $\widehat{C}^{(23)}C^{(12)}$ and also equivalent under the relation transfer matrix equivalence proposition \ref{prop:transfer_matrix_equavalence}.
\end{prop}

\section{Twist}\label{sec:Twist}
In the subsection \ref{sub:int_bet_YBE}, when we take partial traces of homomorphisms $\Hom(V^{\pi_1}\otimes V^{\pi},V^{\pi}\otimes V^{\pi_2})$, it gives rise to modules of convolution algebras. More generally we can also take "matrix element" of various homomorphisms, this will gives us the various transfer operators defined below. We are mainly interested in the matrix element of the $\Hom(V^{\pi_1}\otimes V^{\pi},V^{\pi}\otimes V^{\pi_2})$ and $\Hom(V^{\pi_1}\otimes V^{\pi},V^{\pi_1}\otimes V^{\pi})$, so the corresponding transfer operator will mainly be two types, we call them transfer operator of shape $\square$ and $\rhd$.

\subsection{Transfer operator}\label{sub:transfer_operator}
\begin{definition}
A transfer operator of shape $\square$ from $V^{\pi_1}$ to $V^{\pi_2}$ are graded by two arrows of $V^{\pi}$ which denotes the initial degree and final degree. We denote the set of transfer operators by $\Gamma(\pi\times\pi,\overline{\operatorname{Hom}}(V^{\pi_1}, V^{\pi_2}))$,
it consists of maps $l:\pi\times \pi\to \overline{\operatorname{Hom}}(V^{\pi_1},V^{\pi_2})$ such that
\begin{itemize}
	\item [(1)] $l(\beta_1,\beta_2)\in \overline{\operatorname{Hom}}(V^{\pi_1},V^{\pi_2})_{\beta_1,\beta_2}$\\
	\item [(2)] for every $\beta_1,\beta_2$, the number of $\alpha,\gamma$ are finite.\\
\end{itemize}
and graphically it is simply
\begin{equation}
	l(\beta_1,\beta_2)\in \oplus_{\alpha,\gamma}
\begin{tikzcd}
	a_1\arrow[r,"\beta_2"]\arrow[d,"V^{\pi_1}_{\alpha}"']&c_1\arrow[d,"V^{\pi_2}_{\gamma}"]\\
	a_1\arrow[r,"\beta_1"]&c_1
\end{tikzcd}
\end{equation}

The transfer operators have the following fusion product, for any $f,g\in \Gamma(\pi\times\pi,\overline{\operatorname{Hom}}(V^{\pi_1}, V^{\pi_2}))$, $f*_{\pi_1}g\in \Gamma\big(\pi\times\pi,\overline{\operatorname{Hom}}(V^{\pi_1}\otimes_{\pi_1}V^{\pi_1}, V^{\pi_2}\otimes_{\pi_2}V^{\pi_2})\big)$ is defined by
\begin{equation}\label{eq:fusion_transfer12}
(f*_{\otimes}g)(\beta_1,\beta_2)=\sum_{\beta}f(\beta,\beta_2)\otimes_{\pi_1} g(\beta_1,\beta)\in \overline{\operatorname{Hom}}(V^{\pi_1}\otimes_{\pi_1}V^{\pi_1}, V^{\pi_2}\otimes_{\pi_2}V^{\pi_2})) 
\end{equation}
and graphically it is the 
\begin{equation}
\sum_{\beta,s_{\pi}(\beta)=a',\gamma_1,\gamma_2}
\begin{tikzcd}
	a\arrow[r,"\beta_2"]\arrow[d, "V^{\pi_1}_{\alpha_1}"']\arrow[rd,phantom,"f"]&c\arrow[d,"V^{\pi_2}_{\gamma_1}"]\\
	a'\arrow[r,"\beta"]\arrow[d,"V^{\pi_1}_{\alpha_2}"']\arrow[rd,phantom,"g"]&c'\arrow[d,"V^{\pi_2}_{\gamma_2}"]\\
	a''\arrow[r,"\beta_1"]&c''\\
\end{tikzcd}
\end{equation}

The transfer operators also carry action of the group $\operatorname{End}(V^{\pi_2})$, for an operator $Q\in \operatorname{End}(V^{\pi_2})$, the action is defined by composition with the image of each component of $l\in \Gamma\big(\pi\times\pi,\operatorname{Hom}(V^{\pi_1}, V^{\pi_2})\big)$, 
\[
Ql:=Ql(\beta_1,\beta_2)
\]

\end{definition}
Similarly we can define the transfer operator $\Gamma\big(\pi\times\pi, \overline{\operatorname{Hom}}(V^{\pi_2},V^{\pi_1})\big)$, it consists of maps $l:\pi\times\pi \to \overline{\operatorname{Hom}}(V^{\pi_2},V^{\pi_1})$ which graphically is
\begin{equation}
l(\beta_1,\beta_2)\in \oplus_{\alpha,\gamma}
\begin{tikzcd}
c\arrow[d,"V^{\pi_2}_{\gamma}"']&a\arrow[l,"\beta_1"']\arrow[d,"V^{\pi_1}_{\alpha}"]\\
c'&a'\arrow[l,"\beta_2"]
\end{tikzcd}
\end{equation}
and the transfer operators have the following fusion product
\begin{align*}
(f*g)(\beta_1,\beta_2)=\sum_{\beta}f(\beta_1,\beta)\otimes_{\pi_2}g(\beta,\beta_2)\in\overline{\operatorname{Hom}}(V^{\pi_2}\otimes_{\pi_2}V^{\pi_2}, V^{\pi_1}\otimes_{\pi_1}V^{\pi_1})
\end{align*}
graphically it is simply
\begin{equation}
\sum_{\beta,t_{\pi}(\beta)=c',\alpha_1,\alpha_2}
\begin{tikzcd}
	c\arrow[d,"V^{\pi_2}_{\gamma_1}"']&a\arrow[ld,phantom,"f"]\arrow[l,"\beta_1"']\arrow[d,"V^{\pi_1}_{\alpha_1}"]\\
	c'\arrow[d,"V^{\pi_2}_{\gamma_2}"']&a'\arrow[ld,phantom,"g"]\arrow[l,"\beta"]\arrow[d,"V^{\pi_1}_{\alpha_2}"]\\
	c''&a''\arrow[l,"\beta_2"]
\end{tikzcd}
\end{equation}
and notice the difference of orders with \eqref{eq:fusion_transfer12}   there is also the action of $\End(V^{\pi_1})$.

\begin{example}[Verma type graded vector space]
A large source of examples of the transfer operator of type $\square$ come from taking the matrix elements of homomorphism between groupoid vector space, one kind of these graded vector spaces that are most relevant to later considerations are the following.

A $\pi$ graded vector space is of \emph{Verma} type if for any $\beta\in \pi$, there is a one dimensional distinguished sub vector space generated by vector $v_{\beta}$ for each $V^{\pi}_{\beta}$, then we can define the linear form
\begin{equation*}
\begin{aligned}
&v^{*}_{\beta}(v)=1,\quad \rm{if}\quad v=v_{\beta}\\
&v^{*}_{\beta}(v)=0,\quad \rm{if}\quad v\ne v_{\beta}\\	
\end{aligned}
\end{equation*}

Assume that $V^{\pi}$ is of \emph{Verma type}, then for a map $J\in \operatorname{Hom}(V^{\pi_1}\otimes V^{\pi},V^{\pi}\otimes V^{\pi_2})$, we can define its matrix element $\operatorname{Mat}_{V^{\pi}}J\in \Gamma\big(\pi\times \pi,\overline{\operatorname{Hom}}(V^{\pi_2},V^{\pi_1})\big)$ with respect to the distinguished basis element as follows
\begin{equation}
\operatorname{Mat}_{V^{\pi}}J: \beta_1\times \beta_2 \mapsto \oplus_{\alpha,\gamma}(v^{*}_{\beta_2}\otimes 1)J(1\otimes v_{\beta_1})\in \overline{\operatorname{Hom}}(V^{\pi_1}, V^{\pi_2})_{\beta_1,\beta_2}
\end{equation}
\end{example}

%

Now we introduce another type of graded vector space, let $V^{\pi_1}\in \operatorname{Vect}_k(\pi_1)$, we use $\operatorname{End}_{\rhd}(V^{\pi_1})$ to denote the collection of $\pi\times\pi$ graded vector spaces, for each $\beta_1,\beta_2$ with $t_{\pi}(\beta_1)=t_{\pi}(\beta_2)$, $\operatorname{End}_{\rhd}(V^{\pi_1})_{\beta_1,\beta_2}=\oplus_{\alpha_1,\alpha_2,\beta_1\circ\alpha_1=\beta_2\circ\alpha_2}\Hom({V^{\pi_1}_{\alpha_1}},V^{\pi_1}_{\alpha_2})$, and
graphically it is simply
\begin{equation}
\begin{tikzcd}
	&a\arrow[ld,"V^{\pi_1}_{\alpha_1}"']\arrow[rd,"V^{\pi_1}_{\alpha_2}"]&\\
	a'\arrow[rd,"\beta_1"']&&a''\arrow[ld,"\beta_2"]\\
	&c&\\
\end{tikzcd}
\quad \rm{or} \quad
\begin{tikzcd}
	a\arrow[d,"V^{\pi_1}_{\alpha_1}"']\arrow[r,"V^{\pi_1}_{\alpha_2}"]&a''\arrow[d,"\beta_2"]\\
	a'\arrow[r,"\beta_1"']&c\\
\end{tikzcd}
\end{equation}

\begin{definition}
Transfer operators of shape $\rhd$ are maps from $\pi\times \pi\to \operatorname{End}_{\rhd}(V^{\pi_1})$ with finite condition that for each $a\in \operatorname{\pi_1}$, there are finitely many arrows with source points $a$, we denote it by $\Psi(\pi\times \pi,\operatorname{End}_{\rhd}(V^{\pi_1}))$. 

For any two maps $l_1,l_2\in \Psi(\pi\times \pi,\operatorname{End}_{\rhd}(V^{\pi_1}))$, the product $l_1*_{\rhd}l_2\in \Psi(\pi\times \pi,\operatorname{End}_{\rhd}(V^{\pi_1}))$ is defined by
\begin{equation}
l_1*_{\rhd}l_2(\beta_1,\beta_2)=\sum_{\beta}l_1(\beta,\beta_2)l_2(\beta_1,\beta)
\end{equation}

And similarly we can define the space $\Psi(\pi\times \pi,\End_{\rhd}(V^{\pi_2}))$.
\end{definition}

\begin{example}
	When we consider the homomorphism $S\in \Hom(V^{\pi_1}\otimes V^{\pi},V^{\pi_1}\otimes V^{\pi})$, assume that $V^{\pi}$ is of Verma type, then the matrix element $\rm{Mat}_{V^{\pi}}S\in \Psi(\pi\times\pi,\End(V^{\pi_1}))$ of $S$ with respect to the distinguished basis are
	\begin{equation}
		\rm{Mat}_{V^\pi}S:\beta_1\times \beta_2\mapsto \oplus_{\alpha_1,\alpha_2}(v^{*}_{\beta_2}\otimes 1)S(v_{\beta_1}\otimes 1)\in \End_{\rhd}(V^{\pi_1})_{\beta_1,\beta_2}
	\end{equation}
\end{example}

\begin{definition}[Connecting system for $\pi_2$]\label{def:connecting_system_2}
A connecting system for $\pi_2$ is a set $\Sigma_{\pi_2}$ of arrows in $\pi$, such that for	
each path $\gamma\in \pi_2$, there exists unique $\beta_{\gamma}\in \Sigma_{\pi_2},\beta_{\gamma}\in \pi$ with the condition $t_{\pi}(\beta_{\gamma})=s_{\pi_2}(\gamma)$, they satisfy the compatible condition
\begin{equation}
\beta_{\gamma\circ \gamma'}=\beta_{\gamma'}.
\end{equation}
And for each $\beta\in \Sigma_{\pi_2}$, there exist at least one $\gamma\in \pi_2$ such that $\beta=\beta_{\gamma}$.

The connecting system is called \emph{unique} if
\begin{equation}\label{eq:unique_connecting_system}
s^{-1}_{\pi}(s_{\pi}(\beta_{\gamma}))=\beta_{\gamma},\forall \gamma
\end{equation} 
which means that there is only one unique edge starting from the point $s_{\pi}(\beta_{\gamma})$.
\end{definition}

\begin{definition}[Connecting system for $\pi_1$]\label{def:connecting_system_1}
A connecting system for $\pi_1$ is a set $\Sigma_{\pi_1}$ of arrows in $\pi$, such that for each path $\alpha\in \pi_1$, there exists a unique $\beta^{\alpha}\in \pi$ with $s_{\pi}(\beta)=t_{\pi_1}(\alpha)$ and satisfies the compatible condition
\begin{equation}
\beta^{\alpha\circ\alpha'}=\beta^{\alpha'}.
\end{equation}
And for each $\beta\in \Sigma_{\pi_1}$, there exist at least one $\alpha$ such that $\beta=\beta^\alpha$.

The connecting system is called \emph{unique} if $t_{\pi}^{-1}(t_{\pi}(\beta))=\beta$, that is there is only a unique edge with target the point $t_{\pi}(\beta)$.
\end{definition}

\begin{lemma}\label{lemma_endo}
For each connecting system $\Sigma_{\pi_1}$, there exist an algebra map $\rhd$ from $\End(V^{\pi_1})$ to $\Psi(\pi\times \pi,\operatorname{End}_{\rhd}(V^{\pi_1}))$,
\begin{equation}
\rhd:\operatorname{End}(V^{\pi_1})\rightarrow\Psi(\pi\times \pi,\operatorname{End}_{\rhd}(V^{\pi_1})), l\mapsto l_{\rhd}
\end{equation}
And also the inverse direction algebra map
\begin{equation}
\lhd:\Psi(\pi\times \pi,\operatorname{End}_{\rhd}(V^{\pi_1}))\rightarrow \operatorname{End}(V^{\pi_1}), f\mapsto f_{\lhd}
\end{equation}
\end{lemma}

\begin{proof}
For each operator $S\in \operatorname{End}(V^{\pi_1})$, with the choice of $\Sigma_{\pi_1}$, it induces a map $S_{\rhd} \in\Psi(\pi\times \pi,\operatorname{End}_{\rhd}(V^{\pi_1}))$, for each $\beta\in \Sigma_{\pi_1}$,
\begin{equation}
S_{\rhd}(\beta,\beta)=\oplus_{\alpha,t_{\pi_1}(\alpha)=s_{\pi}(\beta)}S_{\alpha}
\end{equation}
other components of $S$ are set to be 0.

And similarly, with the fixed choice of $\beta$, for each map $l\in \Psi(\pi\times \pi,\operatorname{End}(V^{\pi_1}))$, the component $l(\beta,\beta)$ gives an operator in $\operatorname{End}(V^{\pi_1})$, this map is an algebra morphism.
\end{proof}

With the above preparation, we now can introduce the notion of invertibility for the transfer operator, which are parallel to the invertibility introduced in \eqref{eq:left_invert} and \eqref{eq:right_invert}. The difference is that instead of the vector spaces graded $\pi$ as in \eqref{eq:left_invert} and \eqref{eq:right_invert}, $\pi$ here are purely gradings and no vector spaces involved.

\begin{definition}
For $f\in \Gamma(\pi\times\pi,\overline{\operatorname{Hom}}(V^{\pi_1}, V^{\pi_2}))$, it is \emph{left invertible} if there exists $f^{-1}\in \Gamma(\pi\times\pi,\overline{\operatorname{Hom}}(V^{\pi_2}, V^{\pi_1}))$ such that 
\begin{equation}\label{eq:transfer_left_invert}
f^{-1}*_{\circ}f(\beta_1,\beta_2):=\sum_{\beta}f^{-1}(\beta,\beta_2)f(\beta_1,\beta)=\delta_{\beta_1,\beta_2}\operatorname{id}_{f}(\beta_1,\beta_1)\in \End_{\rhd}(V^{\pi_1})
\end{equation}
here $\operatorname{id}_{f}\in \Psi(\pi\times \pi,\operatorname{End}_{\rhd}(V^{\pi_1}))$, and each of its component $\id_{f}(\beta_1,\beta_1)$ is the identity map on $V^{\pi_1}_{\alpha}$ for every $V^{\pi_1}_\alpha$ that $f$ transfers along $\beta_1$.

And it is \emph{right invertible} if there exists $f^{-1}\in \Gamma(\pi\times\pi,\overline{\operatorname{Hom}}(V^{\pi_2}, V^{\pi_1}))$ such that
\begin{equation}\label{eq:transfer_right_invert}
f*_{\circ}f^{-1}(\beta_1,\beta_2):=\sum_{\beta}f^{-1}(\beta,\beta_2)f(\beta_1,\beta)=\delta_{\beta_1,\beta_2}\operatorname{id}_{f^{-1}}(\beta_1,\beta_1)\in \End_{\rhd}(V^{\pi_2})
\end{equation}
here $\id_{f^{-1}}\in\Psi\big(\pi\times \pi,\operatorname{End}_{\rhd}(V^{\pi_2})\big)$, and each of its component $\id_{f^{-1}}(\beta_1,\beta_1)$ is identity map on $V^{\pi_2}_{\gamma}$ for any $V^{\pi_2}_{\gamma}$ that $f^{-1}$ transfers along $\beta_1$.

$f$ is called invertible if it satisfies both \eqref{eq:transfer_left_invert} and \eqref{eq:transfer_right_invert} with the same $f^{-1}$.
\end{definition}

Similar definition for the maps in $\Gamma(\pi\times\pi,\overline{\operatorname{Hom}}(V^{\pi_2}, V^{\pi_1}))$ and also its tensor product.

\begin{lemma}\label{lem:factorization_invertible}
Suppose that $f_1,f_2\in \Gamma\big(\pi\times\pi,\overline{\operatorname{Hom}}(V^{\pi_1}, V^{\pi_2})\big)$ is right invertible with inverse maps $f_1^{-1}, f_2^{-1}$, then $f_1*_{\otimes}f_2\in \Gamma(\pi\times\pi,\overline{\operatorname{Hom}}(V^{\pi_1}\otimes V^{\pi_1},V^{\pi_2}\otimes V^{\pi_2})$ is right invertible with inverse $f^{-1}_1*_{\otimes}f^{-1}_{2}$.
\end{lemma}

\begin{proof}
We can directly check the definitions
\begin{equation}
\begin{aligned}
&(f_1*_{\otimes} f_2)*_{\circ}(f^{-1}_1*_{\otimes}f^{-1}_2)(\beta_1,\beta_2)\\
&=\sum_{\beta}(f_1*_{\otimes} f_2)(\beta,\beta_2)(f^{-1}_1*_{\otimes}f^{-1}_2)(\beta_1,\beta)\\
&=\sum_{\beta',\beta,\beta''}(f_1(\beta',\beta_2)\otimes_{\pi_1} f_2(\beta,\beta')))(f^{-1}_1(\beta_1,\beta'')\otimes_{\pi_2} f^{-1}_2(\beta'',\beta))\\
&=\sum_{\beta',\beta''}(f_1(\beta',\beta_2)f^{-1}_1(\beta_1,\beta''))\otimes_{\pi_2} \sum_{\beta}f_2(\beta,\beta')f^{-1}_2(\beta'',\beta)\\
&=\sum_{\beta',\beta''}(f_1(\beta',\beta_2)f^{-1}_1(\beta_1,\beta''))\otimes_{\pi_2}\delta_{\beta',\beta''}\operatorname{id}_{f^{-1}_2}(\beta',\beta')\\
&=\delta_{\beta_1,\beta_2}\operatorname{id}_{f^{-1}_1\otimes f^{-1}_2}(\beta_1,\beta_1)	
\end{aligned}
\end{equation}	

And in a similar way, we can prove the situation for the left inverse	$(f^{-1}_1*_{\otimes}f^{-1}_2)*_{\circ}(f_1*_{\otimes} f_2)$ and also the inverse.
\end{proof}

\subsection{Unique connecting system}\label{sub:unique_twist}
\hfill\\
We fix a unique connecting system $\Sigma_{\pi_2}$ for $\pi_2$. Suppose that we have an Yang--Baxter operator $\check{R}_1$ on $V^{\pi_1}$, given a right invertible map $j\in \Gamma(\pi\times \pi,\overline{\operatorname{Hom}}(V^{\pi_2}\otimes V^{\pi_2}, V^{\pi_1}\otimes V^{\pi_1}))$, we can define the operator $\check{R}_2\in \Psi\big(\pi\times\pi,\operatorname{End}_{\rhd}(V^{\pi_2}\otimes V^{\pi_2})\big)$, for any $\beta_1\in \Sigma_{\pi_2}$,
\begin{equation}
\check{R}_2(\beta_1,\beta_1)=j^{-1}*_{\circ}(\check{R_1}j):=\sum_{\beta}j^{-1}(\beta,\beta_1)\check{R}_1j(\beta_1,\beta)
\end{equation}
and graphically it is simply
\begin{equation}\label{eq:graph_twist}
	\sum_{\color{red}\rm{red}}
	\begin{tikzcd}
		&\circ\arrow[ld,"V^{\pi_2}_{\gamma_1}"']&\bullet\arrow[l,"\beta_1"']\arrow[ld,red]\arrow[rd,red]\arrow[r,"\beta_1"]&\circ\arrow[rd,"V^{\pi_2}_{\gamma_3}"]&\\
		\circ\arrow[rd,"V^{\pi_2}_{\gamma_2}"']&{\color{red}\bullet}\arrow[l,phantom,"j"]\arrow[rd,red]&\check{R}_1(z)&{\color{red}\bullet}\arrow[r,phantom,"j^{-1}"]\arrow[ld,red]&\circ\arrow[ld,"V^{\pi_2}_{\gamma_4}"]\\
		&\circ&{\color{red}\bullet}\arrow[l,red,"\beta''"]\arrow[r,red,"\beta''"']&\circ&
	\end{tikzcd}
	=
	\begin{tikzcd}
		\bullet\arrow[r,"\beta_1"]&\circ\arrow[ld,"V^{\pi_2}_{\gamma_1}"']\arrow[rd,"V^{\pi_2}_{\gamma_3}"]&\bullet\arrow[l,"\beta_1"']\\
		\circ\arrow[rd,"V^{\pi_2}_{\gamma_2}"']&\check{R}_2(z)&\circ\arrow[ld,"V^{\pi_2}_{\gamma_4}"]\\
		&\circ&
	\end{tikzcd}
\end{equation}

By the lemma \ref{lemma_endo} above , the map $\check{R}_2$ induce a map $(\check{R}_2)_{\lhd}\in \operatorname{End}(V^{\pi_2}\otimes V^{\pi_2})$.

The map $j$ is called a twist if there exists a right invertible map $q\in \Gamma(\pi \times \pi,\overline{\operatorname{Hom}}(V^{\pi_2}\otimes V^{\pi_2}\otimes V^{\pi_2}, V^{\pi_1}\otimes V^{\pi_1}\otimes V^{\pi_1}))$ such that
\begin{subequations}
\begin{equation}\label{eq:uqtwist23}
\operatorname{id}_{V^{\pi_2}}\otimes (\check{R}_2)_{\lhd}=(q^{-1}*_{\circ}\big((\operatorname{id}_{V^{\pi_1}}\otimes \check{R}_1)q\big))_{\lhd}
\end{equation}
\begin{equation}\label{eq:uqtwist12}
(\check{R}_2)_{\lhd}\otimes \operatorname{id}_{V^{\pi_2}}=(q^{-1}*_{\circ}\big((\check{R}_1\otimes \operatorname{id}_{V^{\pi_1}})q\big))_{\lhd}
\end{equation}
\end{subequations}
and graphically \eqref{eq:uqtwist12} is just
\begin{equation}\label{eq:graph_twist23}
	\sum_{\color{red}\rm{red}}
	\begin{tikzcd}
		&\circ\arrow[d,"V^{\pi_2}_{\gamma_1}"]&\bullet\arrow[l,"\beta_1"']\arrow[d,red]\arrow[r,"\beta_1"]&\circ\arrow[d,"V^{\pi_2}_{\gamma_1}"]&\\
		&\circ\arrow[ld,"V^{\pi_2}_{\gamma_2}"']&{\color{red}\bullet}\arrow[ld,red]\arrow[rd,red]&\circ\arrow[rd,"V^{\pi_2}_{\gamma_5}"]&\\
		\circ\arrow[rd,"V^{\pi_2}_{\gamma_3}"']&{\color{red}\bullet}\arrow[l,phantom,"q"]\arrow[rd,red]&\check{R}_1(z)&{\color{red}\bullet}\arrow[r,phantom,"q^{-1}"]\arrow[ld,red]&\circ\arrow[ld,"V^{\pi_2}_{\gamma_6}"]\\
		&\circ&{\color{red}\bullet}\arrow[l,red,"\beta''"]\arrow[r,red,"\beta''"']&\circ&
	\end{tikzcd}
	=
	\begin{tikzcd}
		&\circ\arrow[d,"V^{\pi_2}_{\gamma_1}"]&\\
		\bullet\arrow[r,"\beta"]&\circ\arrow[ld,"V^{\pi_2}_{\gamma_2}"']\arrow[rd,"V^{\pi_2}_{\gamma_5}"]&\bullet\arrow[l,"\beta"']\\
		\circ\arrow[rd,"V^{\pi_2}_{\gamma_3}"']&\check{R}_2(z)&\circ\arrow[ld,"V^{\pi_2}_{\gamma_6}"]\\
		&\circ&
	\end{tikzcd}
\end{equation}
and graphically \eqref{eq:uqtwist23} is just

\begin{equation}\label{eq:graph_twist12}
	\sum_{\color{red}\rm{red}}
	\begin{tikzcd}
		&\circ\arrow[ld,"V^{\pi_2}_{\gamma_1}"']&\bullet\arrow[l,"\beta_1"']\arrow[ld,red]\arrow[rd,red]\arrow[r,"\beta_1"]&\circ\arrow[rd,"V^{\pi_2}_{\gamma_4}"]&\\
		\circ\arrow[rd,"V^{\pi_2}_{\gamma_2}"']&{\color{red}\bullet}\arrow[l,phantom,"q"]\arrow[rd,red]&\check{R}_1(z)&{\color{red}\bullet}\arrow[r,phantom,"q^{-1}"]\arrow[ld,red]&\circ\arrow[ld,"V^{\pi_2}_{\gamma_5}"]\\
		&\circ\arrow[d,"V^{\pi_2}_{\gamma_3}"']&{\color{red}\bullet}\arrow[d,red]&\circ\arrow[d,"V^{\pi_2}_{\gamma_3}"]&\\
		&\circ&{\color{red}\bullet}\arrow[l,red,"\beta''"]\arrow[r,red,"\beta''"']&\circ&\\
	\end{tikzcd}
	=
	\begin{tikzcd}
		\bullet\arrow[r,"\beta_1"]&\circ\arrow[ld,"V^{\pi_2}_{\gamma_1}"']\arrow[rd,"V^{\pi_2}_{\gamma_4}"]&\bullet\arrow[l,"\beta_1"']\\
		\circ\arrow[rd,"V^{\pi_2}_{\gamma_2}"']&\check{R}_2(z)&\circ\arrow[ld,"V^{\pi_2}_{\gamma_5}"]\\
		&\circ\arrow[d,"V^{\pi_2}_{\gamma_3}"]&\\
		&\circ&\\
	\end{tikzcd}
\end{equation}

\begin{prop}
Suppose that the connecting system for $\pi_2$ is unique, if $j$ is a twist, then $(\check{R}_2)_{\lhd}$ is an Yang-Baxter operator on $V^{\pi_2}$.
\end{prop}

\begin{proof}
For each $\gamma\in \pi_2$, because the connecting system is unique, there exists a unique path $\beta_1$ for $\gamma$ with $t(\beta_1)=s(\gamma)$ and $s^{-1}_{\pi}(s_{\pi}(\beta_1))=\beta_1$.	
	
\begin{equation*}
\begin{aligned}
&\operatorname{id}_{V^{\pi_2}}\otimes (\check{R}_2)_{\lhd}|_{(V^{\pi_2}\otimes V^{\pi_2}\otimes V^{\pi_2})_{\gamma}}\\
&=\big(q^{-1}*_{\circ}\big((\operatorname{id}_{V^{\pi_1}}\otimes\check{R}_1 )q\big)(\beta_1,\beta_1)=\sum_{\beta}q^{-1}(\beta,\beta_1)(\operatorname{id}_{V^{\pi_1}}\otimes\check{R}_1)q(\beta_1,\beta)
\end{aligned}
\end{equation*}

\begin{equation*}
	\begin{aligned}
		&(\check{R}_2)_{\lhd}\otimes\operatorname{id}_{V^{\pi_2}} |_{(V^{\pi_2}\otimes V^{\pi_2}\otimes V^{\pi_2})_{\gamma}}\\
	 &=\big(q^{-1}*_{\circ}\big((\check{R}_1\otimes \operatorname{id}_{V^{\pi_1}})q\big)(\beta_1,\beta_1)=\sum_{\beta}q^{-1}(\beta,\beta_1)(\check{R}_1\otimes\operatorname{id}_{V^{\pi_1}})q(\beta_1,\beta)
	\end{aligned}
\end{equation*}
then we have that
\begin{equation*}
\begin{aligned}
&\big((\check{R}_2)_{\lhd}\otimes\operatorname{id}_{V^{\pi_2}}\big)\big(\operatorname{id}_{V^{\pi_2}}\otimes (\check{R}_2)_{\lhd}\big)|_{(V^{\pi_2}\otimes V^{\pi_2}\otimes V^{\pi_2})_{\gamma}}\\
&=\big(\sum_{\beta'}q^{-1}(\beta',\beta_1)(\check{R}_1\otimes\operatorname{id})q(\beta_1,\beta')\big)\big(\sum_{\beta}q^{-1}(\beta,\beta_1)(\operatorname{id}\otimes\check{R}_1)q(\beta_1,\beta)\big)
\end{aligned}
\end{equation*}

By the right invertibility of the $q$ and uniqueness of the connecting system, we have that
\begin{equation}
q(\beta_1,\beta')q^{-1}(\beta,\beta_1)=\delta_{\beta,\beta'}q*_{\circ}q^{-1}(\beta,\beta_1)=\delta_{\beta,\beta'}\id_{q^{-1}}(\beta,\beta)
\end{equation} 
so we have that
\begin{equation*}
	\begin{aligned}
		&\big((\check{R}_2)_{\lhd}\otimes\operatorname{id}_{V^{\pi_2}}\big)\big(\operatorname{id}_{V^{\pi_2}}\otimes (\check{R}_2)_{\lhd}\big)|_{(V^{\pi_2}\otimes V^{\pi_2}\otimes V^{\pi_2})_{\gamma}}\\
		&=\big(\sum_{\beta}q^{-1}(\beta,\beta_1)(\check{R}_1\otimes\operatorname{id})(\operatorname{id}\otimes\check{R}_1)q(\beta_1,\beta)\big)
	\end{aligned}
\end{equation*}
then we know that it satisfies the Yang-Baxter equation.
\end{proof}

\subsection{Quasi-unique connecting system}\label{sub:quasi_unique}
In the last section, when we talk about twist, we assume that the connecting system is unique, this condition is too restrictive and eliminate the possibility of the symmetry.
\begin{definition}[Quasi-unique connecting system]\label{def:qunique}
Given a connecting system $\Sigma_{\pi_2}$,for each path $\gamma\in \pi_2$, we have a connecting path $\beta'\in \Sigma_{\pi_2}$. $\Sigma_{\pi_2}$ is called quasi-unique, if there exists a $\gamma'\in \pi_2$ with $t(\gamma')=s(\gamma)$ and there exist an edge $\beta_1\in \Sigma_{\pi_2}$ with $t_{\pi}(\beta_1)=s(\gamma')$ and the edge $\beta_1$ satisfies the property \eqref{eq:quasi_unique}, graphically it is \eqref{eq:graph_quasi_uni}.
\begin{equation}\label{eq:quasi_unique}
s^{-1}_{\pi}(s_{\pi}(\beta_1))=\beta_1
\end{equation}
\begin{equation}\label{eq:graph_quasi_uni}
\begin{tikzcd}
	{}&a\arrow[l,"\beta_1"',"\text{unique}"]\arrow[r,"\beta_1"]& s(\gamma')\arrow[d,"\gamma'"]\\
	&a'\arrow[r,"\beta'"]&s(\gamma)\arrow[d,"\gamma"]\\
	&&t(\gamma)
\end{tikzcd}
\end{equation}
\end{definition}

Suppose we have an Yang-Baxter operator $\check{R}_1$ on $V^{\pi_1}$, given an right invertible map $j\in \Gamma(\pi\times \pi,\overline{\operatorname{Hom}}(V^{\pi_2}\otimes V^{\pi_2}, V^{\pi_1}\otimes V^{\pi_1}))$, we can define the operator $\check{R}_2\in \Psi(\pi\times\pi,\operatorname{End}_{\rhd}(V^{\pi_2}\otimes V^{\pi_2}))$, for any $\beta_1,\beta_2\in \pi$
\begin{equation}
\check{R}_2(\beta_2,\beta_1):=j^{-1}*_{\circ}(\check{R}_1j)=\sum_{\beta}j^{-1}(\beta,\beta_2)\check{R}_1j(\beta_1,\beta)
\end{equation}
graphically it is similar as \eqref{eq:graph_twist}.

The map $j$ is called a twist for this quasi-unique connecting system, if for any length 3 path $\gamma$, there exists a right invertible map $q_2\in \Gamma(\pi \times \pi,\overline{\operatorname{Hom}}(V^{\pi_2}\otimes V^{\pi_2}\otimes V^{\pi_2}, V^{\pi_1}\otimes V^{\pi_1}\otimes V^{\pi_1}))$, such that for any $\beta_2$ with $t_{\pi}(\beta_2)=s_{\pi_2}(\gamma)$ and any $s_{\pi}(\beta'_2)=s_{\pi}(\beta_2)$, we have the following weight zero or ice rule conditions \eqref{eq:ice_rule}.
\begin{subequations}\label{eq:ice_rule}
\begin{equation}\label{eq:quntwist23}
	\operatorname{id}_{V^{\pi_2}}\otimes (\check{R}_2)_{\lhd}\delta_{\beta_2,\beta'_2}=\big(q_2^{-1}*_{\circ}\big((\operatorname{id}_{V^{\pi_1}}\otimes \check{R}_1)q_2\big)\big)(\beta_2,\beta'_2)
\end{equation}
\begin{equation}\label{eq:quntwist12}
	(\check{R}_2)_{\lhd}\otimes \operatorname{id}_{V^{\pi_2}}\delta_{\beta_2,\beta'_2}=\big(q_2^{-1}*_{\circ}\big((\check{R}_1\otimes \operatorname{id}_{V^{\pi_1}})q_2\big)\big)(\beta_2,\beta'_2)
\end{equation}
\end{subequations}
And graphically the equation \eqref{eq:quntwist12} is equivalent to the following two relations
\begin{subequations}\label{eq:graph_uqtwist}
\begin{equation}\label{eq:graph_uqtwist12}
	\sum_{\color{red}\rm{red}}
	\begin{tikzcd}[scale cd=0.8]
		&\circ\arrow[ld,"V^{\pi_2}_{\gamma_1}"']&\bullet\arrow[l,"\beta_2"']\arrow[ld,red]\arrow[rd,red]\arrow[r,"\beta_2'"]&\circ\arrow[rd,"V^{\pi_2}_{\gamma_4}"]&\\
		\circ\arrow[rd,"V^{\pi_2}_{\gamma_2}"']&{\color{red}\bullet}\arrow[l,phantom,"q_2"]\arrow[rd,red]&\check{R}_1(z)&{\color{red}\bullet}\arrow[r,phantom,"q_2^{-1}"]\arrow[ld,red]&\circ\arrow[ld,"V^{\pi_2}_{\gamma_5}"]\\
		&\circ\arrow[d,"V^{\pi_2}_{\gamma_3}"']&{\color{red}\bullet}\arrow[d,red]&\circ\arrow[d,"V^{\pi_2}_{\gamma_6}"]&\\
		&\circ&{\color{red}\bullet}\arrow[l,red,"\beta''"]\arrow[r,red,"\beta''"']&\circ&\\
	\end{tikzcd}
	=\delta_{\beta_2,\beta'_2}
	\begin{tikzcd}[scale cd=0.8]
		\bullet\arrow[r,"\beta_2"]&\circ\arrow[ld,"V^{\pi_2}_{\gamma_1}"']\arrow[rd,"V^{\pi_2}_{\gamma_4}"]&\bullet\arrow[l,"\beta_2"']\\
		\circ\arrow[rd,"V^{\pi_2}_{\gamma_2}"']&\check{R}_2(z)&\circ\arrow[ld,"V^{\pi_2}_{\gamma_5}"]\\
		&\circ\arrow[d,"V^{\pi_2}_{\gamma_3}"]&\\
		&\circ&\\
	\end{tikzcd}
\end{equation}
\begin{equation}\label{eq:graph_uqtwist12_2}
	\sum_{\color{red}\rm{red}}
	\begin{tikzcd}[scale cd=0.8]
		&\circ\arrow[ld,"V^{\pi_2}_{\gamma_1}"']&\bullet\arrow[l,"\beta_2"']\arrow[ld,red]\arrow[rd,red]\arrow[r,"\beta_2"]&\circ\arrow[rd,"V^{\pi_2}_{\gamma_4}"]&\\
		\circ\arrow[rd,"V^{\pi_2}_{\gamma_2}"']&{\color{red}\bullet}\arrow[l,phantom,"q_2"]\arrow[rd,red]&\check{R}_1(z)&{\color{red}\bullet}\arrow[r,phantom,"q_2^{-1}"]\arrow[ld,red]&\circ\arrow[ld,"V^{\pi_2}_{\gamma_5}"]\\
		&\circ\arrow[d,"V^{\pi_2}_{\gamma_3}"']&{\color{red}\bullet}\arrow[d,red]&\circ\arrow[d,"V^{\pi_2}_{\gamma_6}"]&\\
		&\circ&{\color{red}\bullet}\arrow[l,red,"\beta''"]\arrow[r,red,"\beta''"']&\circ&\\
	\end{tikzcd}
	=\sum_{\color{red}\rm{red}}\begin{tikzcd}[scale cd=0.8]
		&\circ\arrow[ld,"V^{\pi_2}_{\gamma_1}"']&\bullet\arrow[l,"\beta_3"']\arrow[ld,red]\arrow[rd,red]\arrow[r,"\beta_3"]&\circ\arrow[rd,"V^{\pi_2}_{\gamma_4}"]&\\
		\circ\arrow[rd,"V^{\pi_2}_{\gamma_2}"']&{\color{red}\bullet}\arrow[l,phantom,"q_2"]\arrow[rd,red]&\check{R}_1(z)&{\color{red}\bullet}\arrow[r,phantom,"q_2^{-1}"]\arrow[ld,red]&\circ\arrow[ld,"V^{\pi_2}_{\gamma_5}"]\\
		&\circ\arrow[d,"V^{\pi_2}_{\gamma_3}"']&{\color{red}\bullet}\arrow[d,red]&\circ\arrow[d,"V^{\pi_2}_{\gamma_6}"]&\\
		&\circ&{\color{red}\bullet}\arrow[l,red,"\beta''"]\arrow[r,red,"\beta''"']&\circ&\\
	\end{tikzcd}
\end{equation}
\end{subequations}

And there exists a right invertible map $q\in \Gamma(\pi\times \pi,\overline{\operatorname{Hom}}((V^{\pi_2})^{\otimes n}, (V^{\pi_1})^{\otimes n} ))$ satisfies the following conditions
\begin{subequations}
\begin{equation}\label{eq:qutwist_q23}
\operatorname{id}^{\otimes n-2}_{V^{\pi_2}}\otimes (\check{R}_2)_{\lhd}|_{(V^{\pi_2})^{\otimes n}_{\gamma\gamma'}}=(q^{-1}*_{\circ}\big((\operatorname{id}^{\otimes n-2}_{V^{\pi_1}}\otimes \check{R}_1)q\big))_{\lhd}
\end{equation}
\begin{equation}\label{eq:qutwist_q12}
\id^{\otimes n-3}_{V^{\pi_2}}\otimes (\check{R}_2)_{\lhd}\otimes \operatorname{id}_{V^{\pi_2}}|_{(V^{\pi_2})^{\otimes n}_{\gamma\gamma'}}=(q^{-1}*_{\circ}\big((\id^{\otimes n-3}_{V^{\pi_1}}\otimes \check{R}_1\otimes \operatorname{id}_{V^{\pi_1}})q\big))_{\lhd}
\end{equation}
\end{subequations}
and graphically the equation \eqref{eq:qutwist_q23} is the following \eqref{eq:graph_qutwist_q23} and similarly for \eqref{eq:qutwist_q12}.
\begin{equation}\label{eq:graph_qutwist_q23}
	\sum_{\color{red}\rm{red}}
	\begin{tikzcd}
    	&\circ\arrow[d,"V^{\pi_2}_{\gamma'}"]&\bullet\arrow[l,"\beta_1"']\arrow[d,red]\arrow[r,"\beta_1"]&\circ\arrow[d,"V^{\pi_2}_{\gamma'}"]&\\
		&\circ\arrow[d,"V^{\pi_2}_{\gamma_1}"]&{\color{red}\bullet}\arrow[d,red]&\circ\arrow[d,"V^{\pi_2}_{\gamma_1}"]&\\
		&\circ\arrow[ld,"V^{\pi_2}_{\gamma_2}"']&{\color{red}\bullet}\arrow[ld,red]\arrow[rd,red]&\circ\arrow[rd,"V^{\pi_2}_{\gamma_5}"]&\\
		\circ\arrow[rd,"V^{\pi_2}_{\gamma_3}"']&{\color{red}\bullet}\arrow[l,phantom,"q"]\arrow[rd,red]&\check{R}_1(z)&{\color{red}\bullet}\arrow[r,phantom,"q^{-1}"]\arrow[ld,red]&\circ\arrow[ld,"V^{\pi_2}_{\gamma_6}"]\\
		&\circ&{\color{red}\bullet}\arrow[l,red,"\beta''"]\arrow[r,red,"\beta''"']&\circ&
	\end{tikzcd}
	=
	\begin{tikzcd}
		&\circ\arrow[d,"V^{\pi_2}_{\gamma'}"]&\\
		&\circ\arrow[d,"V^{\pi_2}_{\gamma_1}"]&\\
		\bullet\arrow[r,"\beta"]&\circ\arrow[ld,"V^{\pi_2}_{\gamma_2}"']\arrow[rd,"V^{\pi_2}_{\gamma_5}"]&\bullet\arrow[l,"\beta"']\\
		\circ\arrow[rd,"V^{\pi_2}_{\gamma_3}"']&\check{R}_2(z)&\circ\arrow[ld,"V^{\pi_2}_{\gamma_6}"]\\
		&\circ&
	\end{tikzcd}
\end{equation}

And there exists an invertible map $q_1\in  \Gamma\big(\pi\times \pi,\overline{\operatorname{Hom}}((V^{\pi_2})^{\otimes n-3}, (V^{\pi_1})^{\otimes n-3} )\big)$ such that the factorization property holds \eqref{twist_factorization}.
\begin{equation}\label{twist_factorization}
q|_{(V^{\pi_2})^{\otimes n}_{\gamma\gamma'}}=q_1|_{(V^{\pi_2})^{\otimes n-3}_{\gamma'}}*_{\otimes}q_2|_{(V^{\pi_2})^{\otimes 3}_{\gamma}}
\end{equation}

\begin{prop}
Suppose that the connecting system for $\pi_2$ is quasi-unique and there is no multi-edge in $\pi$, if $j$ is a generalized twist, then $(\check{R}_2)_{\rhd}$ is an Yang-Baxter operator on $V^{\pi_2}$.
\end{prop}

\begin{proof}
From the factorization condition \eqref{twist_factorization} and the lemma \ref{lem:factorization_invertible}, we have the relation that
\begin{align*}
&\big(q^{-1}*_{\circ}(\operatorname{id}^{\otimes n-2}_{V^{\pi_1}}\otimes \check{R}_1)q\big)(\beta_1,\beta_1)\\
&=(q^{-1}_1*_{\otimes}q^{-1}_2)*_{\circ}\big((\operatorname{id}^{\otimes n-2}_{V^{\pi_1}}\otimes \check{R}_1)q_1*_{\otimes}q_2\big)(\beta_1,\beta_1)\\
&=\sum_{\beta}q^{-1}_1*_{\otimes}q^{-1}_2(\beta,\beta_1)(\operatorname{id}^{\otimes n-2}_{V^{\pi_1}}\otimes \check{R}_1)q_1*_{\otimes}q_2(\beta_1,\beta)\\
&=\sum_{\beta,\beta_2,\beta'_2}\big(q^{-1}_1(\beta'_2,\beta_1)\otimes q^{-1}_2(\beta,\beta'_2)\big)(\operatorname{id}^{\otimes n-2}_{V^{\pi_1}}\otimes \check{R}_1)(q_1(\beta_1,\beta_2)\otimes q_2(\beta_2,\beta))
\end{align*}

Then from the condition \eqref{eq:quntwist12} and the \eqref{eq:quntwist23}, we know that
\begin{equation}
\big(q^{-1}*_{\circ}(\operatorname{id}^{\otimes n-2}_{V^{\pi_1}}\otimes \check{R}_1)q\big)(\beta_1,\beta_1)=\id_{q_1}(\beta_1,\beta_1)\otimes \id_{V^{\pi_2}}\otimes (\check{R_2})_{\lhd}
\end{equation}
and similarly we get the equation
\begin{equation}
q^{-1}*_{\circ}\big((\id^{\otimes n-3}_{V^{\pi_1}}\otimes \check{R}_1\otimes \operatorname{id}^{\otimes n-2}_{V^{\pi_1}})q\big)(\beta_1,\beta_1)=\id_{q_1}(\beta_1,\beta_1)\otimes (\check{R_2})_{\lhd}\otimes \id_{V^{\pi_2}}
\end{equation}

Then we multiply these two operators
\begin{align*}
&\big(q^{-1}*_{\circ}(\operatorname{id}^{\otimes n-2}_{V^{\pi_1}}\otimes \check{R}_1)q\big)(\beta_1,\beta_1)\big(q^{-1}*_{\circ}\big((\id^{\otimes n-3}_{V^{\pi_1}}\otimes \check{R}_1\otimes \operatorname{id}^{\otimes n-2}_{V^{\pi_1}})q\big)(\beta_1,\beta_1)\\
&=\sum_{\beta,\beta'}\big(q^{-1}(\beta,\beta_1)(\operatorname{id}^{\otimes n-2}_{V^{\pi_1}}\otimes \check{R}_1)q(\beta_1,\beta)\big)\big(q^{-1}(\beta',\beta_1)(\id^{\otimes n-3}_{V^{\pi_1}}\otimes \check{R}_1\otimes \operatorname{id}^{\otimes n-2}_{V^{\pi_1}})q(\beta_1,\beta')\big)
\end{align*}
And then from the property \eqref{eq:quasi_unique} and the right invertibility of $q^{-1}_1$, we have that
\begin{equation*}
q(\beta_1,\beta)q^{-1}(\beta',\beta_1)=\delta_{\beta',\beta}\id_{q^{-1}}
\end{equation*}
we get the relation
\begin{equation}\label{eq:qutwist_proof_1}
\begin{split}
&\big(q^{-1}*_{\circ}(\operatorname{id}^{\otimes n-2}_{V^{\pi_1}}\otimes \check{R}_1)q\big)(\beta_1,\beta_1)\big(q^{-1}*_{\circ}\big((\id^{\otimes n-3}_{V^{\pi_1}}\otimes \check{R}_1\otimes \operatorname{id}^{\otimes n-2}_{V^{\pi_1}})\big)q\big)(\beta_1,\beta_1)\\
&=\big(q^{-1}*_{\circ}\big(\operatorname{id}^{\otimes n-3}_{V^{\pi_1}}\otimes (\id_{V^{\pi_1}}\otimes \check{R}_1)(\check{R}_1\otimes \operatorname{id}_{V^{\pi_1}})q\big)(\beta_1,\beta_1)\\
&=\id_{q_1}(\beta_1,\beta_1)\otimes( \id_{V^{\pi_2}}\otimes (\check{R_2})_{\lhd})((\check{R_2})_{\lhd}\otimes \id_{V^{\pi_2}})
\end{split}
\end{equation}

And then from the equation \eqref{eq:qutwist_proof_1} and the fact that $\check{R_1}$ is an Yang-Baxter operator, we know that $(\check{R}_2)_{\lhd}$ satisfies the Yang-Baxter equation
\begin{equation}
\begin{split}
&(\id_{V^{\pi_2}}\otimes (\check{R_2})_{\lhd})((\check{R_2})_{\lhd}\otimes \id_{V^{\pi_2}})(\id_{V^{\pi_2}}\otimes (\check{R_2})_{\lhd})\\
&=((\check{R_2})_{\lhd}\otimes \id_{V^{\pi_2}})(\id_{V^{\pi_2}}\otimes (\check{R_2})_{\lhd})((\check{R_2})_{\lhd}\otimes \id_{V^{\pi_2}})
\end{split}
\end{equation}
\end{proof}

\subsection{Intertwiner from twists}\label{sub:intertwiner_from_twist}
Suppose now that we have a \emph{quasi-unique connecting system} for $\pi_2$ and an invertible $\widetilde{C}\in \Gamma\big(\pi\times \pi,\overline{\operatorname{Hom}}(V^{\pi_1},V^{\pi_2})\big)$ with its inverse $\widetilde{C}^{-1}\in \Gamma\big(\pi\times \pi,\overline{\operatorname{Hom}}(V^{\pi_2},V^{\pi_1})\big)$.

Let $j=\widetilde{C}^{-1}*_{\otimes}\widetilde{C}^{-1}\in \Gamma(\pi\times \pi,\overline{\operatorname{Hom}}(V^{\pi_2}\otimes V^{\pi_2},V^{\pi_1}\otimes V^{\pi_1}))$, from the lemma \ref{lem:factorization_invertible}, it has inverse with $j^{-1}=\widetilde{C}*_{\otimes}\widetilde{C}$. With the map $j$ and a given quasi-unique connecting system, we can define $\check{R}_2\in \Psi\big(\pi\times\pi,\operatorname{End}_{\rhd}(V^{\pi_2}\otimes V^{\pi_2})\big)$ as before
\begin{equation}\label{cell_twist_R_2}
	\check{R}_2(\beta_2,\beta_1)=j^{-1}*_{\circ}(\check{R}_1j):=\sum_{\beta}j^{-1}(\beta,\beta_2)\check{R}_1j(\beta_1,\beta)
\end{equation}

\begin{definition}
	$\widetilde{C}$ is called a cell-twist if for any $\gamma\in \pi_2$, we have the following relations 
	\begin{subequations}\label{eq:cell_twist}
		\begin{equation}
			\sum_{\beta}j^{-1}(\beta,\beta_1)\check{R}_1j(\beta_1,\beta)=\sum_{\beta'}j^{-1}(\beta',\beta'_1)\check{R}_1j(\beta'_1,\beta')
		\end{equation}
		\begin{equation}
			\sum_{\beta}j^{-1}(\beta,\beta_2)\check{R}_1j(\beta_1,\beta)=0
		\end{equation}
	\end{subequations}
	for any $\beta_1$,$\beta'_1$ with $t_{\pi}(\beta'_1)=t_{\pi}(\beta_1)=s_{\pi_2}(\gamma)$ and for any $\beta_2\ne \beta_1$ with $s_{\pi}(\beta_2)=s_{\pi}(\beta_1)$.
\end{definition}

\begin{prop}
	Suppose that $\widetilde{C}$ is a cell-twist, then $j$ is a twist for the chosen quasi-unique connecting system.
\end{prop}

\begin{proof}
	With the given quasi-connecting system, the corresponding $q_1$ is $(n-3)$ times $*_{\otimes}$ of $\widetilde{C}^{-1}$ and $q_2=\widetilde{C}^{-1}*_{\otimes}\widetilde{C}^{-1}*_{\otimes}\widetilde{C}^{-1}$ as in the definition of quasi-unique twist and the relation \eqref{eq:quntwist12},\eqref{eq:quntwist23}, \eqref{eq:qutwist_q12} and \eqref{eq:qutwist_q23} follows from \eqref{eq:cell_twist} and the invertibility of $\widetilde{C}$.
\end{proof}

The simplest cell twist is the following gauge transform.

\begin{definition}[Gauge transform]
For each ($\pi_1,V^{\pi_1},\check{R}_1)$, we can assume that $\pi_1=\pi_2,V^{\pi_2}=V^{\pi_2}$ and $\pi$ is the arrows which connect the same vertices. 

We now define the following map $\widetilde{C}^{-1}\in \Gamma(\pi\times \pi,\overline{\operatorname{Hom}}(V^{\pi_2},V^{\pi_1}))$, for each of the component \eqref{eq:gauge_cell},
\begin{equation}\label{eq:gauge_cell}
	\begin{tikzcd}
		\circ_i\arrow[d,"V^{\pi_2}_{\alpha}"']&\bullet_i\arrow[l,"\beta"']\arrow[d,"V^{\pi_1}_\alpha"]\\
		\circ_j&\arrow[l,"\beta"]\bullet_j
	\end{tikzcd}
\end{equation}
let $c_{\alpha,\beta}$ to be nonzero complex number, then $\widetilde{C}^{-1}$ is defined to be
\begin{equation}
\begin{split}
\widetilde{C}^{-1}(\beta,\beta):V^{\pi_2}_{\alpha}&\to V_{\alpha}^{\pi_1}: v \mapsto c_{\alpha,\beta}v 
\end{split}
\end{equation}

$\widetilde{C}^{-1}$ is easily seen to be invertible with the inverse defined by the inverse of the complex numbers $c_{\alpha,\beta}$. And the condition \eqref{eq:cell_system_twist} are easily verified, because in this case, there is only one unique connecting edge for each vertex. So $j=\widetilde{C}^{-1}*_{\otimes}\widetilde{C}^{-1}$ is a twist, and the $(\check{R}_2)_{\lhd}$ defined by \eqref{cell_twist_R_2} is called the gauge transform of $\check{R}_1$.
\end{definition}

Now we can formulate the relation with the intertwiner.

\begin{prop}
For each $\beta\in \pi$, we can assign a one dimensional vector space, then it forms a vector space $V^{\pi}$, let $\widetilde{C}$ be a cell-twist, then it gives rise to an intertwiner $C\in \operatorname{Hom}(V^{\pi_1}\otimes_{*\pi_1}V^{\pi}, V^{\pi}\otimes_{*\pi_2} V^{\pi_2})$. 
\end{prop}

\begin{proof}
In the case of one dimension for each edge $\beta$, the relation \eqref{eq:graph_further_relation} and the relation \eqref{eq:cell_twist} are the same. And in this case also the RCC relation \eqref{eq:RCC} is equivalent to the relation \eqref{eq:further_relation}.
\end{proof}

\section{Example: Drinfeld twist}\label{sec:Drinfeld_twist}
We first recall the theory of twist for quasi-triangular Hopf algebra by Drinfeld in \cite{Drinfeld1987,Drinfeld1990a, Drinfeld1991}, see also the section 8.10 of \cite{Etingof2015}. Let $(H,m,\Delta,\epsilon,S)$ be a Hopf algebra, the iterated coproduct for $i\ge 1$ is
\begin{equation*}
\Delta_i:H^{\otimes n}\to H^{\otimes n+1},\quad \Delta_i=\operatorname{id}\otimes \dots\otimes \Delta\otimes \dots\operatorname{id}
\end{equation*}
with the convention $\Delta_0=1\otimes (-)$ and $\Delta_{n+1}=(-)\otimes 1$, an $n$ cochain $\chi$ is an invertible element of $H^{\otimes n}$ and we define its coboundary as the $n+1$ cochain
\begin{align*}
	\partial \chi:=(\prod_{i=0}^{i~\text{odd}}\Delta_i \chi)(\prod_{i=0}^{i~\text{even}}\Delta_i \chi^{-1})
\end{align*}
the cochain $\chi$ is an cocycle if $\partial \chi=1$,
An $n$ cocycle for a Hopf algebra or bialgebra is an invertible element $\chi$ in $H^{\otimes n}$ such that $\partial\chi=1$, finally a cochain or cocycle is counital if $\epsilon_i\chi=1$.

\begin{definition}
Let $J\in H\otimes H$ be an invertible element, then $J$ is called a Drinfeld twist if it is a 2 cocycle
\begin{equation}
(\Delta\otimes \operatorname{id})(J)(J\otimes 1)=(\operatorname{id}\otimes \Delta)(J)(1\otimes J)
\end{equation}
or equivalently $J^{12,3}J^{1,2}=J^{1,23}J^{2,3}$.
\end{definition}

\begin{prop}[\cite{Etingof2015},Exercise 5.14.2.]
Let $\Delta_J:H\to H\otimes H$ be given by
\[
\Delta_J(x):=J^{-1}\Delta(x)J
\]
then there exist $S_J:H\to H$ such that $(H,m,\Delta_J,\epsilon,S_J)$ is a Hopf algebra and is called the twist of $H$ by $J$, denoted by $H^J$.
\end{prop}

\begin{prop}[\cite{Etingof2015},Proposition 8.3.14.]
	If $(H,R)$ is a quasi-triangular Hopf algebra with twist $J$, then $(H^J,R^J)$, where $R_J=(J^{21})^{-1}RJ$ is also quasi-triangular.
\end{prop}

Now we reformulate the definition of Drinfeld twist and give a link with the twist developed in subsection \ref{sub:unique_twist}. We choose $V$ to be a representation of the quasi-triangular Hopf algebra $(H,R)$, and $P:V\otimes V\to V\otimes V$ the permutation matrix and $\check{R}:=PR$, suppose we have a twist, then we have
\begin{equation}
\begin{aligned}
\check{R}_J&=P(J^{21})^{-1}RJ=P^{12}P^{12}(J^{12})^{-1}PR^{12}J^{12}\\
&=(J^{12})^{-1}\check{R}^{(12)}J^{12}\\
\check{R}_J\otimes \operatorname{id}&=(J^{12})^{-1}\check{R}^{(12)}\Delta\otimes\operatorname{id}(J^{-1})\Delta\otimes\operatorname{id}(J)J^{12}\\
&=(J^{12})^{-1}\Delta\otimes\operatorname{id}(J^{-1})\check{R}^{(12)}\Delta\otimes\operatorname{id}(J)J^{12}\\
&=(J^{12,3}J^{12})^{-1}\check{R}^{(12)}J^{12,3}J^{12}\\
\operatorname{id}\otimes\check{R}_J&=(J^{1,23}J^{23})^{-1}\check{R}^{23}(J^{1,23}J^{23})	
\end{aligned}
\end{equation}

And notice that the cocycle equation appears in the two side of $\operatorname{id}\otimes \check{R}$ and $\check{R}\otimes\operatorname{id}$, so we can have the following equivalent way of defining Drinfeld twist
\begin{definition}\label{def:Drinfeld_twist}
[Drinfeld Twist] Suppose we have a vector space $V$ and a Yang-Baxter operator $\check{R}$ acting on $V\otimes V$, an invertible operator $J:V\otimes V\to V\otimes V$ is a twist if there exist invertible operator $Q:V\otimes V\otimes V\to V\otimes V\otimes V$ such that:
\begin{equation}
	\begin{aligned}
		&\check{R}_J=J^{-1}\check{R}J\\
		&\check{R}_J\otimes\operatorname{id}=Q(\check{R}\otimes\operatorname{id})Q^{-1}\\
		&\operatorname{id}\otimes \check{R}_J=Q(\operatorname{id}\otimes\check{R})Q^{-1}\\
	\end{aligned}
\end{equation}
\end{definition}

In this case, we can assume that the groupoid is one point $a$ with one edge $\alpha$ as follows
\[
\pi_1=
\begin{tikzcd}
	a\arrow[loop,"\alpha"']
\end{tikzcd}\quad
\pi_2=
\begin{tikzcd}
	a\arrow[loop,"\alpha"']
\end{tikzcd}\quad
\pi=
\begin{tikzcd}
a\arrow[r,"\beta"]&a
\end{tikzcd}
\]
And $V=V^{\pi_1}=V^{\pi_2}$, the operators $j\in \Gamma(\pi\times \pi,\overline{\operatorname{Hom}}(V^{\pi_1}\otimes V^{\pi_1}, V^{\pi_2}\otimes V^{\pi_2}))$ and $q\in \Gamma(\pi\times \pi,\overline{\operatorname{Hom}}(V^{\pi_2}\otimes V^{\pi_2}\otimes V^{\pi_2}, V^{\pi_1}\otimes V^{\pi_1}\otimes V^{\pi_1}))$ are defined as follows:
\begin{equation*}
\begin{aligned}
&j:\beta\times \beta \mapsto J\\	
&q:\beta\times \beta \mapsto Q\\	
\end{aligned}
\end{equation*}

Then the twist condition for the transfer operators of $j$ in subsection \ref{sub:unique_twist} will be the same as that of the operator $J$ in the definition of \ref{def:Drinfeld_twist}.

\section{Example: dynamical twist}\label{sec:dynamical_twist}
We first recall the Etingof--Varchenko construction \cite{Etingof1999} of dynamical twist in terms of fusion operator for any semisimple lie algebra and its corresponding quantum group, following the book \cite{Etingof2005}.

Let $q$ be a nonzero complex number which is not a root of unity and $\mathfrak{g}$ be a semisimple Lie algebra. And $V$ be a finite dimensional representations of the quantum group $U_q(\mathfrak{g})$, $R$ denote the $R$ matrix of the quantum group. For generic weight $\lambda\in \mathfrak{h}^{*}$, there exists universal fusion operator $J(\lambda)$ and it has the following property.

\begin{theorem}[\cite{Etingof2005},Theorem 3.8]
The universal fusion operator $J(\lambda)$ satisfies the dynamical twist equation on the space $V\otimes V\otimes V$.
\[
J^{12,3}(\lambda)J^{1,2}(\lambda-h^3)=J^{1,23}(\lambda)J^{2,3}(\lambda)
\]
Here the superscripts of $J$ stands for both the components on which the first and second components of $J$ act and the coproduct action of quantum group, for example $J^{1,23}$ means $(1\otimes \Delta)(J)$. The dynamical notation means
\[
J^{1,2}(\lambda-h^3)(a_1\otimes a_2\otimes a_3)=J^{1,2}(\lambda-\rm{weight}(a_3))(a_1\otimes a_2\otimes a_3)
\]
\end{theorem}

\begin{cor}[\cite{Etingof2005},Corallory 3.11]\label{cor_twist}
The fusion operator $J(\lambda)$ also satisfies
\begin{equation*}
	J^{32,1}(\lambda)J^{32}(\lambda-h^1)=J^{3,21}(\lambda)J^{21}(\lambda)
\end{equation*}
\end{cor}

\begin{theorem}[\cite{Etingof2005},Theorem 3.15]
The exchange operator $R_{J}(\lambda):V\otimes V\to V\otimes V$ is defined by $R_{J}(\lambda)=J^{-1}(\lambda)R^{21}J^{21}(\lambda)$, it satisfies the dynamical Yang-Baxter equation
\[
R_{J}^{12}(\lambda-h^3)R_{J}^{13}(\lambda)R_{J}^{23}(\lambda-h^1)=R^{23}_{J}(\lambda)R^{13}_{J}(\lambda-h^2)R_{J}^{12}(\lambda)
\]
\end{theorem}

We give another proof that the exchange operator $R_{J}(\lambda)$ satisfies the dynamical Yang-Baxter equation, then reformulate the definition of twist, we first have the following lemma. The $R$ matrix has another version with multiplication by a permutation matrix
\[
\check{R}_J(\lambda)=PR_J(\lambda)=(J^{21}(\lambda))^{-1}PR^{21}J^{21}
\]

\begin{lemma}
	Let $\tilde{R}=RP$, then we have
	\begin{equation}\label{eq_R}
		\tilde{R}^{12}\tilde{R}^{23}\tilde{R}^{12}=\tilde{R}^{23}\tilde{R}^{12}\tilde{R}^{23}
	\end{equation}
\end{lemma}

\begin{proof}
	Because the left hand side of \eqref{eq_R} equals $P^{13}R^{32}R^{31}R^{21}$ and the right hand side of \eqref{eq_R} equals $P^{13}R^{21}R^{31}R^{32}$ and $R$ satisfies the Yang-Baxter equation, so the equation \eqref{eq_R} follows.
\end{proof}

We also have the relations,
\begin{equation}
\begin{aligned}
\check{R}_J(\lambda)\otimes\operatorname{id}_{V}&=(J^{21}(\lambda))^{-1}P^{12}R^{21}J^{21}(\lambda)\\
&=(J^{21}(\lambda))^{-1}P^{12}R^{21}(J^{3,21}(\lambda))^{-1}J^{3,21}(\lambda)J^{21}(\lambda)\\
&=(J^{21}(\lambda))^{-1}P^{12}P^{21}(J^{3,21}(\lambda))^{-1}P^{21}R^{21}J^{3,21}(\lambda)J^{21}(\lambda)\\
&=(J^{3,21}(\lambda)J^{21}(\lambda))^{-1}R_{12}P_{12}(J^{3,21}(\lambda)J^{21}(\lambda))\\
\end{aligned}
\end{equation}

\begin{equation}
\begin{aligned}
	\operatorname{id}_{V}\otimes\check{R}_J(\lambda-h^1)&=(J^{32}(\lambda-h^1))^{-1}P^{23}R^{32}J^{32}(\lambda-h^1)\\
&=(J^{32}(\lambda-h^1))^{-1}P^{23}R^{32}(J^{32,1}(\lambda))^{-1}J^{32,1}(\lambda)J^{32}(\lambda-h^1)\\
&=(J^{32}(\lambda-h^1))^{-1}P^{23}P^{32}(J^{32,1}(\lambda))^{-1}P^{32}R^32J^{32,1}(\lambda)J^{32}(\lambda-h^1)\\
&=(J^{32,1}(\lambda)J^{32}(\lambda-h^1))R^{23}P^{23}(J^{32,1}(\lambda)J^{32}(\lambda-h^1))\\	
\end{aligned}
\end{equation}

From equation \eqref{eq_R} and the corollary \ref{cor_twist}, we know that $\check{R}(\lambda)$ satisfies the equation:
\[
\check{R}^{(23)}_J(\lambda-h^{1})\check{R}_J^{(12)}(\lambda)\check{R}_J^{(23)}(\lambda-h^1)=\check{R}_J^{(12)}(\lambda)\check{R}_J^{(23)}(\lambda-h^1)\check{R}_J^{(12)}(\lambda)
\]

From the proof above, we can reformuate the dynamical drinfeld twist as the following.

\begin{definition}
	Suppose we have a vector space $V$ and a Yang-Baxter operator $\check{R}$ acting on $V\otimes V$, an invertible operator $J(\lambda):V\otimes V\to V\otimes V$ is a twist if there exists invertible operator $Q(\lambda):V\otimes V\otimes V\to V\otimes V\otimes V$ such that
\begin{equation}
\begin{aligned}
&\check{R}_J(\lambda)=J^{-1}(\lambda)\check{R}J(\lambda)\\
&\check{R}_J(\lambda)\otimes \operatorname{id}=Q(\lambda)(\check{R}\otimes \operatorname{id})Q^{-1}(\lambda)\\
&\operatorname{id}\otimes \check{R}_J(\lambda-h^1)=Q(\lambda)(\operatorname{id}\otimes \check{R})Q^{-1}(\lambda)	
\end{aligned}
\end{equation}
then it follows automatically, $\check{R}_J$ satisfies the Yang-Baxter equation.
\end{definition}

In order to be more explicitly show the relation with the groupoids twist in the subsection \ref{sub:unique_twist}.

\subsection{The $\mathfrak{sl}_2$ case preliminary}
We consider the lie algebra $\mathfrak{sl}_2$ over $\C$, let $M_{\lambda}$ denote the Verma module over $\mathfrak{g}$ with highest weight $\lambda\in \mathfrak{g}^{*}$, $x_{\lambda}$ being its highest weight vector and $x^{*}_{\lambda}$ the lowest weight vector of the dual module. Let $V=\rm{span}\{v_{+},v_{-}\}$ be the two dimensional representation of $\mathfrak{sl}_2$, the intertwining operator $\Phi:M_{\lambda}\to M_{\mu}\otimes V$, the vector $x_{\mu}^{*}(\Phi x_{\lambda})\in V[\lambda-\mu]$ is called the expectation value of $\Phi$ and denoted by $\langle \Phi\rangle$.

\begin{lemma}
	$M_{\lambda}$ is irreducible for $\lambda\ne 0,1,2,\dots$, i.e all positive integers and 0.
\end{lemma}

\begin{lemma}
If $M_{\mu}$ is irreducible, then for all homogeneous $v\in V$, there exists a unique intertwining operator $\Phi:M_{\mu+\rm{wt}v}\to M_{\mu}\otimes V$ such that $\langle \Phi\rangle =v$, this intertwine will be denoted by $\Phi^{v}_{\mu+\rm{wt}v}$.
\end{lemma}

Let $V,W$ be finite-dimensional representations of $\mathfrak{g}$, fix a generic $\lambda$, let $\gamma,\beta\in \mathfrak{h}^{*}$ and $w\in W[\gamma],v\in V[\beta]$, the assignment
\begin{align*}
	w,v\to \langle (\Phi^w_{\lambda-\beta}\otimes \id)\Phi^v_{\lambda}\rangle \in (W\otimes V)[\gamma+\beta]
\end{align*}
combining these maps for all $\gamma,\beta$, we get a linear map $J_{WV}(\lambda):W\otimes V\to W\otimes V$

\begin{theorem}
\label{Th_E_V}
These operator satisfy the following dynamical twist equation in $V\otimes W\otimes U$:
\begin{align*}
	J^{12,3}_{V\otimes W,U}(\lambda)J^{(12)}_{VW}(\lambda-h^3)=J^{1,23}_{V,W\otimes U}(\lambda)J^{23}_{WU}(\lambda)
\end{align*}
\end{theorem}

Let $P_{12}:W\otimes V\to V\otimes W$ denote the permutation operator and let $J^{21}_{VW}(\lambda):W\otimes V\to W\otimes V$ be defined as follows: $J^{21}_VW:=P_{12}J_{VW}(\lambda)P_{12}$, we can define the following operators on $V\otimes W\otimes U:$
\begin{align*}
	&J^{13}_{VU}(\lambda-h^2):=P_{23}J^{12}_{VU}(\lambda-h^3)P_{23}\\
	&J^{23}_{WU}(\lambda-h^1):=P_{12}P_{23}J^{12}_{WU}(\lambda-h^3)P_{23}P_{12}
	&J^{13,2}_{V\otimes U,W}(\lambda)=P_{23}J^{12,3}_{V\otimes U,W}P_{23}
\end{align*}
other notations are defined similarly by using there permutation operators.

\begin{cor}
\begin{align*}
	J^{32}_{UW}(\lambda-h^1)=J^{23,1}_{W\otimes U,V}(\lambda)^{-1}J^{3,12}_{U,V\otimes W}(\lambda)J^{21}_{WV}(\lambda)
\end{align*}
\end{cor}

we also know that there exists an obvious Yang-Baxter operator $R_1=P$ on $V$, then we can define the following operator 
\begin{align*}
	R_2(\lambda)=(J_{\lambda})^{-1}R_1J_{\lambda}=(J_{VV}(\lambda)P)^{-1}P(J_{VV}(\lambda)P)
\end{align*}

Then we have the following
\begin{align*}
	R_2(\lambda)\otimes \id&=\left((J_{VV}(\lambda)P)^{-1}P(J_{VV}(\lambda)P)\right)\otimes \id\\
	&=\left( (J_{VV}(\lambda)P)^{-1}\otimes \id\right)(P\otimes \id)(\left(J_{VV}(\lambda)P\right)\otimes \id)\\
	&=\left((PJ_{VV}(\lambda)P)^{-1}\otimes \id\right)(P\otimes \id)\left(PJ_{VV}(\lambda)P\right)\otimes \id)\\
	&=(J^{21}_{VV}(\lambda))^{-1}(P\otimes\id)J^{21}_{VV}(\lambda)\\
	&=(J^{3,12}(\lambda)_{V,V\otimes V}J^{21}_{VV}(\lambda))^{-1}(P\otimes \id)J^{3,12}_{V,V\otimes V}(\lambda)J^{21}_{VV}(\lambda)\\
\end{align*}
\begin{align*}
	\id\otimes R_2(\lambda-h^1)&=\id\otimes \left((J_{VV}(\lambda-h^1)P)^{-1}P(J_{VV}(\lambda-h^1)P)\right)\\
	&=\left(\id\otimes (PJ_{VV}(\lambda-h^1)P)^{-1}\right)(\id\otimes P)\left(\id\otimes (PJ_{VV}(\lambda-h^1)P)\right)\\
	&=              (J^{32}(\lambda-h^{1}))^{-1}(\id\otimes P)J^{32}(\lambda-h^{1})\\
	&=\left(J^{23,1}_{V\otimes V,V}(\lambda)J_{VV}^{32}(\lambda-h^{1})\right)^{-1}(\id\otimes P) J^{23,1}_{V\otimes V,V}(\lambda)J_{VV}^{32}(\lambda-h^{1})
\end{align*}

We know that $\check{R}_2(\lambda)$ satisfies the dynamical Yang-Baxter equation.

\subsection{Groupoid picture for the $\mathfrak{sl}_2$ case}
For the $\mathfrak{sl}_2$ case, the groupoid we are going to construct is to insure $M_{\lambda}$ is irreducible, then the groupoid should not contain $\{0,1,2,\dots,\}$ and also want avoid the singularity of the $R$ matrix which is $(-1)$ and the singularity of the twist we construct which is $-2$.

Let $\pi_1$ denote the action groupoid $(\mathbb{Z}+b)\rtimes \mathbb{Z}$ where $b$ is some generic shift to avoid the singularity. Let $\pi_2$ denote the same groupoid, then we can form an intertwiner groupoid by simply connect the same object in $\pi_1$ and $\pi_2$ as in the figure \ref{E_V}, by abuse of notations the numbers in $\pi_1$ and $\pi_2$ are actually shifted with this $b$. For example $-4$ actually is $-4+b$:
\[
\begin{tikzcd}
\pi_1&\dots\arrow[r,bend right,"+"]&3\arrow[l, bend right,"-"']\arrow[r,bend right,"+"]&4\arrow[r,bend right,"+"]\arrow[l, bend right,"-"']&5\arrow[l,bend right,"-"']\arrow[r,bend right,"+"]&\dots\arrow[l,bend right,"-"']\\
\pi_1&\dots\arrow[r,bend right,"+"]&3\arrow[l, bend right,"-"']\arrow[r,bend right,"+"]&4\arrow[r,bend right,"+"]\arrow[l, bend right,"-"']&5\arrow[l,bend right,"-"']\arrow[r,bend right,"+"]&\dots\arrow[l,bend right,"-"']\\
\pi& (i)_{\pi_1}\arrow[r]&(i)_{\pi_2}&i\in \Z
\end{tikzcd}
\]

We can define the following groupoid graded vector spaces for arrows $(a,a+1)\in \pi_1=\pi_2$ and arrows $(a,a-1)\in \pi_1=\pi_2$
\begin{align*}
	&V^{\pi_1}_{a,a+1}=\C v_{+},\quad V^{\pi_1}_{a,a-1}=\C v_{-}\\
	&V^{\pi_2}_{a,a+1}=\C v_{+},\quad V^{\pi_2}_{a,a+1}=\C v_{-}\\	
\end{align*}

We can define the following transfer operator $j\in \Gamma(\pi_1\times \pi_1,\operatorname{Hom}(V^{\pi_1}\otimes V^{\pi_1}, V^{\pi_2}\otimes V^{\pi_2}))$,

\begin{equation}
\begin{aligned}
&j:(\lambda+\mu,\lambda)\mapsto \oplus_{\mu=\mu_1+\mu_2}\langle (\Phi^{v_{\mu_2}}_{\lambda-\mu_1}\otimes 1)\Phi^{v_{\mu_1}}_{\lambda}\rangle=J_{\lambda}|_{V\otimes V[\mu]}\\
&q:(\lambda+\mu,\lambda)\mapsto J^{3,12}(\lambda)J^{21}(\lambda)|_{V\otimes V\otimes V[\mu]}\\
\end{aligned}
\end{equation}

\begin{remark}
Notice the difference between the groupoid structure of subsection \ref{sub:8v_sos_example} of eight vertex and SOS models, this is for the consideration of the weight zero condition.
\end{remark}

\section{Example: cell system}\label{sec:example_cell_system}
We first reformulate Ocneanu cell calculus \cite{Ocneanu1988,Ocneanu1991} and also the cell system introduced in \cite{Roche1990}. The theory of cell calculus is a rather rich theory and we only focus one part of it.

Let $\pi_1$ and $\pi_2$ be two finite unoriented graphs with distinguished nodes $*_1$ and $*_2$, with two incidence matrices $M_1$ and $M_2$, we say that there exists a intertwiner from $M_1$ to $M_2$ if there exists a $\rm{card}(\operatorname{Ob}(\pi_1))\times \rm{card}(\operatorname{Ob}(\pi_2))$ matrix $C$ with nonnegative integer coeffcients, such that:
\begin{equation}\label{eq:initial_point}
M_1C=CM_2,\quad C_{*_1,i}=1\Leftrightarrow i=*_2,
\end{equation}
and for any $i\in \operatorname{Ob}(\pi_1)$, there exist $j\in \operatorname{Ob}(\pi_2)$ such that $C_{i,j}\ne 0$, this matrix will give rise to the connecting groupoid, we denote it by $\pi$.

\begin{lemma}\label{lem:cell_lemma_connecting_system}
There exists a quasi-unique connecting system for $\pi_2$.
\end{lemma}
\begin{proof}
From the condition that $C_{*_1,i}=1\Leftrightarrow i=*_2,$ for any path $\gamma$ in $\pi_2$, we can have a path from the starting point $*_2$ to $s_{\pi}(\gamma)$ and the edge $\beta_1$ which connect $*_1$ and $*_2$ satisfies the condition requirement of the quasi-unique connecting system. 
\end{proof}

The natural $\pi_1$ graded vector space $V^{\pi_1}$ is given by
$V^{\pi_1}_{\alpha}:=\C e_{\alpha}$, that is each edge is associated to a one dimensional vector space, similar for $V^{\pi_2}$, suppose that we have an dynamical Yang-Baxter operator $\check{R}_1$ on $V^{\pi_1}$. 

For each \emph{left invertible} map $c\in \Gamma(\pi\times \pi,\overline{\operatorname{Hom}}(V^{\pi_1},V^{\pi_2}))$ with its inverse $c^{-1}\in \Gamma(\pi\times \pi,\overline{\operatorname{Hom}}(V^{\pi_2},V^{\pi_1}))$ , let $j=c^{-1}*_{\otimes}c^{-1}\in \Gamma(\pi\times \pi,\overline{\operatorname{Hom}}(V^{\pi_2}\otimes V^{\pi_2},V^{\pi_1}\otimes V^{\pi_1}))$, from the lemma \ref{lem:factorization_invertible}, it has inverse with $j^{-1}=c^{-1}*_{\otimes}c^{-1}$.

With the map $j$ and a chosen quasi-unique connecting system by lemma \ref{lem:cell_lemma_connecting_system}, we can define  $\check{R}_2\in \Psi\big(\pi\times\pi,\operatorname{End}_{\rhd}(V^{\pi_2}\otimes V^{\pi_2})\big)$ 
\begin{equation}\label{eq:cell_twist_R}
	\check{R}_2(\beta_1,\beta_1)=j^{-1}*_{\circ}(\check{R}_1j):=\sum_{\beta}j^{-1}(\beta,\beta_1)\check{R}_1j(\beta_1,\beta)
\end{equation}

\begin{definition}
	A \emph{left invertible} $c$ is called a cell system if for any $\gamma\in \pi_2$, we have the following relations 
	\begin{subequations}\label{eq:cell_system_twist}
		\begin{equation}
			\sum_{\beta}j^{-1}(\beta,\beta_1)\check{R}_1j(\beta_1,\beta)=\sum_{\beta'}j^{-1}(\beta',\beta'_1)\check{R}_1j(\beta'_1,\beta')
		\end{equation}
		\begin{equation}
			\sum_{\beta}j^{-1}(\beta,\beta_2)\check{R}_1j(\beta_1,\beta)=0
		\end{equation}
	\end{subequations}
	for any $\beta_1$ and $\beta'_1$ with $t_{\pi}(\beta'_1)=t_{\pi}(\beta_1)=s_{\pi_2}(\gamma)$, for any $\beta_2\ne \beta_1$ with $s_{\pi}(\beta_2)=s(\beta_1)$.
\end{definition}

\begin{prop}
Suppose that $c$ is invertible, then $c$ is a cell twist.
\end{prop}

\begin{proof}
As in lemma \ref{lem:cell_lemma_connecting_system}, there always exists a quasi-unique connecting system, the natural choice for any $\gamma$ is the path starting from the initial point $*_2$ to the point $\gamma'$, suppose that the length is $n$, then $q$ is $n$ times $*_{\otimes}$ of $c^{-1}$,  and the corresponding $q_1$ is $(n-3)$ times $*_{\otimes}$ of $c^{-1}$ and $q_2=c^{-1}*_{\otimes}c^{-1}*_{\otimes}c^{-1}$ as in the definition of quasi-unique twist and the relation \eqref{eq:quntwist12},\eqref{eq:quntwist23}, \eqref{eq:qutwist_q12} and \eqref{eq:qutwist_q23} follows from \eqref{eq:cell_system_twist} and the invertibility of c.
\end{proof}

\begin{remark}
From the definition of cell system, for example in \cite{Roche1990}, the unitary condition in \cite{Roche1990} only implies that $c$ is left invertible, but in order to check that it is indeed a twist, we need to check the right invertiblity.
\end{remark}

\subsection{Twist $A_{2L-3}$ elliptic  $\check{R}$ matrix to elliptic $D_{L}$}
In this subsection, we describe the cell twist between the type $A_{2L-3}$ and $D_{L}$. The example is a reformulation of Roche's example in \cite{Roche1990} and also the non-critical Fendley--Ginsparg orbifold constructions in \cite{Fendley1989}.

The dynamical $\check{R}$ matrix on the Dynkin diagram $A_{2L-3}$ is defined by restricting the dynamical $\check{R}$ matrix on the unrestricted groupoid \eqref{eq:graph_unres_A}.

Fix two complex numbers $\tau$ such that $\rm{Im} \tau>0$ and $\frac{1}{2L-2}\notin \Z+\tau\Z, L\ge 2$, let
\begin{equation}
	\theta(z,\tau)=-\sum_{n\in \Z}e^{i\pi(n+\frac{1}{2})^2\tau+2\pi i(n+\frac{1}{2})(z+\frac{1}{2})}
\end{equation}
be the odd Jacobi theta function and $[z]=\theta( z/(2L-2),\tau)/(\theta^{'}(0,\tau)/(2L-2))$ is normalized to have derivative 1 at $z=0$.

We first define the unrestricted case, the action groupoid is $\pi_A^{\text{unres}}=(\Z+b)\rtimes \Z$, where $\Z$ acts on $\Z$ by translation, $b\in \R$ is a generic shift to avoid the singularity of theta function, graphically it is the infinite length chain \eqref{eq:graph_unres_A}.
\begin{equation}\label{eq:graph_unres_A}
\begin{tikzpicture}
	\draw (-3,0) node{\large $\pi^{\text{unres}}_A:$};
	\draw (-1,0)--(0,0)--(1,0)--(2,0)--(3,0)--(4,0)--(5,0)--(6,0);
	\node at (-1,0)[circle,fill,inner sep=1.5pt]{};
	\draw (-1,-0.3) node{\small $\dots$};
	\node at (0,0)[circle,fill,inner sep=1.5pt]{};
	\draw (0,-0.3) node{\small $1+b$};
	\node at (1,0)[circle,fill,inner sep=1.5pt]{};
	\draw (1,-0.3) node{\small $2+b$};
	\node at (2,0)[circle,fill,inner sep=1.5pt]{};
	\draw (2,-0.3) node{\small $\dots$};
	\node at (3,0)[circle,fill,inner sep=1.5pt]{};
	\node at (4,0)[circle,fill,inner sep=1.5pt]{};
	\node at (5,0)[circle,fill,inner sep=1.5pt]{};
	\draw (5,-0.3) node{\small $n+b$};
	\node at (6,0)[circle,fill,inner sep=1.5pt]{};
	\draw (6,-0.3) node{\small $\dots$};
\end{tikzpicture}
\end{equation}

The corresponding groupoid graded vector space is one dimension for each oriented edge.
\begin{equation}
	V^{\pi^{\text{unres}}_A}_{(a,-1)}=\C e_{(a,-1)},V^{\pi^{\text{unres}}_A}_{(a,+1)}=\C e_{(a,+1)},\forall a\in \rm{Ob}(\pi^{\text{unres}}_A) \\
\end{equation} 

Then we have the following isomorphism of the source fiber space for any $a\in \rm{Ob}(\pi^{\rm{unres}})$
\begin{equation}\label{eq:fiberspace}
	\oplus_{\alpha\in s^{-1}(a)}(V^{\pi^{\text{unres}}_A}\otimes V^{\pi^{\text{unres}}_A})_{\alpha}\cong \C^4
\end{equation}
by identifying $e_{(a,+1)}\otimes e_{(a+1,+1)}$ with $(1,0)\otimes (1,0)$, $e_{(a,+1)}\otimes e_{(a+1,-1)}$ with $(1,0)\otimes (0,1)$,$e_{(a,-1)}\otimes e_{(a-1,+1)}$ with $(0,1)\otimes (1,0)$ and $e_{(a,-1)}\otimes e_{(a-1,-1)}$ with $(0,1)\otimes (0,1)$.

With the identification \eqref{eq:fiberspace}, if we let $E_{ij}$ be the $2\times 2$ matrix unit such that $E_{ij}e_k=\delta_{jk}e_i$ for all $k\in \{1,2\}$, then Felder's elliptic dynamical \cite{Felder1994} $R$ matrix with Andrews--Baxter--Forrester \cite{Andrews1984} parametrization is
\begin{equation}\label{eq:ellp_A_Rmatrix}
\begin{split}
	\check{R}^{\text{ell}}_A(z,a)&=\sum_{i=1}^{2}E_{ii}\otimes E_{ii}+\frac{\sqrt{[a-1][a+1]}[z]}{[a][1-z]}E_{21}\otimes E_{12}+\frac{\sqrt{[a+1][a-1]}[z]}{[a][1-z]}E_{12}\otimes E_{21}
\\
+&\frac{[a+z][1]}{[a][1-z]}E_{11}\otimes E_{22}+\frac{[a-z][1]}{[a][1-z]}E_{22}\otimes E_{11}
\end{split}
\end{equation} 

Now set $b=0$ and restrict to the following groupoid, we get our elliptic $\check{R}$ matrix on Dynkin diagram $A_{2L-3}$. In the restricted case, we consider only the subgroupoid $\pi^{\text{unres}}_A$ which is the full subgroupoid of $\Z\rtimes \Z$ support on the set $P^{2L-3}_{++}:=\{\lambda|\lambda\in \Z,1\le \lambda\le 2L-3\}$, graphically it is simply \eqref{graph_A} and $1*$ is the initial point required in \eqref{graph_A}.
\begin{equation}\label{graph_A}
\begin{tikzpicture}
	\draw (-1,0) node{\large $\pi_A:$};
	\draw (0,0)--(1,0)--(2,0)--(3,0)--(4,0)--(5,0)--(6,0)--(7,0);
	\node at (0,0)[circle,fill,inner sep=1.5pt]{};
	\draw (0,-0.3) node{\small $1*$};
	\node at (1,0)[circle,fill,inner sep=1.5pt]{};
	\draw (1,-0.3) node{\small $2$};
	\node at (2,0)[circle,fill,inner sep=1.5pt]{};
	\draw (2,-0.3) node{\small $\dots$};
	\node at (3,0)[circle,fill,inner sep=1.5pt]{};
	\draw (3,-0.3) node{\small $L-2$};
	\node at (4,0)[circle,fill,inner sep=1.5pt]{};
	\draw (4,-0.3) node{\small $L-1$};
	\node at (5,0)[circle,fill,inner sep=1.5pt]{};
	\draw (5,-0.3) node{\small $L$};
	\node at (6,0)[circle,fill,inner sep=1.5pt]{};
	\draw (6,-0.3) node{\small $\dots$};
	\node at (7,0)[circle,fill,inner sep=1.5pt]{};
	\draw (7,-0.3) node{\small $2L-3$};
\end{tikzpicture}
\end{equation}

Then we define the groupoid vector space $V^{\pi_A}$ to be 1 dimension on each component,
\begin{align*}
	&V^{\pi_A}_{(a,-1)}=\C e_{(a,-1)},~\rm{if}~ a,a-1\in P^{2L-3}_{++}\\
	&V^{\pi_A}_{(a,+1)}=\C e_{(a,+1)},~\rm{if}~ a,a+1\in P^{2L-3}_{++}
\end{align*} 

In this case, we do not have the isomorphism \eqref{eq:fiberspace} for all $a\in \pi^{\rm{res}}$, because there are boundary points, the source fibers of boundary fibers are different from that of middle points. And these source fibers can all be seen as subspaces of the unrestricted one, by the corrollary 4.2 of the \cite{Felder2020}, the elliptic $\check{R}(z,a)$ behaves well in these subspaces.

For the type $D$ side, the groupoid $\pi_D$ is the groupoid based on the unoriented Dynkin diagram $D_{L}$. The groupoid structure is defined in the example \ref{ex:graph}, graphically it is \eqref{eq:graph_D_L}, here $1*$ is the initial point required by \ref{eq:initial_point}. And for each oriented edge, we define a one dimensional vector space this gives our $V^{\pi_D}$.
\begin{equation}\label{eq:graph_D_L}
\begin{tikzpicture}
	\draw (-1,-2) node{\large $\pi_D:$};
	\draw (1,-2)--(2,-2)--(3,-2)--(4,-2)--(5,-2)--(6,-1.5);\draw (5,-2)--(6,-2.5);
	\node at (1,-2)[circle,fill,inner sep=1.5pt]{};
	\node at (2,-2)[circle,fill,inner sep=1.5pt]{};
	\node at (3,-2)[circle,fill,inner sep=1.5pt]{};
	\node at (4,-2)[circle,fill,inner sep=1.5pt]{};
	\node at (5,-2)[circle,fill,inner sep=1.5pt]{};
	\node at (6,-1.5)[circle,fill,inner sep=1.5pt]{};
	\node at (6,-2.5)[circle,fill,inner sep=1.5pt]{};
	\draw (1,-2.3) node{\small $1*$};
	\draw (2,-2.3) node{\small $2$};
	\draw (3,-2.3) node{\small $3$};
	\draw (4,-2.3) node{\small $\dots$};
	\draw (5,-2.3) node{\small $L-2$};
	\draw (6.3,-1.5) node{\small $L$};
	\draw (6.5,-2.5) node{\small $L-1$};
\end{tikzpicture}
\end{equation}

We now define the connecting groupoid for the $\pi_A$ and $\pi_D$, it is graphically the following:
\begin{equation}
\begin{tikzpicture}
\draw (-1,-4) node{\large $\pi:$};
\draw (1,-3.5)--(2,-4)node[sloped,pos=0.5,allow upside down]{\arrowIn};
\draw (1,-4.5)--(2,-4)node[sloped,pos=0.5,allow upside down]{\arrowIn};
\node at (1,-3.5)[circle,fill,inner sep=1.5pt]{};
\node at (2,-4)[circle,fill,inner sep=1.5pt]{};
\node at (1,-4.5)[circle,fill,inner sep=1.5pt]{};
\draw (2.5,-4) node{\small $(1_{*})_D,$};
\draw (0.5,-3.5) node{\small $(1_{*})_A$};
\draw (0.2,-4.5) node{\small $(2L-3)_A$};
\draw (3.5,-4) node{\large$\dots$};
\draw (5,-3.5)--(6,-4)node[sloped,pos=0.5,allow upside down]{\arrowIn};
\draw (5,-4.5)--(6,-4)node[sloped,pos=0.5,allow upside down]{\arrowIn};
\node at (5,-3.5)[circle,fill,inner sep=1.5pt]{};
\node at (6,-4)[circle,fill,inner sep=1.5pt]{};
\node at (5,-4.5)[circle,fill,inner sep=1.5pt]{};
\draw (6.7,-4) node{\small $(L-2)_D,$};
\draw (4.3,-3.5) node{\small $(L-2)_A$};
\draw (4.5,-4.5) node{\small $(L)_A$};
\draw (8.7,-4)--(9.7,-3.5)node[sloped,pos=0.5,allow upside down]{\arrowIn};
\draw (8.7,-4)--(9.7,-4.5)node[sloped,pos=0.5,allow upside down]{\arrowIn};
\node at (8.7,-4)[circle,fill,inner sep=1.5pt]{};
\node at (9.7,-3.5)[circle,fill,inner sep=1.5pt]{};
\node at (9.7,-4.5)[circle,fill,inner sep=1.5pt]{};
\draw (8.0,-4) node{\small $(L-1)_A$};
\draw (10.4,-3.5) node{\small $(L-1)_D$};
\draw (10.2,-4.5) node{\small $(L)_D$};
\end{tikzpicture}
\end{equation}

We now describe the homomorphism $\widetilde{C}\in \Gamma(\pi\times \pi,\overline{\operatorname{Hom}}(V^{\pi_A},V^{\pi_D}))$. For $i=1,\dots,L-3$ and $j=0,1,\dots,L-4$, we have the following coefficients of the $\tilde{C}$ with respect to the chosen basis vector of each one dimensional component.
\begin{equation*}
\begin{tikzpicture}[scale=1]
	\draw (0,1)--(0,0)node[sloped,pos=0.5,allow upside down]{\arrowIn}--(1,0)node[sloped,pos=0.5,allow upside down]{\arrowIn};
	\draw (0,1)--(1,1)node[sloped,pos=0.5,allow upside down]{\arrowIn}--(1,0)node[sloped,pos=0.5,allow upside down]{\arrowIn};	
	\node at  (0,0)[circle,fill=black,inner sep=1pt]{};
	\node at  (1,0)[circle,fill=black,inner sep=1pt]{};
	\node at  (0,1)[circle,fill=black,inner sep=1pt]{};
	\node at  (1,1)[circle,fill=black,inner sep=1pt]{};
	\node at (-0.3,1.3){\small$(i+1)_A$};
	\node at (-0.3,-0.3){\small$(i)_A$}; 
	\node at (1.4,1.3){\small $(i+1)_D$};
	\node at (1.4,-0.3){\small$(i)_D$};
	\node at (2.0,0.5){$=$};
	\draw (3,1)--(3,0)node[sloped,pos=0.5,allow upside down]{\arrowIn}--(4,0)node[sloped,pos=0.5,allow upside down]{\arrowIn};
	\draw (3,1)--(4,1)node[sloped,pos=0.5,allow upside down]{\arrowIn}--(4,0)node[sloped,pos=0.5,allow upside down]{\arrowIn};	
	\node at  (3,0)[circle,fill=black,inner sep=1pt]{};
	\node at  (4,0)[circle,fill=black,inner sep=1pt]{};
	\node at  (3,1)[circle,fill=black,inner sep=1pt]{};
	\node at  (4,1)[circle,fill=black,inner sep=1pt]{};
	\node at (2.7,1.3){\small$(i)_A$};
	\node at (2.7,-0.3){\small$(i+1)_A$}; 
	\node at (4.3,1.3){\small$(i)_D$};
	\node at (4.3,-0.3){\small$(i+1)_D$};
	\node at (5.5,0.5){$=$};
	\draw (6.5,1)--(6.5,0)node[sloped,pos=0.5,allow upside down]{\arrowIn}--(7.5,0)node[sloped,pos=0.5,allow upside down]{\arrowIn};
	\draw (6.5,1)--(7.5,1)node[sloped,pos=0.5,allow upside down]{\arrowIn}--(7.5,0)node[sloped,pos=0.5,allow upside down]{\arrowIn};	
	\node at  (6.5,0)[circle,fill=black,inner sep=1pt]{};
	\node at  (7.5,0)[circle,fill=black,inner sep=1pt]{};
	\node at  (6.5,1)[circle,fill=black,inner sep=1pt]{};
	\node at  (7.5,1)[circle,fill=black,inner sep=1pt]{};
	\node at (6,1.3){\small$(2L-2-j)_A$};
	\node at (6,-0.3){\small$(2L-3-j)_A$}; 
	\node at (8,1.3){\small$(i+2)_D$};
	\node at (8,-0.3){\small$(i+1)_D$};
	\node at (9,0.5){$=$};
	\draw (10.5,1)--(10.5,0)node[sloped,pos=0.5,allow upside down]{\arrowIn}--(11.5,0)node[sloped,pos=0.5,allow upside down]{\arrowIn};
	\draw (10.5,1)--(11.5,1)node[sloped,pos=0.5,allow upside down]{\arrowIn}--(11.5,0)node[sloped,pos=0.5,allow upside down]{\arrowIn};	
	\node at  (10.5,0)[circle,fill=black,inner sep=1pt]{};
	\node at  (11.5,0)[circle,fill=black,inner sep=1pt]{};
	\node at  (10.5,1)[circle,fill=black,inner sep=1pt]{};
	\node at  (11.5,1)[circle,fill=black,inner sep=1pt]{};
	\node at (10,1.3){\small$(2L-3-j)_A$};
	\node at (10,-0.3){\small$(2L-2-j)_A$}; 
	\node at (12,1.3){\small$(i+1)_D$};
	\node at (12,-0.3){\small$(i+2)_D$};
	\node at (12,0.5){$=1$};
\end{tikzpicture}
\end{equation*}

\begin{equation*}
\begin{tikzpicture}[scale=1]
	\draw (0,1)--(0,0)node[sloped,pos=0.5,allow upside down]{\arrowIn}--(1,0)node[sloped,pos=0.5,allow upside down]{\arrowIn};
	\draw (0,1)--(1,1)node[sloped,pos=0.5,allow upside down]{\arrowIn}--(1,0)node[sloped,pos=0.5,allow upside down]{\arrowIn};	
	\node at  (0,0)[circle,fill=black,inner sep=1pt]{};
	\node at  (1,0)[circle,fill=black,inner sep=1pt]{};
	\node at  (0,1)[circle,fill=black,inner sep=1pt]{};
	\node at  (1,1)[circle,fill=black,inner sep=1pt]{};
	\node at (-0.3,1.3){\small$(L-1)_A$};
	\node at (-0.3,-0.3){\small$(L-2)_A$}; 
	\node at (1.4,1.3){\small $(L-1)_D$};
	\node at (1.4,-0.3){\small$(L-2)_D$};
	\node at (1.5,0.5){$=\frac{1}{\sqrt{2}}$};
	\draw (3.2,1)--(3.2,0)node[sloped,pos=0.5,allow upside down]{\arrowIn}--(4.2,0)node[sloped,pos=0.5,allow upside down]{\arrowIn};
	\draw (3.2,1)--(4.2,1)node[sloped,pos=0.5,allow upside down]{\arrowIn}--(4.2,0)node[sloped,pos=0.5,allow upside down]{\arrowIn};	
	\node at  (3.2,0)[circle,fill=black,inner sep=1pt]{};
	\node at  (4.2,0)[circle,fill=black,inner sep=1pt]{};
	\node at  (3.2,1)[circle,fill=black,inner sep=1pt]{};
	\node at  (4.2,1)[circle,fill=black,inner sep=1pt]{};
	\node at (2.9,1.3){\small$(L-1)_A$};
	\node at (2.9,-0.3){\small$(L-2)_A$}; 
	\node at (4.5,1.3){\small$(L)_D$};
	\node at (4.5,-0.3){\small$(L-2)_D$};
	\node at (4.8,0.5){$=-\frac{1}{\sqrt{2}}$};
	\draw (6.7,1)--(6.7,0)node[sloped,pos=0.5,allow upside down]{\arrowIn}--(7.7,0)node[sloped,pos=0.5,allow upside down]{\arrowIn};
	\draw (6.7,1)--(7.7,1)node[sloped,pos=0.5,allow upside down]{\arrowIn}--(7.7,0)node[sloped,pos=0.5,allow upside down]{\arrowIn};	
	\node at  (6.7,0)[circle,fill=black,inner sep=1pt]{};
	\node at  (7.7,0)[circle,fill=black,inner sep=1pt]{};
	\node at  (6.7,1)[circle,fill=black,inner sep=1pt]{};
	\node at  (7.7,1)[circle,fill=black,inner sep=1pt]{};
	\node at (6.2,1.3){\small$(L-1)_A$};
	\node at (6.2,-0.3){\small$(L)_A$}; 
	\node at (8.2,1.3){\small$(L-1)_D$};
	\node at (8.2,-0.3){\small$(L-2)_D$};
	\node at (8.4,0.5){$=\frac{1}{\sqrt{2}}$};
	\draw (10.5,1)--(10.5,0)node[sloped,pos=0.5,allow upside down]{\arrowIn}--(11.5,0)node[sloped,pos=0.5,allow upside down]{\arrowIn};
	\draw (10.5,1)--(11.5,1)node[sloped,pos=0.5,allow upside down]{\arrowIn}--(11.5,0)node[sloped,pos=0.5,allow upside down]{\arrowIn};	
	\node at  (10.5,0)[circle,fill=black,inner sep=1pt]{};
	\node at  (11.5,0)[circle,fill=black,inner sep=1pt]{};
	\node at  (10.5,1)[circle,fill=black,inner sep=1pt]{};
	\node at  (11.5,1)[circle,fill=black,inner sep=1pt]{};
	\node at (10,1.3){\small$(L-1)_A$};
	\node at (10,-0.3){\small$(L)_A$}; 
	\node at (12,1.3){\small$(L)_D$};
	\node at (12,-0.3){\small$(L-2)_D$};
	\node at (12,0.5){$=\frac{1}{\sqrt{2}}$};
\end{tikzpicture}
\end{equation*}

\begin{equation*}
\begin{tikzpicture}[scale=1]
	\draw (0,1)--(0,0)node[sloped,pos=0.5,allow upside down]{\arrowIn}--(1,0)node[sloped,pos=0.5,allow upside down]{\arrowIn};
	\draw (0,1)--(1,1)node[sloped,pos=0.5,allow upside down]{\arrowIn}--(1,0)node[sloped,pos=0.5,allow upside down]{\arrowIn};	
	\node at  (0,0)[circle,fill=black,inner sep=1pt]{};
	\node at  (1,0)[circle,fill=black,inner sep=1pt]{};
	\node at  (0,1)[circle,fill=black,inner sep=1pt]{};
	\node at  (1,1)[circle,fill=black,inner sep=1pt]{};
	\node at (-0.3,1.3){\small$(L-2)_A$};
	\node at (-0.3,-0.3){\small$(L-1)_A$}; 
	\node at (1.4,1.3){\small $(L-2)_D$};
	\node at (1.4,-0.3){\small$(L-1)_D$};
	\node at (1.5,0.5){$=1$};
	\draw (3.2,1)--(3.2,0)node[sloped,pos=0.5,allow upside down]{\arrowIn}--(4.2,0)node[sloped,pos=0.5,allow upside down]{\arrowIn};
	\draw (3.2,1)--(4.2,1)node[sloped,pos=0.5,allow upside down]{\arrowIn}--(4.2,0)node[sloped,pos=0.5,allow upside down]{\arrowIn};	
	\node at  (3.2,0)[circle,fill=black,inner sep=1pt]{};
	\node at  (4.2,0)[circle,fill=black,inner sep=1pt]{};
	\node at  (3.2,1)[circle,fill=black,inner sep=1pt]{};
	\node at  (4.2,1)[circle,fill=black,inner sep=1pt]{};
	\node at (2.9,1.3){\small$(L-2)_A$};
	\node at (2.9,-0.3){\small$(L-1)_A$}; 
	\node at (4.5,1.3){\small$(L-2)_D$};
	\node at (4.5,-0.3){\small$(L)_D$};
	\node at (4.8,0.5){$=-1$};
	\draw (6.7,1)--(6.7,0)node[sloped,pos=0.5,allow upside down]{\arrowIn}--(7.7,0)node[sloped,pos=0.5,allow upside down]{\arrowIn};
	\draw (6.7,1)--(7.7,1)node[sloped,pos=0.5,allow upside down]{\arrowIn}--(7.7,0)node[sloped,pos=0.5,allow upside down]{\arrowIn};	
	\node at  (6.7,0)[circle,fill=black,inner sep=1pt]{};
	\node at  (7.7,0)[circle,fill=black,inner sep=1pt]{};
	\node at  (6.7,1)[circle,fill=black,inner sep=1pt]{};
	\node at  (7.7,1)[circle,fill=black,inner sep=1pt]{};
	\node at (6.2,1.3){\small$(L)_A$};
	\node at (6.2,-0.3){\small$(L-1)_A$}; 
	\node at (8.2,1.3){\small$(L-2)_D$};
	\node at (8.2,-0.3){\small$(L-1)_D$};
	\node at (8.4,0.5){$=1$};
	\draw (10.5,1)--(10.5,0)node[sloped,pos=0.5,allow upside down]{\arrowIn}--(11.5,0)node[sloped,pos=0.5,allow upside down]{\arrowIn};
	\draw (10.5,1)--(11.5,1)node[sloped,pos=0.5,allow upside down]{\arrowIn}--(11.5,0)node[sloped,pos=0.5,allow upside down]{\arrowIn};	
	\node at  (10.5,0)[circle,fill=black,inner sep=1pt]{};
	\node at  (11.5,0)[circle,fill=black,inner sep=1pt]{};
	\node at  (10.5,1)[circle,fill=black,inner sep=1pt]{};
	\node at  (11.5,1)[circle,fill=black,inner sep=1pt]{};
	\node at (10,1.3){\small$(L)_A$};
	\node at (10,-0.3){\small$(L-1)_A$}; 
	\node at (12,1.3){\small$(L-2)_D$};
	\node at (12,-0.3){\small$(L)_D$};
	\node at (12,0.5){$=1$};
\end{tikzpicture}
\end{equation*}

And  $(\widetilde{C})^{-1}\in \Gamma(\pi\times \pi,\overline{\operatorname{Hom}}(V^{\pi_D},V^{\pi_A}))$ is described by
\begin{equation}
\begin{tikzcd}
		\bullet_{d_1}\arrow[d,"V^{\pi_D}_{\gamma}"']&\bullet_{a_1}\arrow[l,"\beta_2"']\arrow[d,"V^{\pi_A}_\alpha"]\\
		\bullet_{d_2}&\arrow[l,"\beta_1"]\bullet_{a_2}
\end{tikzcd}
=~
\begin{tikzcd}
\bullet_{a_1}\arrow[d,"V^{\pi_A}_{\alpha}"']\arrow[r,"\beta_1"]&\bullet_{d_1}\arrow[d,"V^{\pi_D}_{\gamma}"]\\
\bullet_{a_2}\arrow[r,"\beta_2"]&\bullet_{d_2}
\end{tikzcd}
\end{equation}

\begin{prop}[Roche \cite{Roche1990}]
$\tilde{C}$ is a cell system
\end{prop}

\begin{lemma}
	$\widetilde{C}$ is invertible with the inverse $\widetilde{C}^{-1}$.
\end{lemma}

\begin{prop}
$\tilde{C}$ is a cell twist, so from the formula \eqref{eq:cell_system_twist}, we get an $\check{R}^{\rm{ell}}_D$ defined on the space $V^{\pi_D}$.
\end{prop}

\subsection{Twist $A_{11}$ trigonometric  $\check{R}$ matrix to trigonometric $E_{6}$} 
\hfill\\
This example is a reformulation of Roche's example of intertwiner in \cite{Roche1990}, see also the work of Francesco--Zuber \cite{DiFrancesco1990}. We got the twist from the intertwiners.

For the restricted dynamical $\check{R}^{\text{ell}}_A$ matrix defined on the vector space $V^{\pi_A}$., we take the trigonometric limit. Let $\langle z\rangle:=\sin(\frac{\pi z}{12})$, then we have the trigonometric $\check{R}^{\rm{tri}}_A$ matrix \eqref{eq:tri_A_Rmatrix} on the $V^{\pi^{\rm{unres}}_A}$, similarly by restriction we get $\check{R}^{\rm{tri}}_A$ on $V^{\pi_A}$.

\begin{equation}\label{eq:tri_A_Rmatrix}
\begin{split}
	\check{R}^{\text{tri}}_A(z,a)&=\sum_{i=1}^{2}E_{ii}\otimes E_{ii}+\frac{\sqrt{\langle a-1\rangle\langle a+1\rangle}\langle z\rangle}{\langle a\rangle\langle1-z\rangle}E_{21}\otimes E_{12}+\frac{\sqrt{\langle a+1\rangle\langle a-1\rangle}\langle z\rangle}{\langle a\rangle\langle 1-z\rangle}E_{12}\otimes E_{21}
\\
+&\frac{\langle a+z\rangle\langle 1\rangle}{\langle a\rangle\langle 1-z\rangle}E_{11}\otimes E_{22}+\frac{\langle a-z\rangle\langle 1\rangle}{\langle a\rangle\langle 1-z\rangle}E_{22}\otimes E_{11}
\end{split}
\end{equation}

The groupoid $\pi_E$ is the following \eqref{graph_E} Dynkin digram of type $E$ and vector space $V^{\pi_E}$ is defined by assigning a one dimensional vector space to each oriented edge.
\begin{equation}\label{graph_E}
	\scalebox{0.7}{
	\begin{tikzpicture}
		\draw (-1,0) node{\large $\pi_E:$};
		\draw (0,0)--(1,0)--(2,0)--(3,0)--(4,0);
		\node at (0,0)[circle,fill,inner sep=1.5pt]{};
		\draw (0,0.3) node{\small $1*$};
		\node at (1,0)[circle,fill,inner sep=1.5pt]{};
		\draw (1,0.3) node{\small $2$};
		\node at (2,0)[circle,fill,inner sep=1.5pt]{};
		\draw (2,0.3) node{\small $3$};
		\draw (2,0)--(2,-1);
		\node at (2,-1)[circle,fill,inner sep=1.5pt]{};
		\draw (2,-1.3) node{\small $6$};
		\node at (3,0)[circle,fill,inner sep=1.5pt]{};
		\draw (3,0.3) node{\small $4$};
		\node at (4,0)[circle,fill,inner sep=1.5pt]{};
		\draw (4,0.3) node{\small $5$};
	\end{tikzpicture}}
\end{equation}

The connecting groupoid $\pi$ is described by the following:
\begin{equation}\label{graph_A_E}
	\scalebox{0.8}{
	\begin{tikzpicture}
		\begin{scope}[shift={(-2,-0.3)}]
		\draw (0,0) node{\large $\pi:$};
		\draw (1,0)--(2,0)node[sloped,pos=0.5,allow upside down]{\arrowIn};
		\draw (3,0)--(2,0)node[sloped,pos=0.5,allow upside down]{\arrowIn};
		\draw (3,0)--(4,0)node[sloped,pos=0.5,allow upside down]{\arrowIn};
		\draw (5,0)--(4,0)node[sloped,pos=0.5,allow upside down]{\arrowIn};
		\draw (5,0)--(6,0)node[sloped,pos=0.5,allow upside down]{\arrowIn};
		\draw (7,0)--(6,0)node[sloped,pos=0.5,allow upside down]{\arrowIn};
		\draw(4,1)--(4,0)node[sloped,pos=0.5,allow upside down]{\arrowIn};
		\draw(4,-1)--(4,0)node[sloped,pos=0.5,allow upside down]{\arrowIn};
		
		\node at (4,1)[circle,fill,inner sep=1.5pt]{};
		\draw (4,1) node[above]{\small $9_A$};
		\node at (4,-1)[circle,fill,inner sep=1.5pt]{};
		\draw (4,-1) node[below]{\small $3_A$};
		\node at (1,0)[circle,fill,inner sep=1.5pt]{};
		\draw (1,-0.3) node{\small $1_A$};
		\node at (2,0)[circle,fill,inner sep=1.5pt]{};
		\draw (2,-0.3) node{\small $1_E$};
		\node at (3,0)[circle,fill,inner sep=1.5pt]{};
		\draw (3,-0.3) node{\small $7_A$};
		\node at (4,0)[circle,fill,inner sep=1.5pt]{};
		\draw (4,0) node[above right]{\small $3_E$};
		\node at (5,0)[circle,fill,inner sep=1.5pt]{};
		\draw (5,-0.3) node{\small $5_A$};
		\node at (6,0)[circle,fill,inner sep=1.5pt]{};
		\draw (6,-0.3) node{\small $5_E$};
		\node at (7,0)[circle,fill,inner sep=1.5pt]{};
		\draw (7,-0.3) node{\small $11_A$};
\end{scope}	
\end{tikzpicture}
\begin{tikzpicture}
	\draw (1,0)--(2,0)node[sloped,pos=0.5,allow upside down]{\arrowIn};
	\draw (3,1)--(2,0)node[sloped,pos=0.5,allow upside down]{\arrowIn};
	\draw (3,1)--(4,0)node[sloped,pos=0.5,allow upside down]{\arrowIn};
	\draw (5,0)--(4,0)node[sloped,pos=0.5,allow upside down]{\arrowIn};
	\draw (2.5,-0.5)--(3,-1)node[sloped,pos=0.5,allow upside down]{\arrowIn};
	\draw (2.5,-0.5)--(2,0)node[sloped,pos=0.5,allow upside down]{\arrowIn};
    \draw (3.5,-0.5)--(3,-1)node[sloped,pos=0.5,allow upside down]{\arrowIn};
    \draw (3.5,-0.5)--(4,0)node[sloped,pos=0.5,allow upside down]{\arrowIn};
	\node at (1,0)[circle,fill,inner sep=1.5pt]{};
	\draw (1,0) node[left]{\small $2_A$};
	\node at (2,0)[circle,fill,inner sep=1.5pt]{};
	\draw (2,0) node[above]{\small $2_E$};
	\node at (3,1)[circle,fill,inner sep=1.5pt]{};
	\draw (3,1) node[above]{\small $6_A$};
	\node at (4,0)[circle,fill,inner sep=1.5pt]{};
	\draw (4,0) node[above]{\small $4_E$};
	\node at (5,0)[circle,fill,inner sep=1.5pt]{};
	\draw (5,0) node[right]{\small $10_A$};
	\node at (2.5,-0.5)[circle,fill,inner sep=1.5pt]{};
	\draw (2.5,-0.5) node[below left]{\small $8_A$};
	\node at (3.5,-0.5)[circle,fill,inner sep=1.5pt]{};
	\draw (3.5,-0.5) node[below right]{\small $4_A$};
	\node at (3,-1)[circle,fill,inner sep=1.5pt]{};
	\draw (3,-1) node[below]{\small $6_E$};	
	\end{tikzpicture}}
\end{equation}

We now describe the homomorphism $\widetilde{C}\in \Gamma(\pi\times \pi,\overline{\operatorname{Hom}}(V^{\pi_A},V^{\pi_E}))$, as the component of $V^{\pi_A}$ and $V^{\pi_E}$ are all one dimensional. It suffice to describe the coefficients of the $\widetilde{C}$ with respect to these one dimensional basis.

We let $\widetilde{C}(a_1,a_2,e_1,e_2)$ to describe the coefficient of the following component \eqref{eq:component_twist}.
\begin{equation}\label{eq:component_twist}
\begin{tikzcd}
	\bullet_{a_1}\arrow[d,"V^{\pi_A}_{\alpha}"']\arrow[r,"\beta_1"]&\bullet_{e_1}\arrow[d,"V^{\pi_E}_{\gamma}"]\\
	\bullet_{a_2}\arrow[r,"\beta_2"]&\bullet_{e_2}
\end{tikzcd}
\end{equation}
then the $\tilde{C}$ is described by
	\begin{equation*}
	\widetilde{C}(1_A,2_A,1_E,2_E)=\widetilde{C}(2_A,1_A,2_E,1_E)=\widetilde{C}(2_A,3_A,2_E,3_E)=1
	\end{equation*}
    \begin{equation*}
    \widetilde{C}(3_A,4_A,3_E,4_E)=\widetilde{C}(3_A,2_A,3_E,2_E)=-\widetilde{C}(3_A,4_A,3_E,4_E)=-\widetilde{C}(4_A,5_A,4_E,5_E)=1
    \end{equation*}
	\begin{equation*}
	\widetilde{C}(4_A,3_A,4_E,3_E)=\widetilde{C}(4_A,5_A,6_E,3_E)=3^{-1/4}
	\end{equation*}
	\begin{equation*}
	-\widetilde{C}(4_A,3_A,6_E,3_E)=\widetilde{C}(4_A,5_A,4_E,3_E)=(1-\frac{1}{\sqrt{3}})^{1/2}
	\end{equation*}
	\begin{equation*}
	\widetilde{C}(5_A,4_A,3_E,6_E)=\widetilde{C}(5_A,6_A,3_E,2_E)=1
	\end{equation*}
	\begin{equation*}
	-\widetilde{C}(5_A,4_A,5_E,4_E)=\widetilde{C}(5_A,6_A,3_E,4_E)=(2\sqrt{3}-3)^{1/2}
	\end{equation*}
	\begin{equation*}
	\widetilde{C}(5_A,4_A,3_E,4_E)=\widetilde{C}(5_A,6_A,5_E,4_E)=\sqrt{3}-1
	\end{equation*}
	\begin{equation*}
	\widetilde{C}(6_A,7_A,2_E,1_E)=\widetilde{C}(6_A,5_A,4_E,5_E)=1
	\end{equation*}
	\begin{equation*}
	\widetilde{C}(6_A,5_A,2_E,3_E)=\widetilde{C}(6_A,7_A,4_E,3_E)=\frac{1}{2}(\sqrt{3}+1)^{1/2}
	\end{equation*}
	\begin{equation*}
	\widetilde{C}(6_A,5_A,4_E,3_E)=-\widetilde{C}(6_A,7_A,2_E,3_E)=\frac{1}{2}(3-\sqrt{3})^{1/2}
	\end{equation*}
	\begin{equation*}
	\widetilde{C}(7_A,6_A,3_E,4_E)=\widetilde{C}(7_A,8_A,3_E,6_E)=1
	\end{equation*}
	\begin{equation*}
	\widetilde{C}(7_A,8_A,1_E,2_E)=-\widetilde{C}(7_A,6_A,3_E,2_E)=(2\sqrt{3}-3)^{1/2}
	\end{equation*}
	\begin{equation*}
	\widetilde{C}(7_A,8_A,3_E,2_E)=\widetilde{C}(7_A,6_A,1_E,2_E)=\sqrt{3}-1
	\end{equation*}
	\begin{equation*}
	\widetilde{C}(8_A,7_A,2_E,1_E)=1,\widetilde{C}(8_A,7_A,6_E,3_E)=\widetilde{C}(8_A,9_A,2_E,3_E)=3^{-1/4}
	\end{equation*}
	\begin{equation*}
	-\widetilde{C}(8_A,9_A,6_E,3_E)=\widetilde{C}(8_A,7_A,2_E,3_E)=\big(1-\frac{1}{\sqrt{3}} \big)^{1/2}
	\end{equation*}
	\begin{equation*}
	\widetilde{C}(9_A,8_A,3_E,2_E)=\widetilde{C}(9_A,10_A,3_E,4_E)=-\widetilde{C}(9_A,8_A,3_E,6_E)=1
	\end{equation*}
	\begin{equation*}
    \widetilde{C}(10_A,9_A,4_E,3_E)=\widetilde{C}(10_A,11_A,4_E,5_E)=\widetilde{C}(11_A,10_A,5_E,4_E)=1
	\end{equation*}

And we also let $(\widetilde{C})^{-1}\in \Gamma(\pi\times \pi,\overline{\operatorname{Hom}}(V^{\pi_E},V^{\pi_A}))$ be described by
\begin{equation}
(\widetilde{C})^{-1}~\rm{coefficent}\quad \begin{tikzcd}
		\bullet_{e_1}\arrow[d,"V^{\pi_E}_{\gamma}"']&\bullet_{a_1}\arrow[l,"\beta_1"']\arrow[d,"V^{\pi_A}_\alpha"]\\
		\bullet_{e_2}&\arrow[l,"\beta_2"]\bullet_{a_2}
	\end{tikzcd}
	=\widetilde{C}~\rm{coefficent}~
	\begin{tikzcd}
		\bullet_{a_1}\arrow[d,"V^{\pi_A}_{\alpha}"']\arrow[r,"\beta_1"]&\bullet_{e_1}\arrow[d,"V^{\pi_E}_{\gamma}"]\\
		\bullet_{a_2}\arrow[r,"\beta_2"]&\bullet_{e_2}
	\end{tikzcd}
\end{equation}

\begin{prop}[Roche \cite{Roche1990}]
	$\tilde{C}$ is a cell system
\end{prop}
\begin{lemma}
	$\widetilde{C}$ is invertible with the inverse $\widetilde{C}^{-1}$.
\end{lemma}
\begin{prop}
	$\tilde{C}$ is a cell twist, so from the formula \eqref{eq:cell_system_twist}, we get an $\check{R}^{\rm{ell}}_D$ defined on the space $V^{\pi_D}$.
\end{prop}

\subsection{Other cell systems}
There are many other important and interesting  intertwiners constructed by Pasquier \cite{Pasquier1988a} (this example maybe included as in subsection \ref{sec:dynamical_twist}, see \cite{Etingof1999} on discussion about the relation of twist with 3j symbols and 6j symbols), also in Roche \cite{Roche1990} and Francesco--Zuber \cite{DiFrancesco1990}.  It is not difficult to check that whether these intertwiners are from some cell twists, because the intertwiner relation are proved by the above authors, we only need to check the right invertiblity and also the quasi-unique connecting system condition.

\bibliographystyle{plain}
\bibliography{intertwiner_twist}

\begin{thebibliography}{10}

\bibitem{Aganagic2021}
M.~Aganagic and A.~Okounkov.
\newblock {Elliptic stable envelopes}.
\newblock {\em J.Am.Math.Soc.}, 34(1):79--133, dec 2021.

\bibitem{Andrews1984}
G.E. Andrews, R.J. Baxter, and P.J. Forrester.
\newblock {Eight-vertex SOS model and generalized Rogers-Ramanujan-type
  identities}.
\newblock {\em Journal of Statistical Physics}, 35(3-4):193--266, 1984.

\bibitem{Arnaudon1998}
D.~Arnaudon, E.~Buffenoir, E.~Ragoucy, and Ph~Roche.
\newblock {Universal Solutions of Quantum Dynamical Yang-Baxter Equations}.
\newblock {\em Letters in Mathematical Physics}, 44(3):201--214, 1998.

\bibitem{Babelon1991}
O.~Babelon.
\newblock {Universal exchange algebra for Bloch waves and Liouville theory}.
\newblock {\em Communications in Mathematical Physics 1991 139:3},
  139(3):619--643, aug 1991.

\bibitem{Babelon1996}
O.~Babelon, D.~Bernard, and E.~Billey.
\newblock {A quasi-Hopf algebra interpretation of quantum {3-j} and {6-j}
  symbols and difference equations}.
\newblock {\em Physics Letters, Section B: Nuclear, Elementary Particle and
  High-Energy Physics}, 375(1-4):89--97, 1996.

\bibitem{Baxter1973}
R.J. Baxter.
\newblock {Eight vertex model in lattice statistics and one-dimensional
  anisotropic Heisenberg chain. 2. Equivalence to a generalized ice-type
  lattice model}.
\newblock {\em Annals Phys.}, 76(1):25--47, 1973.

\bibitem{ChariPressleyBook1994}
V.~Chari and A.~Pressley.
\newblock {\em {A guide to quantum groups}}.
\newblock Cambridge University Press, Cambridge, 1994.

\bibitem{CostelloWittenYamazakiI2018}
K.~Costello, E.~Witten, and M.~Yamazaki.
\newblock {Gauge theory and integrability, {I}}.
\newblock {\em ICCM Not.}, 6(1):46--119, 2018.

\bibitem{CostelloWittenYamazakiII2018}
K.~Costello, E.~Witten, and M.~Yamazaki.
\newblock {Gauge theory and integrability, {II}}.
\newblock {\em ICCM Not.}, 6(1):120--146, 2018.

\bibitem{Date1987}
E.~Date, M.~Jimbo, A.~Kuniba, T.~Miwa, and M.~Okado.
\newblock {Exactly Solvable SOS Models II: Proof of the star-triangle relation
  and combinatorial identities}.
\newblock {\em Conformal Field Theory and Solvable Lattice Models}, pages
  17--122, jun 1987.

\bibitem{DiFrancesco1990}
P.~{Di Francesco} and J.B. Zuber.
\newblock {{$SU(N)$} lattice integrable models associated with graphs}.
\newblock {\em Nuclear Physics, Section B}, 338(3):602--646, 1990.

\bibitem{Drinfeld1987}
V.~Drinfeld.
\newblock {Quantum groups}.
\newblock In {\em Proceedings of the {I}nternational {C}ongress of
  {M}athematicians, {V}ol. 1, 2 ({B}erkeley, {C}alif., 1986)}, pages 798--820.
  Amer. Math. Soc., Providence, RI, 1987.

\bibitem{Drinfeld1990a}
V.~Drinfeld.
\newblock {Quasi-Hopf algebras}.
\newblock {\em Leningrad Math. J.}, 1(6):1419--1457, 1990.

\bibitem{Drinfeld1991}
V.~Drinfeld.
\newblock {On quasitriangular quasi-{H}opf algebras and on a group that is
  closely connected with {${\rm Gal}(\overline{\bf Q}/{\bf Q})$}}.
\newblock {\em Leningrad Math. J.}, 2(4):829--860, 1991.

\bibitem{Etingof2015}
P.~Etingof, S.~Gelaki, D.~Nikshych, and V.~Ostrik.
\newblock {\em {Tensor categories}}, volume 205.
\newblock American Mathematical Society, Providence, RI, 2015.

\bibitem{Etingof2005}
P.~Etingof and F.~Latour.
\newblock {\em {The dynamical Yang-Baxter equation, representation theory, and
  quantum integrable systems}}.
\newblock Oxford University Press, 2005.

\bibitem{EtingofSchiffmann1999}
P.~Etingof and O.~Schiffmann.
\newblock {Lectures on the dynamical {Y}ang-{B}axter equations}.
\newblock In {\em Quantum groups and {L}ie theory ({D}urham, 1999)}, volume 290
  of {\em London Math. Soc. Lecture Note Ser.}, pages 89--129. Cambridge Univ.
  Press, Cambridge, 2001.

\bibitem{EtingofVarchenkoCMP1998}
P.~Etingof and A.~Varchenko.
\newblock {Geometry and classification of solutions of the classical dynamical
  {Y}ang-{B}axter equation}.
\newblock {\em Comm. Math. Phys.}, 192(1):77--120, 1998.

\bibitem{EtingofVarchenko1998}
P.~Etingof and A.~Varchenko.
\newblock {Solutions of the quantum dynamical {Y}ang-{B}axter equation and
  dynamical quantum groups}.
\newblock {\em Comm. Math. Phys.}, 196(3):591--640, 1998.

\bibitem{EtingofVarchenko1999}
P.~Etingof and A.~Varchenko.
\newblock {Exchange dynamical quantum groups}.
\newblock {\em Comm. Math. Phys.}, 205(1):19--52, 1999.

\bibitem{Etingof1999}
P~Etingof and A~Varchenko.
\newblock {Exchange Dynamical Quantum Groups}.
\newblock {\em Commun. Math. Phys}, 205:19--52, 1999.

\bibitem{Faddeev1998}
L.D. Faddeev.
\newblock {How the algebraic {B}ethe ansatz works for integrable models}.
\newblock In {\em Sym\'{e}tries quantiques ({L}es {H}ouches, 1995)}, pages
  149----219. North-Holland, Amsterdam, 1998.

\bibitem{Faddeev1988}
L.D. Faddeev, N.Yu. Reshetikhin, and L.A. Takhtajan.
\newblock {Quantization of Lie Groups and Lie Algebras}.
\newblock {\em Algebraic Analysis}, pages 129--139, jan 1988.

\bibitem{Felder1994}
G.~Felder.
\newblock {Conformal field theory and integrable systems associated with
  elliptic curves}.
\newblock {\em Proceedings of the International Congress of Mathematicians},
  pages 1247--1255, jul 1994.

\bibitem{FelderICMP1995}
G.~Felder.
\newblock {Elliptic quantum groups}.
\newblock In {\em X{I}th {I}nternational {C}ongress of {M}athematical {P}hysics
  ({P}aris, 1994)}, pages 211--218. Int. Press, Cambridge, MA, 1995.

\bibitem{Felder2020}
G.~Felder and M.~Ren.
\newblock {Quantum Groups for Restricted SOS Models}.
\newblock {\em Symmetry, Integrability and Geometry: Methods and Applications
  (SIGMA)}, 17:26, oct 2020.

\bibitem{FelderTarasovVarchenkoAMS1997}
G.~Felder, V.~Tarasov, and A.~Varchenko.
\newblock {Solutions of the elliptic q{KZB} equations and {B}ethe ansatz. {I}}.
\newblock In {\em Topics in singularity theory}, volume 180 of {\em Amer. Math.
  Soc. Transl. Ser. 2}, pages 45--75. Amer. Math. Soc., Providence, RI, 1997.

\bibitem{FelderTarasovVarchenko1999}
G.~Felder, V.~Tarasov, and A.~Varchenko.
\newblock {Monodromy of solutions of the elliptic quantum
  {K}nizhnik-{Z}amolodchikov-{B}ernard difference equations}.
\newblock {\em Internat. J. Math.}, 10(8):943--975, 1999.

\bibitem{FelderVarchenko1996}
G.~Felder and A.~Varchenko.
\newblock {On representations of the elliptic quantum group
  {$E_{\tau,\eta}({\rm sl}_2)$}}.
\newblock {\em Comm. Math. Phys.}, 181(3):741--761, 1996.

\bibitem{Fendley1989}
P.~Fendley and P.~Ginsparg.
\newblock {Non-critical orbifolds}.
\newblock {\em Nuclear Physics B}, 324(3):549--580, oct 1989.

\bibitem{Gervais1984}
J.L. Gervais and A.~Neveu.
\newblock {Novel triangle relation and absence of tachyons in Liouville string
  field theory}.
\newblock {\em Nuclear Physics B}, 238(1):125--141, may 1984.

\bibitem{Goodman-Harpe-Jones}
V.F.R. {Goodman, F.M. and de la Harpe, P. and Jones}.
\newblock {\em {Coxeter graphs and towers of algebras}}.
\newblock Springer-Verlag, New York, 1989.

\bibitem{Jimbo1999}
M.~Jimbo, S.~Odake, H.~Konno, and J.~Shiraishi.
\newblock {Quasi-Hopf twistors for elliptic quantum groups}.
\newblock {\em Transformation Groups}, 4(4):303--327, 1999.

\bibitem{KasselBook1995}
Christian Kassel.
\newblock {\em {Quantum groups}}, volume 155 of {\em Graduate Texts in
  Mathematics}.
\newblock Springer-Verlag, New York, 1995.

\bibitem{Kawahigashi}
Y~Kawahigashi.
\newblock {on flatness of Oceanous' s connection.pdf}.
\newblock {\em J. Funct. Anal.}, 127:63--107, 1995.

\bibitem{Konno2020}
H.~Konno.
\newblock {\em {Elliptic quantum groups}}, volume~37 of {\em SpringerBriefs in
  Mathematical Physics}.
\newblock Springer Singapore, 2020.

\bibitem{LusztigBook1993}
George Lusztig.
\newblock {\em {Introduction to quantum groups}}, volume 110 of {\em Progress
  in Mathematics}.
\newblock Birkh{\"{a}}user Boston, Inc., Boston, MA, 1993.

\bibitem{Ocneanu1988}
A.~Ocneanu.
\newblock {Quantized groups, string algebras and {G}alois theory for algebras}.
\newblock In {\em Operator algebras and applications, {V}ol. 2}, pages
  119----172. Cambridge Univ. Press, Cambridge, 1988.

\bibitem{Ocneanu1991}
A.~Ocneanu.
\newblock {Quantum Symmetry, Differential Geometry of Finite Graphs and
  Classification of Subfactors}.
\newblock In {\em University of Tokyo, seminary notes}, page~59. Department of
  Mathematics, University of Tokyo, 1991.

\bibitem{Pasquier1987b}
V.~Pasquier.
\newblock {Two-dimensional critical systems labelled by Dynkin diagrams}.
\newblock {\em Nuclear Physics B}, 285(C):162--172, jan 1987.

\bibitem{Pasquier1988}
V.~Pasquier.
\newblock {Continuum limit of lattice models built on quantum groups}.
\newblock {\em Nuclear Physics, Section B}, 295(4):491--510, 1988.

\bibitem{Pasquier1988a}
V.~Pasquier.
\newblock {Etiology of IRF models}.
\newblock {\em Communications in Mathematical Physics}, 118(3):355--364, 1988.

\bibitem{Pasquier1990a}
V.~Pasquier and H.~Saleur.
\newblock {Common structures between finite systems and conformal field
  theories through quantum groups}.
\newblock {\em Nuclear Physics B}, 330(2-3):523--556, jan 1990.

\bibitem{Pearce1993}
Paul~A. Pearce and Yu-kui Zhou.
\newblock {Intertwiners and $ADE$ Lattice Models}.
\newblock pages 1--48, 1993.

\bibitem{Ren2023}
M.~Ren.
\newblock {Baxterization for the dynamical Yang-Baxter equation}.
\newblock {\em arXiv:2310.04728 [math.RT]}, oct 2023.

\bibitem{Roche1990}
P.~Roche.
\newblock {Ocneanu cell calculus and integrable lattice models}.
\newblock {\em Communications in Mathematical Physics}, 127(2):395--424, 1990.

\bibitem{Stokman2022}
J.~V. Stokman and N.~Reshetikhin.
\newblock {N-point spherical functions and asymptotic boundary KZB equations}.
\newblock {\em Inventiones Mathematicae}, 229(1):1--86, jul 2022.

\end{thebibliography}
\end{document}